\documentclass[a4paper]{article}

\usepackage[utf8]{inputenc}

\usepackage{amsthm,amsmath,amsfonts,amssymb, bbm}
\usepackage[numbers]{natbib}
\usepackage{xcolor}
\usepackage[colorlinks,allcolors=blue]{hyperref}
\usepackage{graphicx}
\usepackage{mathtools}
\usepackage{dsfont}
\usepackage{etoolbox}
\apptocmd{\thebibliography}{\raggedright}{}{}
\usepackage{listings}
\usepackage{enumerate}
\usepackage{dirtytalk}
\usepackage{todonotes}

\newcommand{\one}{\mathds{1}}
\newcommand{\eps}{\varepsilon}
\newcommand{\E}{\mathbb{E}}
\newcommand{\hE}{\widehat{\mathbb{E}}}
\newcommand{\hPsi}{\widehat{\Psi}}
\newcommand{\hX}{\widehat{X}}
\newcommand{\tiX}{\widetilde{X}}
\newcommand{\Prob}{\mathbb{P}}
\newcommand{\rhobar}{\bar\rho}
\newcommand{\bB}{\bar{B}}
\newcommand{\R}{{\mathbb R}}

\theoremstyle{plain}
\newtheorem{theorem}{Theorem}[section]

\newtheorem{lemma}[theorem]{Lemma}

\newtheorem{atheo}{Theorem}

\newtheorem{assumption}{Assumption}

\theoremstyle{remark}
\newtheorem{definition}[theorem]{Definition}
\newtheorem{example}[theorem]{Example}
\newtheorem{remark}[theorem]{Remark}

\usepackage{bm}
\usepackage{multicol}
\usepackage{wrapfig}
\usepackage{float}
\usepackage[margin=3cm]{geometry}
\usepackage{lipsum}
\usepackage{comment}

\author{
Peter K. Friz
\thanks{Weierstrass Institute for Analysis and Stochastics and TU Berlin, Germany; \href{mailto:friz@math.tu-berlin.de}{friz@math.tu-berlin.de}.}
,
Benjamin Jourdain
\thanks{CERMICS, Ecole des Ponts, IP Paris, INRIA, Marne-la-Vall\'ee, France; \href{mailto:benjamin.jourdain@enpc.fr}{benjamin.jourdain@enpc.fr}}
,
Thomas Wagenhofer
\thanks{TU Berlin, Germany; \href{mailto:wagenhof@math.tu-berlin.de}{wagenhof@math.tu-berlin.de}.
}, Alexandre Zhou}
\date{\today}
\title{On the Weak Error for Local Stochastic Volatility Models}

\begin{document}
\makeatletter
    \def\@maketitle{%
        \newpage
        \null
        \vskip 0.1em%
        \begin{center}%
            \let \footnote \thanks
            {\LARGE \@title \par}%
            \vskip 1.5em%
                {\large
            \lineskip .5em%
            \begin{tabular}[t]{c}%
                \@author
            \end{tabular}\par}%
            \vskip 1em%
                {\large \@date}%
        \end{center}%
        \par
        \vskip 1.5em}
    \makeatother

\maketitle

\begin{abstract}
    Local stochastic volatility refers to a popular model class in applied mathematical finance that allows for ``calibration-on-the-fly'', typically via a particle method, derived from a formal McKean-Vlasov equation. Well-posedness of this limit is a well-known problem in the field;  the general case is largely open, despite recent progress in Markovian situations. 
    
    Our take is to start with a well-defined Euler approximation to the formal McKean-Vlasov equation, followed by a newly established \emph{half-step}-scheme, allowing for good approximations of conditional expectations. In a sense, we do Euler first, particle second in contrast to previous works that start with the particle approximation.
    
    We show weak order one for the Euler discretization, plus error terms that account for the said approximation. The case of particle approximation is discussed in detail and the error rate is given in dependence of all parameters used. 
\end{abstract}
\vspace{0.3cm}


\tableofcontents

\section{Introduction}
Consider a  process $(\xi_t)_{t \in [0,T]}$ and a Brownian motion $(W_t)_{t\in[0,T]}$, both one-dimensional. Motivated by finance applications, we consider dynamics of the form
\begin{gather}\label{Eq:NonMKV}
\begin{aligned}
    dX_t&= \frac{\xi_{t}}{\sqrt{\E\bigl[\xi_{t}^2\big|X_{t} \bigr]}}\sigma(t,X_t)\,dW_t, \qquad t \in [0,T],
    \\
    X_0&=x_0.
\end{aligned}
\end{gather}
Under suitable assumptions on $\xi$ and $\sigma$ this process is a martingale, but in general not a Markov process. This SDE arises naturally in math finance, as particular form of a Bachelier model with randomized volatility, as specified. In case $\xi \equiv  1$ the process $X$ is known as Dupire's local volatility models, see for example \cite{Dupire1994} or \cite{Dupire1997}. In case $\sigma(t,X_t) \equiv \sqrt{\E\bigl[\xi_{t}^2\big|X_{t} \bigr]} $ this models class contains general stochastic volatility models (in Bachelier form).

A priori, it is not clear that this SDE admits a (unique) solution.  However, it can be shown, that the solution of the SDE, if it exists, has the same one-dimensional marginals as the solution of the ``simpler'' SDE 
\begin{align*}
    dY_t&=\sigma(t,Y_t)\,dW_t, \qquad t \in [0,T],
    \\
    Y_0&=y_0,
\end{align*}
a result mostly attributed to Gy\"ongy \cite{Gyongy1986}. 

Preceding Gy\"ongy's findings, Kellerer  \cite{Kellerer1972} showed that for every process that is increasing in convex order (also referred to as \emph{peacock}), there exists a Markovian martingale with the same one-dimensional marginals. Further literature on peacocks can be found in \cite{HirschProfetaEtAl2011}.  A particularly interesting question is the case that $\sigma(t,x)\equiv 1$, which connects local stochastic volatility models with the concept of \emph{fake Brownian motion}. This originates back to a problem posed by Hamza \& Klebaner \cite{HamzaKlebaner2007},  of constructing martingales with Gaussian marginals with variance $t$ that differ from Brownian motion. This question has been answered by many different authors, one of the most recent constructions can be found in \cite{BeiglböckLowther2021}.

In \cite{BrunickShreve2013}, the authors show how to match the one-dimensional distributions of certain path-dependent functionals of an It\^o process with the functional of a Markovian SDE. These functionals include for example the running maximum or running average. More recent results/extensions of Gy\"ongy's Lemma can be found in \cite{BentataCont2012} and the references therein. 

Existence of \eqref{Eq:NonMKV} is a very delicate and difficult matter. Under a suitable initial distribution, the authors show in \cite{AbergelTachet10} local existence of solutions. In \cite{JourdainZhou2020}, the authors show existence in the case that $\xi$ is a pure jump process taking finitely many values. In \cite{LackerShkolnikovZhang2019} the authors show existence, when $\xi$ is a Markovian solution of an SDE driven by an independent Brownian motion and if the joint SDE of $(\xi,X)$ is already started in the stationary distribution. More recent progress includes  \cite{Djete22}, also in a Markovian setting. 

Typically, convergence of Markovian mimicking is proved by PDE-methods, in particular by considering the Kolmogorov forward equation i.e.\ the Fokker-Planck equation. 
Conditional expectations $\E[\xi_t^2|X_t]$ can be characterised by the joint law of $\xi_t$ and $X_t$, thus allowing an interpretation as a McKean-Vlasov SDE. 
Guyon \& Henry-Labord\`ere \cite{GuyonHL2013,GuyonHL12} proposed a particle method for approximating these conditional expectations, together with some regularisation to avoid singularities coming from the denominator in \eqref{Eq:NonMKV}. A weak error rate for this kind of particle approximation has been established by Reisinger \& Tsianni \cite{ReisingerTsianni2023}. 

In this paper, we view the problem of approximating the solution of \eqref{Eq:NonMKV} from a different perspective: the marginals of the target processes are known and the  difficulty lies in the computation of the conditional expectation, especially in the time-continuous setting. Therefore for a stepsize $h$ such that $\frac Th\in \mathbb{N}$ we start directly with the Euler-approximation for $j\in \{0,\dots,\frac Th-1\}$
\begin{align*}
     X_{jh+h}^h&=X_{jh}^h+\frac{\xi_{jh}}{\sqrt{\E\bigl[\xi_{jh}^2\big|X_{jh}^h\bigr]}}\sigma\bigl(jh,X_{jh}^h\bigr) \bigl( W_{jh+h}-W_{jh} \bigr),
     \\
     X_0&=x_0.
\end{align*}
without the use of any regularisation for conditional expectation. Existence for the Euler-scheme now does not cause any issue, and we were still able to show that the marginals of this scheme converge to the marginals of the solution of \eqref{Eq:NonMKV} with \emph{weak rate one}, see Theorem \ref{Thm_weakdiffrate}.

Interestingly enough, for this theorem the correlation structure between $\xi$ and $W$ does not matter. However, knowing such a structure allows for a much deeper analysis. Therefore, we assume for the following that $W=\rho B+\rhobar \bB$ or a given correlation parameter $\rho\in[0,1]$ and its associated conjugate $\rhobar=\sqrt{1-\rho^2}$ and with $B$  and $\bB$ two independent Brownian motions such that $\xi$ is measurable w.r.t.\ the filtration generated by $B$. We consider the SDE
\begin{align*}
    dX_t&= \frac{\xi_{t}}{\sqrt{\E\bigl[\xi_{t}^2\big|X_{t} \bigr]}}\sigma(t,X_t)\,\Bigl(\rho dB_t+\rhobar d\bB_t\Bigr), \qquad t \in [0,T],
    \\
    X_0&=x_0.
\end{align*}

We will work under the following mild
\begin{assumption}\label{Ass:Half_Step}
    $B$ and $W$ are not fully correlated, i.e.\ $\rho^2<1$.
\end{assumption}
The condition on the correlation $\rho$ is according to current calibration standards where negative values close to, but away from $-1$ are chosen.  For technichal reasons we also assume the following.

\begin{assumption}\label{Aspt_new}
    The volatility function $\sigma$ is bounded, four times continuously differentiable with bounded derivatives and uniformly elliptic, i.e.\ there are constants $c$ and  $C$ such that $c\le \sigma(t,x)\le C$ for all $t \in [0,T],x\in \mathbb{R}$.
    
    Furthermore the stochastic volatility process $\xi=(\xi_t)_{t\in[0,T]}$ is uniformly bounded from above and below, i.e.\ there are $0<a<1<b<\infty$ such that $a\le\xi_t\le b$ for all $t \in [0,T]$.
    Let $c_{min}$ be such that $2c_{min}=\frac{ac}{b}$.
    In particular 
    for all $t \in [0,T]$ and any random variable $Z$
    \begin{align*}
        2{c_{min}} \le \frac{\xi_t \sigma(t,.)}{\sqrt{\E\bigl[\xi_t^2 \big|Z\bigr]}}.
    \end{align*}
\end{assumption}
Boundedness of the volatility $\sigma$ is not really essential but avoids additional distraction and technicalities.
Under these assumptions, we consider the \emph{half-step-scheme} introduced in \cite{Zhou2018}. Given yet another Brownian motion $Z$, independent of $B$ and $\bB$  the process $X^h$ has the same one-dimensional marginals as $\tiX^h$ given by
\begin{gather}
\begin{aligned}
     \tiX^h_{jh+h/2}&=\tiX^{h}_{jh}+\frac{\xi_{jh}\sigma\bigl(jh,\tiX_{jh}^h\bigr)}{\sqrt{\E\bigl[\xi_{jh}^2\big|\tiX_{jh}^h\bigr]}}\rho \bigl( B_{jh+h}-B_{jh} \bigr)
     \\
     &\qquad +
     \biggl(\frac{\xi_{jh}^2\sigma^2\bigl(jh,\tiX_{jh}^h\bigr)}{{\E\bigl[\xi_{jh}^2\big|\tiX_{jh}^h\bigr]}}
     -{c_{min}^2}
     \biggr)^{1/2}
     \rhobar 
     \bigl( \bB_{jh+h}-\bB_{jh} \bigr),
     \\
     \tiX^h_{jh+h}&=\tiX_{jh+h/2}^h + {{c_{min}}}\rhobar \bigl( Z_{jh+h}-Z_{jh} \bigr),
    \\
    \tiX^h_0&=x_0.
\end{aligned}
\end{gather}
This scheme has the big advantage that it allows for exact computation of conditional expectation: if we denote by $\varphi_\lambda$ the density of the centered Gaussian distribution with  variance $\lambda=hc_{min}^2\rhobar^2$, then
\begin{align*}
    \E\bigl[\xi_{jh}^2\big|\tiX_{jh}^h=y\bigr] = \frac{\E\bigl[\xi_{jh}^2\varphi_\lambda\bigl(\tiX_{jh-h/2}^h-y\bigr) \bigr] }{\E\bigl[\varphi_\lambda\bigl(\tiX_{jh-h/2}^h-y\bigr) \bigr] }.
\end{align*}
For such a simulation scheme we introduce a particle system  and a regularisation parameter $\delta$ with which we write
\begin{align*}
   \frac{\E\bigl[\xi_{jh}^2\varphi_\lambda\bigl(\tiX_{jh}^h-y\bigr) \bigr] }{\E\bigl[\varphi_\lambda\bigl(\tiX_{jh}^h-y\bigr) \bigr] } \approx 
    \frac{\frac1N\sum_{k=1}^N\bigl(\xi_{jh}^{(i)}\bigr)^2\varphi_{\lambda}\bigl(\tiX_{jh}^{h,(k)}-y\bigr) +\delta }
    {\frac1N\sum_{k=1}^N\varphi_\lambda \bigl(\tiX_{jh}^{h,(k)}-y\bigr)+\delta}. 
\end{align*}
Let $\bigl(\tiX_{T}^{h,(k),N,\delta}\bigr)_{k=1,...,N}$ be the resulting particle approximation, then we can show the following.
{
\begin{atheo}\label{Thm_Particle_Rate_Intro}
    Let $f$, $\sigma$ be sufficiently smooth functions and $\xi$ satisfying Assumptions \ref{Ass:Half_Step} and \ref{Aspt_new}. Let $\delta,h$ be small. Then there is a constant $C>0$ such that
    \begin{align*}
         \Bigl|\E\bigl[ f(Y_T)\bigr]-\E\bigl[f\bigl(\tiX_T^{h,(1),N,\delta}\bigr)\bigr]\Bigr|
        \lesssim 
        h + \delta^{1-}
        + \frac1{\sqrt{N}} \exp\Bigl( Ch^{-1-}\delta^{0-}\Bigr).
    \end{align*}
    Here $\delta^{1-}$ means that this term decays faster than $\delta^{1-\gamma}$ as $\delta \rightarrow 0$ for all $\gamma >0$.
\end{atheo}
\addtocounter{atheo}{-1}
}
 
    The novelty of this theorem is that it quantifies all errors adding up to final particle scheme.
    Moreover the result shows, that any error is achievable and furthermore provides a way of choosing the parameters $h, \delta$ and $N$ such that a difference of that order is obtained.

    Here is a brief outline of this paper: Section \ref{Sec:Euler_Conv} studies the convergence rate without the particle system and works with the true conditional expectation. In this setting we derive a weak rate of order \emph{one}. Also, we show that it is possible to consider non-smooth test functions.
    In Section \ref{Sec_Approximation} we introduce an estimator for conditional expectation and quantify the error with regard to this approximation. We furthermore formally introduce the \emph{half-step} scheme. 

    Section \ref{Sec_PoC} is about particle approximations and establishes the setting and proof of Theorem \ref{Thm_Particle_Rate_Intro}. For this we introduce a particle system and show a propagation of chaos result, following a well-established method, see e.g. \cite[Proposition 4.1]{KohatsuShigeyoshi1997}. 
    In Section \ref{Sec:Numerics} we visualize our findings through some simple examples.

    In Section \ref{Sec_Particles} we consider a non-parametric estimator for conditional expectation, namely a Nadaraya-Watson type estimator which is introduced in Section \ref{Sec:Cond_Exp_Estimators}. To do so we need to introduce a new parameter $\eps$, which gives rise to additional error terms. These error terms are studied further  in Section \ref{Sec:Cond_exp_estim}. Similar to Theorem \ref{Thm_Particle_Rate_Intro}, Theorem \ref{Thm_Particle_Rate2} states the weak error depending on all the parameters used. By not exploiting the Gaussian structure as well as via the half-step scheme we see an additional error term arising in this setting. Here again, we establish propagation of chaos (Section \ref{Sec:POC}). 
    
     The advantage of the half-step scheme is best appreciated by comparing the estimate (in the more precise form given in Theorem \ref{Thm_Particle_Rate} (with half-step) with those of Theorem
    \ref{Thm_Particle_Rate2} (without half-step) where we analyze a seemingly more straightforward algorithm (based on kernel regression). 
    From a technical perspective, the half-step scheme has also the advantage that no ad-hoc choice of $\varepsilon$ (kernel regularisation) is necessary, which simplifies the proof (no tracking of $\varepsilon$) and is of obvious practical appeal. 
    
     Let us now mention cases where Assumption \ref{Ass:Half_Step} does not hold: some authors in the literature argue for path-dependent volatility models, see for example \cite{GuyonLekeufack2023}. For these models, as the name suggests, one assumes that the volatility depends on the whole path of the stock $S$ (or log-stock $X$), therefore in such cases it is unreasonable to assume that $|\rho|\not =1$. See Remark \ref{Remark_Correlation} for a brief discussion of that setting. 

\medskip    

{\bf Acknowledgment:} PKF and TW acknowledge support from DFG CRC/TRR 388 ``Rough Analysis, Stochastic Dynamics and Related Fields", Projects A02 and B04. TW was initially supported by MATH+, as PhD student in the Berlin Mathematical School (BMS), and then by IRTG 2544 ”Stochastic Analysis in Interaction” -  DFG project-ID 410208580. BJ acknowledges support of the “Chaire Risques Financiers”, Fondation du Risque. 
TW would also like to thank Christoph Reisinger for many fruitful discussions. 

\section{Notation}

We denote by $B=(B_t)_{t\ge 0}, \bB=(\bB_t)_{t\ge 0}$ as well as $Z=(Z_t)_{t\ge 0}$ Brownian motions adapted to a filtration $\mathbb{F}=(\mathcal{F}_t)_{t\ge 0}$ satisfying the usual condition.
We assume that $B, \bB$ and $Z$ are mutually independent and denote by $\xi=(\xi_t)_{t\ge 0}$ an $\mathbb{F}$-adapted process. Starting in Section \ref{Sec:Half_Step} the process $\xi$ will be adapted to $\mathcal{F}^B$, the filtration generated by $B$. 

For some $p\ge 1$ the space $L^p$ consists of all $\mathbb{F}$-measurable random variables $X$ such that $\|X\|_{L^p}^p=\E\bigl[|X|^p\bigr]<\infty.$ Given such a random variable we denote 
its law by $\mathcal{L}(X)$. The space $\mathcal{P}_2(\mathbb{R}^n)$ denotes the set of probability measures $\mu$ on $\mathbb{R}^n$ such that $\int_{\mathbb{R}^n}|x|^2\,\mu(dx)<\infty$. For $x \in \mathbb{R}^n$ we denote by $\delta_x$ the Dirac measure at point $x$.

For $\lambda>0$ we denote by $\varphi_\lambda(x)= \frac1{\sqrt{2\pi \lambda}} \exp\bigl(-\frac{x^2}{2\lambda} \bigr)$ the density of the centered Gaussian density with variance $\lambda$ and by $\Phi$ the cumulative distribution function of the standard Gaussian distribution, i.e.\ $\Phi(x)=\int_{-\infty}^x \varphi_1(y)\,dy$. 

For a function $f: \mathbb{R}^m \rightarrow \mathbb{R}^n$ we denote $\|f\|_{L^\infty}=\operatorname*{ess\,sup} |f|$. If $f$ is continuously differentiable we denote the $C_1$ norm by 
$\|f\|_{C_1}=\|f\|_{L^\infty}+\sup\limits_{1\le i \le m}\|\partial_{x_i} f(x)\|_{L^\infty}$ and . For $\alpha\in (0,1)$ we denote by $C_\alpha$ the space of $\alpha$-H\"older continuous functions with (semi-)norm $\|f\|_{C_\alpha}=
    \sup_{\substack{x,y \in \mathbb{R}^m
        \\ x\not =y}}
    \frac{|f(x)-f(y)|}{|x-y|^\alpha}$.

In this work $C$ or $\tilde C$ denotes a positive constant that might change from line to line. Dependence on parameters will be indicated. Also, the notation $A \lesssim B$ means that there is a constant $C$ such that $A \le C B$. This hidden constant may also change from line to line. Furthermore, we say $A=\mathcal{O}(B)$ if $|A|\lesssim |B|$.

\section{Euler Scheme for LSV dynamics }\label{Sec:Euler_Conv}

Recall the following one-dimensional SDE on $[0,T]$
\begin{align*}
    dX_t&= \frac{\xi_{t}}{\sqrt{\E\bigl[\xi_{t}^2\big|X_{t} \bigr]}}\sigma(t,X_t)\,dW_t,
    \\
    X_0&=x_0.
\end{align*}
We now quote a special case of \emph{Gy\"ongy's mimicking theorem}, \cite[Theorem 4.6]{Gyongy1986} for our setting.
    \begin{theorem}\label{Thm_mimicking}
        Let $(\eta_t)_{t \ge 0}$ be a bounded adapted process that is also bounded from below by a positive constant. Let $X_t=x_0+\int_0^t \eta_s\,dW_s$. Then there
        exists a weak (Markov) solution of 
\begin{align*}
dY_t&=\widehat{\eta}\bigl(t, Y_t\bigr) \, dW_t,
\\
Y_0&=x_0,
\end{align*}
where the function $\widehat{\eta}$ is given by
$
\widehat{\eta}^2\bigl(t, X_t \bigr)=\E\bigl[ \eta_t^2 \big| X_t\bigr],
$
and it holds that
$Y_t \stackrel{d}{=} X_t$ for all $t \ge 0$.
    \end{theorem}
 
This allows us to find a ``target process'' $Y$, which follows a Markovian dynamic such that $Y_t\stackrel{d}{=}X_t$ for all $t \in [0,T]$, if a solution $X$ exists. In our case the target process follows the dynamic
 \begin{gather}\label{Eq:target}
     \begin{aligned}
    dY_t&=\sigma(t,Y_t)\,dW_t, \qquad t\in[0,T],
     \\
     Y_0&=x_0.
     \end{aligned}
 \end{gather}
\begin{remark}
    The mimicking theorem relates to the following problem from mathematical finance: given a volatility function $\sigma$ and a process $X_t=X_0+\int_0^t \xi_t \lambda(t,X_t)dW_t$, is there a choice for the function $\lambda$ such that the one-dimensional distributions of $X$ agree with those of \eqref{Eq:target}. The function $\lambda$ is also called leverage function, and the function $\sigma$ can for example be chosen as Dupire's local volatility. 
    Due to Theorem \ref{Thm_mimicking} we can do so by choosing $\lambda(t,x)=\frac{\sigma(t,x)}{\sqrt{\E[\xi^2_t|X_t=x]}}$.
\end{remark}

Take some parameter $h \in \mathbb{R}$, such that $\tfrac Th \in \mathbb{N}$ and consider for $j \in \{0,\dots,\tfrac Th-1\}$ and $s\in [0,h]$ the following Euler-Scheme
\begin{gather}\label{Eq:Euler_Scheme}
    \begin{aligned}
    X^h_{jh+s}&=X^{h}_{jh}+\frac{\xi_{jh}}{\sqrt{\E\bigl[\xi_{jh}^2\big|X_{jh}^h\bigr]}}\sigma\bigl(jh,X_{jh}^h\bigr) \bigl( W_{jh+s}-W_{jh} \bigr),
    \\
    X^h_0&=x_0.
\end{aligned}
\end{gather}

It is important to note that convergence of $\bigl(X^h\bigr)_{h >0}$ in a pathwise sense is highly non-trivial to show. We will show that, for the one-dimensional distribution convergence in law holds, meaning that $\lim\limits_{h \rightarrow 0} \mathcal{L}\bigl(X_T^h\bigr)=\mathcal{L}(Y_T)$. Furthermore, we can quantify the convergence speed for a broad class of test functions, only restricted by their regularity.

For $t \in [0,T]$ and $s \in [t,T]$ let $Y_s^{t,x}$ be the strong solution to \eqref{Eq:target} such that $Y_t^{t,x}=x$ given by Itô's theorem. Consider the function
    \begin{align}\label{Eq:u}
        u(t,x)= \E\Bigl[f\bigl(Y_T^{t,x}\bigr)\Bigr], \qquad (t,x) \in [0,T]\times \mathbb{R}.
    \end{align}
    The Feynman-Kac formula implies that $u$ solves the PDE
    \begin{gather}\label{Eq:FeynmanKac}
        \begin{aligned}
        \partial_t u+\frac12 \sigma^2\partial_{x}^2u&=0, \qquad (t,x)\in [0,T]\times \mathbb{R}
        \\
        u(T,x)&=f(x), \qquad x \in \mathbb{R}.
    \end{aligned}
    \end{gather}

We note that the following theorem also covers (with $l =1$) the case when $f$ is taken as call payoff, that is $f(x) = \max\{(x-K),0\}$.
The idea is to employ regularity theory for the parabolic PDE \eqref{Eq:FeynmanKac},
essentially following the well-established method by Talay and Tubaro \cite{TalayTubaro1990}. 

\begin{theorem}\label{Thm_weakdiffrate}
    Let $X^h, Y$ be defined as above, let Assumption \ref{Aspt_new} hold and assume that $f:\mathbb{R}\rightarrow \mathbb{R}$ is $l$-times weakly differentiable with derivatives which are functions of polynomial growth. Then, for $h \in (0,1)$, 
    \begin{align*}
        \Bigl| \E\bigl[f\bigl(X_T^h\bigr)\bigr]-\E\bigl[f\bigl(Y_T\bigr)\bigr] \Bigr|\lesssim \begin{cases}
            \sqrt{h} & \text{if } l=1, \\
            h\log(1/h) & \text{if } l=2, \\
            h & \text{if } l\ge 3.
        \end{cases}
    \end{align*}
\end{theorem}

Before we are able to prove this theorem, we need to establish some regularity estimates for $u$ and moment conditions for the process $X^h$. \begin{lemma}\label{Lem_UEstim}
    Let $l \in \{1,2,3\}$ and $f$ be $l$-times weakly differentiable with all derivatives which are functions of polynomial growth. Let $u(t,x)=\E\bigl[ f\bigl(Y_T\bigr)\big|Y_t=x\bigr]$ and let $0\le 2i+j\le 4$. Then there is a constant $\tilde C$ and some $N \in \mathbb{N}$ such that for all $(t,x)\in[0,T)\times\R$,
    \begin{align*}
        \partial_t^i\partial_x^j u(t,x) &\le 
            \tilde C(T-t)^{-(2i+j-l)/2} \bigl(1+|x|^N\bigr).
    \end{align*}
\end{lemma}
The proof of this lemma can be found in the \hyperref[Sec:Appendix]{Appendix}. Using this, the following lemma shows H\"older-continuity of a certain order, depending on the payoff function $f$.

\begin{lemma}\label{Lem_UCont}
    Let $h\in (0,1)$ such $T/h\in\mathbb{N}$, $u$ be defined in \eqref{Eq:u} and $X_t^h$ be defined in \eqref{Eq:Euler_Scheme}. For $l \in \{1,2\}$ let $f$ be $l$-times weakly differentiable with derivatives of polynomial growth, then
    \begin{align*}
        \biggl| \E\Bigl[u\bigl(T,X_{T}^h\bigr)\Bigr]-\E\Bigl[u\bigl(T-h,X_{T-h}^h\bigr)\Bigr] \biggr| \lesssim h^{l/2}
    \end{align*}
\end{lemma}
\begin{proof}
    We start with the time increment. Let $x\in \mathbb{R}$ and recall the process $Y$ from \eqref{Eq:target}. By polynomial growth condition for $f'$ and existence of moments for $Y$, see Lemma \ref{lem:Tail_estim} with $\bar \beta\equiv 0$, there is some $N \in \mathbb{N}$, $C>0$ such that 
    \begin{align*}
        \Bigl|u(T,x)-u(T-h,x)\Bigr|&=\Bigl|\E\Bigl[ f(x)-f\bigl(Y_T^{T-h,x})  \Bigr]\Bigr|
        \le 
        C \E\Bigl[ |x-Y_T^{T-h,x}| \bigl(1+|x|+|Y_T^{T-h,x}|\bigr)^N \Bigr]\\&\lesssim
        \sqrt{h}(1+|x|)^N.
    \end{align*}
    Here we used Hölder inequality and the fact that $\|x-Y_T^{T-h,x}\|_{L^p}\lesssim \sqrt{h}$ for all $p\ge 1$.
    If $l=2$ we furthermore have due to polynomial growth of $f''$
    \begin{align*}
        \Bigl|\E\Bigl[ f(x)-f\bigl(Y_T^{T-h,x}\bigr)  \Bigr]\Bigr|
        &\le 
        \Bigl|\E\Bigl[ f'(x)(x-Y_T^{T-h,x})\Bigr]\Bigr|+C\E\Bigl[ |x-Y_T^{T-h,x}|^2 \bigl(1+|x|+|Y_T^{T-h,x}|\bigr)^N \Bigr]\\&
        \lesssim
        h (1+|x|)^N.
    \end{align*}
    For the spatial component, by Lemma \ref{Lem_UEstim} applied with $(i,j)=(0,1)$ and Lemma \ref{lem:Tail_estim},
    \begin{align*}
        \Bigl| \E\Bigl[&u(T-h,X_T^h)-u(T-h,X_{T-h}^h)\Bigr]\Bigr| 
        = 
        \\
        &
        =\Bigl| \E\Bigl[u\Bigl(T-h,X_{T-h}^h+\frac{\xi_{T-h}}{\E[\xi_{T-h}^2|X_{T-h}^h]^{1/2}}\sigma\bigl(T-h,X_{T-h}^H\bigr)(W_T-W_{T-h})\Bigr)-u(T-h,X_{T-h}^h\Bigr]\Bigr| 
        \\
        &\le C
        \E\Bigl[ \Bigl|\frac{\xi_{T-h}}{\E[\xi_{T-h}^2|X_{T-h}^h]^{1/2}}\sigma\bigl(T-h,X_{T-h}^H\bigr)(W_T-W_{T-h}) \Bigr|\Bigl(1+|X_T^h|^N+|X_{T-h}^h|^N\Bigr) \Bigl|\Bigr] 
        \\
        &\lesssim \sqrt{h}.
    \end{align*}
    If $l=2$ we use Taylor expansion like before and deduce from Lemma \ref{Lem_UEstim} applied with $(i,j)=(0,2)$ that $\bigl| \E\bigl[u(T-h,X_T^h)-u(T-h,X_{T-h}^h\bigr]\bigr| \lesssim h$.
\end{proof}

We are now able to prove the main theorem of this section.

\begin{proof}[Proof of Theorem \ref{Thm_weakdiffrate}]

For convenience we define $\alpha_{jh}^h=\frac{\xi_{jh}}{\sqrt{\E[\xi_{jh}^2|X_{jh}^h]}}$ for $j \in \{0,\dots,\frac Th-1\}$. 
     By construction of $u$ we can write the difference as a telescoping sum 
    \begin{align*}
        \E\Bigl[ f(X_T^h)-f(Y_T)\Bigr] = \E\Bigl[ u(T,X_T^h) - u(0,x_0)\Bigr]= \sum_{j=0}^{T/h-1} \E\Bigl[u\bigl(jh+h,X_{jh+h}^h\bigr)\Bigr]-\E\Bigl[u\bigl(jh,X_{jh}^h\bigr)\Bigr]. 
    \end{align*}
    We first consider the case that $j < T/h-1$.
    Recall the process $X^h$ defined by Equation \eqref{Eq:Euler_Scheme}. By applying Ito formula we see that
    \begin{align*}
        u(jh+h,X_{jh+h}^h)-u(jh,X_{jh}^h)&=\int_{jh}^{jh+h} \partial_x u(s,X_{s}^h) \alpha_{jh}^h \sigma(jh,X_{jh}^h)\,dW_s
        \\
        &\quad  +\int_{jh}^{jh+h} \partial_t u(s,X^h_{s})+\frac12 \partial_{x}^2 u(s,X_{s}^h)
        \bigl(\alpha_{jh}^h\bigr)^2 \sigma^2(jh,X_{jh}^h)\,ds
    \end{align*}
    Therefore
    \begin{align*}
        \E\Bigl[u\bigl(jh+h,X_{jh+h}^h\bigr)\Bigr]-&\E\Bigl[u\bigl(jh,X_{jh}^h\bigr)\Bigr]
        \\
        &=\int_{jh}^{jh+h} \E\Bigl[\partial_t u(s,X^h_{s})+\frac12 \partial_{x}^2 u(s,X^h_{s})
        \bigl(\alpha_{jh}^h\bigr)^2 \sigma^2(jh,X_{jh}^h)\Bigr]
        \,ds
    \end{align*}
    The next step is to exploit the Feynman-Kac PDE \eqref{Eq:FeynmanKac}, for this we rewrite everything as
    \begin{align}
        \E\Bigl[\partial_t u(s,X^h_{s})+\frac12 \partial_{x}^2 u(s,X_{s}^h)&
        \bigl(\alpha_{jh}^h\bigr)^2 \sigma^2(jh,X_{jh}^h)\Bigr]\nonumber
        \\
        &=\E\Bigl[\partial_t u(jh,X_{jh}^h)+\frac12 \partial_{x}^2 u(jh,X_{jh}^h)
        \bigl(\alpha_{jh}^h\bigr)^2 \sigma^2(jh,X_{jh}^h)\Bigr]\label{Eq:Proof34_1}
        \\
        &+\E\Bigl[\partial_t u(s,X_{s}^h)-\partial_t u(jh,X_{jh}^h)\Bigr]\label{Eq:Proof34_2}
        \\
        &+\E\Bigl[\frac12 \Bigl(\partial_{x}^2 u(s,X_{s}^h)-\partial_{x}^2 u(jh,X_{jh}^h)\Bigr)
        \bigl(\alpha_{jh}^h\bigr)^2 \sigma^2(jh,X_{jh}^h)\Bigr]. \label{Eq:Proof34_3}
    \end{align}
   Note that, by construction, 
    \begin{align*}
        \E\bigl[\bigl(\alpha_{jh}^h\bigr)^2\big|X_{jh}^h\bigr]=\E\Bigl[ \frac{\xi_{jh}^2}{\E[\xi_{jh}^2|X_{jh}^h]}|X_{jh}^h\Bigr]=\frac{\E[\xi_{jh}^2|X_{jh}^h]}{\E[\xi_{jh}^2|X_{jh}^h]}=1.
    \end{align*}
    Therefore, by conditioning and applying the tower property as well as the Feynman-Kac PDE we observe
    \begin{align*}
       \E\Bigl[\partial_t u(jh,X_{jh}^h)+\frac12 \partial_{x}^2 u(jh,X_{jh}^h)&
        \bigl(\alpha_{jh}^h\bigr)^2 \sigma^2(jh,X_{jh}^h)\Bigr]
        \\
        &=
         \E\Bigl[\partial_t u(jh,X_{jh}^h)+\frac12 \partial_{x}^2 u(jh,X_{jh}^h)
         \sigma^2(jh,X_{jh}^h)\Bigr]
         \\&=0.
    \end{align*}
    Applying Itô's formula to \eqref{Eq:Proof34_2} as before and using Lemma \ref{Lem_UEstim} we see that there is some $\tilde C>0$ and $N \in \mathbb{N}$ such that
    \begin{align*}
        \Bigl|\E\Bigl[\partial_t u(s,X_{s}^h)-\partial_t u(jh,X_{jh}^h)\Bigr]\Bigr|
        &=\Bigl|\int_{jh}^s \E\Bigl[\partial_{t}^2 u(r,X_{r}^h)+\frac12 \partial_{x}^2\partial_tu(r,X_r^h)\Bigr]\,dr\Bigr|
        \\
        &\le \int_{jh}^s \E\Bigl[\bigl|\partial_{t}^2 u(r,X_{r}^h)\bigr|+\frac12 \bigl|\partial_{x}^2\partial_tu(r,X_r^h)\bigr|
        \Bigr]\,dr
        \\
        &\le \tilde C \int_{jh}^s \E\Bigl[ (T-r)^{-(4-l)/2}\Bigl(1+\bigl|X_r^h\bigr|^N \Bigr)\Bigr]\,dr.
    \end{align*}
    Moment estimates for $X$, see Lemma  \ref{lem:Tail_estim}, imply that there exists a constant independent of $j$ and $h$ such that for $s \in [jh,jh+h)$
    \begin{align*}
        \Bigl|\E\Bigl[\partial_t u(s,X_{s}^h)-\partial_t u(jh,X_{jh}^h)\Bigr]\Bigr|\le C(s-jh)(T-jh-h)^{-(4-l)/2}.
    \end{align*}

    For \eqref{Eq:Proof34_3}, by applying Ito formula we observe
    \begin{align*}
        \E\Bigl[\Bigl(\partial_{x}^2 u(s,X_{s}^h)-\partial_{x}^2 u(jh&,X_{jh}^h)\Bigr)
        \bigl(\alpha_{jh}^h\bigr)^2 \sigma^2(jh,X_{jh}^h)\Bigr]
        \\
        &=\E\Bigl[
        \biggl(\int_{jh}^s \partial_x^2\partial_t u(r,X_r^h)
        \,dr\biggr)
        \bigl(\alpha_{jh}^h\bigr)^2 \sigma^2(jh,X_{jh}^h)
        \Bigr]
        \\
        &\quad +\E\Bigl[
        \biggl(\int_{jh}^s \partial_x^3 u(r,X_r^h)\alpha_{jh}^h \sigma(jh,X_{jh}^h)\,dW_r\biggr)
        \bigl(\alpha_{jh}^h\bigr)^2 \sigma^2(jh,X_{jh}^h)
        \Bigr]
        \\
        &\quad + \frac 12\E\Bigl[
        \biggl(\int_{jh}^s \partial_x^4 u(r,X_r^h)\bigl(\alpha_{jh}^h\bigr)^2 \sigma^2(jh,X_{jh}^h)\,dr\biggr)
        \bigl(\alpha_{jh}^h\bigr)^2 \sigma^2(jh,X_{jh}^h)
        \Bigr]
    \end{align*}
    By $L^2$ orthogonality for stochastic integrals the second term on the right hand side vanishes. Using Fubini's theorem, we see that
    \begin{align*}
        \Bigl|\E\Bigl[\Bigl(\partial_{x}^2 u(s,X_{s}^h)-\partial_{x}^2 u(jh,X_{jh}^h)\Bigr)&
        \bigl(\alpha_{jh}^h\bigr)^2 \sigma^2(jh,X_{jh}^h)\Bigr]\Bigr|
        \\
        &\le
        \int_{jh}^s \E\Bigl[
         \bigl|\partial_x^2\partial_t u(r,X_r^h)\bigr|
        \bigl(\alpha_{jh}^h\bigr)^2 \sigma^2(jh,X_{jh}^h)
        \Bigr]
        \,dr\\
        &\quad + \frac 12\int_{jh}^s \E\Bigl[
         \bigl|\partial_x^4 u(r,X_r^h)\bigr|
         \bigl(\alpha_{jh}^h\bigr)^4 \sigma^4(jh,X_{jh}^h)
        \Bigr]
        \,dr
    \end{align*}
    By Lemma \ref{Lem_UEstim} with $(i,j)=(1,2)$ and $(i,j)=(0,4)$, respectively, we can find some $\tilde C$ and $N\in \mathbb{N}$ such that
    \begin{align*}
        \Bigl|\E\Bigl[\Bigl(\partial_{x}^2 u(s,X_{s}^h)-\partial_{x}^2 u&(jh,X_{jh}^h)\Bigr)
        \bigl(\alpha_{jh}^h\bigr)^2 \sigma^2(jh,X_{jh}^h)\Bigr]\Bigr|
        \\
        &\le
        \int_{jh}^s \E\Bigl[ (T-r)^{-(4-l)/2}
        \Bigl(1+\bigl|X_r^h\bigr|^N\Bigr)\Bigl(1+ \bigl(\alpha_{jh}^h\bigr)^4 \sigma^4(jh,X_{jh}^h)\Bigr)
        \Bigr]
        \,dr
    \end{align*}
    Boundedness of $\sigma$ and $\alpha$, according to Assumption \ref{Aspt_new} and existence of moments, Lemma \ref{lem:Tail_estim}, therefore again imply that there is a constant $C>0$, independent of $j$ and $h$ such that
    \begin{align*}
        \Bigl|\E\Bigl[\Bigl(\partial_{x}^2 u(s,X_{s}^h)-\partial_{x}^2 u(jh,X_{jh}^h)\Bigr)&
        \bigl(\alpha_{jh}^h\bigr)^2 \sigma^2(jh,X_{jh}^h)\Bigr]\Bigr|
        \le C(s-jh)(T-jh-h)^{-(4-l)/2}.
    \end{align*}
    We now have shown that for $s \in [jh,jh+h)$ there is a constant $C>0$, independent of $j,h$ and $s$ such that
    \begin{align*}
        \Bigl|\E\Bigl[\partial_t u(s,X_{s})+\frac12 \partial_{x}^2 u(s,X_{s}^h)
        \bigl(\alpha_{jh}^h\bigr)^2 \sigma^2(jh,X_{jh}^h)\Bigr]\Bigr|\le 
        C(s-jh)(T-jh-h)^{-(4-l)/2}, 
    \end{align*}
    implying that 
    \begin{align*}
        \Bigl| \E\Bigl[u\bigl(jh+h,X_{jh+h}^h\bigr)\Bigr]-\E\Bigl[u\bigl(jh,X_{jh}^h\bigr)\Bigr]\Bigr|&\le C \int_{jh}^{jh+h} (s-jh)(T-jh-h)^{-(4-l)/2} \,ds
        \\
        &\le \frac C2 h^2 (T-jh-h)^{-(4-l)/2}.
    \end{align*}
    We  use Lemma \ref{Lem_UCont} to see that
    \begin{align*}
        \biggl| \E\Bigl[u\bigl(T,X_{T}^h\bigr)\Bigr]-\E\Bigl[u\bigl(T-h,X_{T-h}^h\bigr)\Bigr] \biggr| \lesssim h^{\min(l,2)/2}.
    \end{align*}
    By using triangle inequality and changing the order of summation we conclude that
    \begin{align*}
        \biggl| \E\bigl[f\bigl(X_T^h\bigr)\bigr]-\E\bigl[f\bigl(Y_T\bigr)\bigr]\biggr| &\le  \sum_{j=0}^{T/h-1} \biggl|\E\Bigl[u\bigl(jh+h,X_{jh+h}^h\bigr)\Bigr]-\E\Bigl[u\bigl(jh,X_{jh}^h\bigr)\Bigr]\biggr|
        \\
        &\le h^{\min(l,2)/2} + \sum_{j=0}^{T/h-2} h^2 (T-jh-h)^{-(4-l)/2} 
        \\
        &\le h^{\min(l,2)/2} + h^{l/2}\sum_{j=1}^{T/h-1}  j^{-(4-l)/2} 
        \\
        &\lesssim \begin{cases}
            \sqrt{h} & \text{if } l=1, \\
            h\log(1/h) & \text{if } l=2, \\
            h & \text{if } l=3.
        \end{cases}
    \end{align*}

\end{proof}

\section{Approximation of Conditional Expectation and Half-Step Scheme}\label{Sec_Approximation}
\subsection{Approximation of Conditional Expectation}\label{Approximation}

When simulating via Euler scheme, a key issue is the implementation of conditional expectation, a classical problem in statistics. 
We will already assume that we have such a method and denote that approximation by $\hE$. Given such an approximation for $T,h$ such that $\frac Th \in \mathbb{N}$ and for $j \in \{0,\dots,\frac Th-1\}$ we consider the modified dynamics
\begin{gather}
\begin{aligned}\label{eq_Xhat}
    \hX^h_{jh+t}&=\hX^{h}_{jh}+\frac{\xi_{jh}}{\sqrt{\hE\bigl[\xi_{jh}^2\big|\hX_{jh}^h\bigr]}}\sigma\bigl(jh,\hX_{jh}\bigr) \bigl( W_{jh+t}-W_{jh} \bigr),
    \\
    \hX^h_0&=x_0.
\end{aligned}
\end{gather}The next theorem quantifies the weak error with respect to the approximation $\hE$. The formulation mimics a well-known approximation result from BSDEs, see for example \cite[Theorem 4.1]{BouchardTouzi04}.

\begin{theorem}\label{Thm:Rate_CondExp}
    Let $h,T$ be such that $\frac Th\in \mathbb{N}$, $\sigma$ and $\xi$ satisfy Assumption \ref{Aspt_new} and $f$ be $3$-times weakly differentiable with derivatives of polynomial growth. Assume that $\left(\hE\bigl[\xi_{jh}|\,\hX_{jh}\bigr]\right)_{0\le j\le T/h-1}$ is uniformly bounded from above and below by positive constants not depending on $h$. Then for all $p>1$
    \begin{align*}
        \Bigl|\E\bigl[ f(Y_T)\bigr]-\E\bigl[f(\hX_T^h)\bigr]\Bigr|\lesssim h+\sup_{j\in\{0,\dots,\frac Th-1\}} \bigl\| {\hE\bigl[\xi_{jh}^2\big|\hX_{jh}^h\bigr]}
        -{\E\bigl[\xi_{jh}^2\big|\hX_{jh}^h\bigr]}\bigr\|_{L^p}.
    \end{align*}
\end{theorem}
\begin{remark}
    In the statement above the process $\hX$ is the condition in both the conditional expectation and the estimator thereof. If $\hX$ has a \lq nice\rq{} structure, this structure could be used for optimizing the error. We will later exploit this via the half-step scheme.
\end{remark}

\begin{proof}[Proof of Theorem \ref{Thm:Rate_CondExp}]
     Let $Y$ be defined as in Equation \eqref{Eq:target} and $u$ as in Equation \eqref{Eq:u}.
    It follows from the proof of Theorem \ref{Thm_weakdiffrate} (with the same function $u$) that
    \begin{align*}
        \E\Bigl[ f(\hX_T^h)-f(Y_T)\Bigr] &= \E\Bigl[ u(T,\hX_T^h) - u(0,x_0)\Bigr]
        \\
        &= \sum_{j=0}^{\frac Th-1} \E\Bigl[ u(jh+h,\hX_{jh+h}^h) - u(jh,\hX_{jh}^h)\Bigr].
    \end{align*}
    We now fix some $j$. Similar to Theorem  \ref{Thm_weakdiffrate} we write
    
    \begin{align}
        \E\Bigl[u\bigl(jh+h,\hX_{jh+h}^h\bigr)\Bigr]-&\E\Bigl[u\bigl(jh,\hX_{jh}^h\bigr)\Bigr]\nonumber
        \\
        &=\int_{jh}^{jh+h} \E\biggl[\partial_t u(s,\hX^h_{s})+\frac12 \partial_{x}^2 u(s,\hX^h_{s})
        \frac{\xi^2_{jh}}{\hE\bigl[\xi_{jh}^2\big|\hX_{jh}^h\bigr]} \sigma^2(jh,\hX_{jh}^h)\biggr]\label{Eq:Int41}
        \,ds.
    \end{align}

    We develop the integrand around time point $jh$ and see that
    \begin{align}
         \E\biggl[\partial_t u(s,\hX^h_{s})+\frac12 \partial_{x}^2 u(s,\hX_{s}^h)&
         \frac{\xi^2_{jh}}{\hE\bigl[\xi_{jh}^2\big|\hX_{jh}^h\bigr]} \sigma^2(jh,\hX_{jh}^h)\biggr]\nonumber
        \\
        &=\E\biggl[ \partial_tu(jh,\hX_{jh}^h)+\frac12\frac{\xi^2_{jh}}{\hE\bigl[\xi_{jh}^2\big|\hX_{jh}^h\bigr]}\sigma^2\bigl(jh,\hX_{jh}^h\bigr)\partial_{x}^2u\bigl(jh,\hX_{jh}^h\bigr)\biggr] 
        \nonumber
        \\
        &+\E\biggl[\partial_t u(s,\hX_{s}^h)-\partial_t u(jh,\hX_{jh}^h)\biggr]\label{Eq:Proof41_2}
        \\
        &+\E\biggl[\frac12 \Bigl(\partial_{x}^2 u(s,\hX_{s}^h)-\partial_{x}^2 u(jh,\hX_{jh}^h)\Bigr)
        \frac{\xi^2_{jh}}{\hE\bigl[\xi_{jh}^2\big|\hX_{jh}^h\bigr]} \sigma^2(jh,\hX_{jh}^h)\biggr]. \label{Eq:Proof41_3}
    \end{align}    
    Similar to the proof of Theorem \ref{Thm_weakdiffrate} we see by conditioning that
    \begin{align*}
        \E\biggl[\frac{\xi^2_{jh}}{\E\bigl[\xi_{jh}^2\big|\hX_{jh}^h\bigr]}\sigma^2\bigl(jh,\hX_{jh}^h\bigr)\partial_{x}^2u\bigl(jh,\hX_{jh}^h\bigr)\biggr]
        = \E\Bigl[\sigma^2\bigl(jh,\hX_{jh}^h\bigr)\partial_{x}^2u\bigl(jh,\hX_{jh}^h\bigr)\Bigr].
    \end{align*}
    The Feyman-Kac formula therefore implies that, again similar to the proof of Theorem \ref{Thm_weakdiffrate},
    \begin{align*}
        \E\Bigl[ \partial_tu(jh,\hX_{jh}^h)\Bigr]=-\E\Bigl[\frac12\frac{\xi^2_{jh}}{\E\bigl[\xi_{jh}^2\big|\hX_{jh}^h\bigr]}\sigma^2\bigl(jh,\hX_{jh}^h\bigr)\partial_{x}^2u\bigl(jh,\hX_{jh}^h\bigr)\Bigr].
    \end{align*}
    Therefore, when considering the term of order $h$ we can calculate
    \begin{align*}
        \E\Bigl[ \partial_tu(jh,\hX_{jh}^h)&+\frac12\frac{\xi^2_{jh}}{\hE\bigl[\xi_{jh}^2\big|\hX_{jh}^h\bigr]}\sigma^2\bigl(jh,\hX_{jh}^h\bigr)\partial_{x}^2u\bigl(jh,\hX_{jh}^h\bigr)\Bigr]
        \\
         &=\frac12\E\Bigl[ \Bigl(\frac{\xi^2_{jh}}{\hE\bigl[\xi_{jh}^2\big|\hX_{jh}^h\bigr]}-\frac{\xi^2_{jh}}{\E\bigl[\xi_{jh}^2\big|\hX_{jh}^h\bigr]}\Bigr)\sigma^2\bigl(jh,\hX_{jh}^h\bigr)\partial_{x}^2u\bigl(jh,\hX_{jh}^h\bigr)\Bigr].
    \end{align*}
    By boundedness of $\frac{\xi_{jh}}{\sqrt{\hE[\xi_{jh}^2|\hX_{jh}^h]}}$ we can apply Lemma \ref{lem:Tail_estim}, which together with Lemma \ref{Lem_UEstim} implies that $\partial_{x}^2u\bigl(jh,\hX_{jh}^h\bigr) \in L^m$ for all $m \ge 1$ and, by boundedness of $\sigma$, the same holds true for $\sigma^2\bigl(jh,\hX_{jh}^h\bigr)\partial_{x}^2u\bigl(jh,\hX_{jh}\bigr)$.
    Using the H\"older-inequality, the boundedness of $\xi$ and the assumptions on $\hE[\xi|\cdot]$  it follows that with hidden constants depending on $p > 1$
    \begin{align*}
        \frac12\Bigl|\E\Bigl[ \Bigl(\frac{\xi^2_{jh}}{\hE\bigl[\xi_{jh}^2\big|\hX_{jh}^h\bigr]}-\frac{\xi^2_{jh}}{\E\bigl[\xi_{jh}^2\big|\hX_{jh}^h\bigr]}\Bigr)\sigma^2\bigl(jh,\hX_{jh}^h\bigr)&\partial_{x}^2u\bigl(jh,\hX_{jh}^h\bigr)\Bigr]\Bigr|
        \\
        &\lesssim \E\Bigl[\Bigl(\frac{1}{\hE\bigl[\xi_{jh}^2\big|\hX_{jh}^h\bigr]}
        -\frac{1}{\E\bigl[\xi_{jh}^2\big|\hX_{jh}^h\bigr]}\Bigr)^p\Bigr]^{1/p}
        \\
        &\lesssim\E\Bigl[\Bigl({\hE\bigl[\xi_{jh}^2\big|\hX_{jh}^h\bigr]}
        -{\E\bigl[\xi_{jh}^2\big|\hX_{jh}^h\bigr]}\Bigr)^p\Bigr]^{1/p}.
    \end{align*}
    For the last inequality we used that $\xi_{jh}$ is bounded from below. 

    Again, by moment estimates due to Lemma \ref{lem:Tail_estim}, we observe that Lemma \ref{Lem_UCont} holds with $X^h_t$ replaced by $\hX^h_t$ and $\E\bigl[\xi_t|X_t^h\bigr]$  replaced by $\hE\bigl[\xi_t|\hX_t^h\bigr]$. By a similar calculation as for \eqref{Eq:Proof34_1} and \eqref{Eq:Proof34_2}  in the proof of Theorem \ref{Thm_weakdiffrate}, using the fact that we do not need to apply Feynman-Kac formula to study these terms, we obtain
    \begin{align*}
        \bigl|\eqref{Eq:Proof41_2}+\eqref{Eq:Proof41_3}\bigr|\lesssim h.
    \end{align*}
    
    Calculating the integral \eqref{Eq:Int41}, which gives another term of order $h$, we summarize that
    \begin{align*}
        \E\Bigl[ u(jh+h,\hX_{jh+h}^h) - u(jh,\hX_{jh}^h)\Bigr]\lesssim h \E\Bigl[\Bigl({\hE\bigl[\xi_{jh}^2\big|\hX_{jh}^h\bigr]}
        -{\E\bigl[\xi_{jh}^2\big|\hX_{jh}^h\bigr]}\Bigr)^p\Bigr]^{1/p}+ h^2.
    \end{align*}
    Taking the supremum over all $j \in \{0,\dots,\frac Th-1\}$ and summing up over the same set, the claim follows.
\end{proof}

\subsection{Half-Step Euler Scheme and Regularisation }\label{Sec:Half_Step}

In this section we consider the dependence structure between the stochastic volatility term $\xi$ and the driving Brownian motion $W$. We establish a novel \emph{half-step} scheme and simultaneously define a family of estimators for the conditional expectation. It turns out that with this tool we can find a closed formula for the conditional expectation, see Lemma \ref{Lem_ExactCondexp} below. In general we will denote by $\tiX$ processes defined via half-step schemes and by $\hX$ processes given by \say{classical} Euler schemes.

To establish the half-step scheme we introduce correlation: we denote by $\rho\in(-1,1)$ the correlation between the Brownian motion $B$, which generates $\xi$ and the Brownian motion $W=\rho B+\rhobar \bB$, that drives the SDE\ eqref{Eq:NonMKV}. Here we denote by $\bB$ a Brownian motion independent of $B$ and define $\rhobar=\sqrt{1-\rho^2}$. We emphasize again, that the fact that $\rho\in(-1,1)$, i.e.\ Assumption \ref{Ass:Half_Step} is crucial for establishing this scheme.

We assume that the process $\xi$ is measurable with respect to the filtration generated by $B$. In particular $\xi$ is independent of $\bB$. The SDE \eqref{Eq:NonMKV} writes
\begin{align*}
    dX_t&= \frac{\xi_{t}}{\sqrt{\E\bigl[\xi_{t}^2\big|X_{t} \bigr]}}\sigma(t,X_t)\,\Bigl(\rho dB_t+\rhobar d\bB_t\Bigr), \qquad t \in [0,T]\\
    X_0&=x_0.
\end{align*}
Recall furthermore that we  work under Assumption \ref{Aspt_new}. Here the existence of the constant $c_{min}$ allows us to rewrite the Euler scheme of this equation in such a way, that in each step a (scaled) increment of an independent Brownian motion is added, which basically translates to the density being convolved with a Gaussian density. This facilitates studying conditional expectation. 
\pagebreak

By independence of $\xi$ and $\bB$ we can introduce yet another Brownian motion $Z$. 
The half-step scheme as in \cite{Zhou2018} and the associated  family of estimators $\hE_{\delta}$ can be defined recursively for $h$ such that $\frac Th\in\mathbb{N}$ and $j \in \{0,\dots,\frac Th-1\}$ via 
\begin{gather}
    \label{Eq:Half_step}
\begin{aligned}
     \tiX^{h,\delta}_{jh+h/2}&=\tiX^{h,\delta}_{jh}+\frac{\xi_{jh}\sigma\bigl(jh,\tiX_{jh}^{h,\delta}\bigr)}{\sqrt{\hE_\delta\bigl[\xi_{jh}^2\big|\tiX_{jh}^{h,\delta}\bigr]}}\rho \bigl( B_{jh+h}-B_{jh} \bigr)
     \\
     &\qquad+
     \biggl(\frac{\xi_{jh}^2\sigma^2\bigl(jh,\tiX_{jh}^{h,\delta}\bigr)}{{\hE_\delta \bigl[\xi_{jh}^2\big|\tiX_{jh}^{h,\delta}\bigr]}}
     -{c_{min}^2}
     \biggr)^{1/2}
     \rhobar 
     \bigl( \bB_{jh+h}-\bB_{jh} \bigr),
     \\
     \tiX^{h,\delta}_{jh+h}&=\tiX_{jh+h/2}^{h,\delta} + {{c_{min}}}\rhobar \bigl( Z_{jh+h}-Z_{jh} \bigr),
    \\
    \tiX^h_0&=x_0,
\end{aligned}
\end{gather}
where for $\lambda=c_{min}^2\rhobar^2h$
\begin{align}\label{eq_hatEdelta}
        \hE_\delta\bigl[ \xi_{jh}^2|\tiX^{h,\delta}_{jh}=x\bigr]=\begin{cases}
            \E[\xi^2_{0}] & \text{if } j=0,
            \\
            \frac{\E\bigl[\xi_{jh}^2\varphi_{\lambda}\bigl(x-\tiX_{jh-h/2}^{h,\delta}
        \bigr)\bigr]+\delta}{\E\bigl[
        \varphi_{\lambda}
        \bigl(x-\tiX_{jh-h/2}^{h,\delta}
        \bigr)
        \bigr]+\delta} & \text{if } j\ge1.
        \end{cases}
\end{align}
Due to the recursive structure all objects are well-defined. Also, by a direct calculation, Assumption \ref{Aspt_new} implies that the terms in the square root are bounded from below by $3c_{min}^2$.

While $\delta = 0$ is perfectly possible here, and leads to some exact expressions, as Lemma \ref{Lem_ExactCondexp} below shows, it will be necessary later on to take $\delta >0$ to achieve propagation of chaos, as in Theorem \ref{Thm:HS_Prop_of_chaos}, with rates depending on $\delta$.

\begin{remark}
    Given a probability measures $\mu$ on $\mathbb{R}^2$ and some $j \in \mathbb{N}$ 
     we define a function $\Psi_j$ via 
    \begin{align}\label{Eq:Psi_Def}
        \Psi_j(x,\mu)=\frac{\int_{\mathbb{R}\times\mathbb{R}}y^2\varphi_{\lambda}(x-z)
        \,d\mu(y,z)+\delta}{\int_{\mathbb{R}\times\mathbb{R}}\varphi_{\lambda}(x-z)
        \,d\mu(y,z)+\delta}.
    \end{align}
    For $j=0$ we set $\Psi_0(x,\mu)=\int_{\mathbb{R}\times \mathbb{R}}y^2 d\mu(y,z)$.
   
    Then for $j\ge 0$ the estimator of the conditional expectation $\hE_\delta$ can be understood as a function $ \hE_\delta\bigl[ \xi_{jh}^2|\tiX^{h,\delta}_{jh}=x\bigr]\equiv \Psi_j\bigl(x,\mathcal{L}({\xi_{jh},\tiX_{jh-h/2}^{h,\delta}})\bigr)$.
    Boundedness of $\xi$, as in Assumption \ref{Aspt_new} implies that $\hE_\delta$ is bounded from above and below.
\end{remark}
For comparison we define the regularised, \say{classical} Euler scheme for some arbitrary $\delta \in (0,1)$
\begin{gather}\label{eq_Xhdelta} 
\begin{aligned}
    \hX_{jh+h}^{h,\delta}&=\hX_{jh}^{h,\delta}+\frac{\xi_{jh}\sigma\bigl(jh,\hX_{jh}^{h,\delta}\bigr)}{\sqrt{\hE_\delta\bigl[\xi_{jh}^2\big|\hX_{jh}^{h,\delta}\bigr]}}\Bigl(\rho\bigl( B_{jh+h}-B_{jh} \bigr)+\rhobar\bigl(\bB_{jh+h}-\bB_{jh}\bigr)\Bigr),
    \\
    \hX_0^{h,\delta}&=x_0.
\end{aligned}
\end{gather}
Here, $\hE_\delta\bigl[\xi_{0}^2\big|\hX_{0}^{h,\delta}\bigr]=\E[\xi_{0}^2]$, while for $j\ge 1$, the estimator of conditional expectation is understood as 
\begin{align*}
    \hE_\delta\bigl[\xi_{jh}^2\big|\hX_{jh}^{h,\delta}\bigr]=\Psi_j\bigl(\hX_{jh}^{h,\delta}, \mathcal{L}(\xi_{jh},\tiX_{jh-h/2}^{h,\delta})\bigr),
\end{align*}
where the law of the processes $\tiX$ from the half-step scheme \eqref{Eq:Half_step} is taken.

\begin{lemma}\label{lem_laws}
    For $h$ such that $\frac Th\in \mathbb{N}$ and $j\in \{0,\dots,\frac Th\}$ let $(\hX_{jh}^{h,\delta},\xi_{jh})_{0\le j\le T/h}$ be defined in \eqref{eq_Xhdelta} and $(\tiX_{jh}^{h,\delta},\xi_{jh})_{0\le j\le T/h}$ be defined in \eqref{Eq:Half_step}. 
    
    Then $\mathcal{L}\bigl((\hX^{h,\delta}_{jh})_{j\le T/h},(B_t)_{t\le T}\bigr)=\mathcal{L}\bigl((\tiX_{jh}^{h,\delta})_{j\le T/h},(B_t)_{t\le T}\bigr)$.
    In particular their joint laws are equal, i.e.\ it holds that $\mathcal{L}\bigl( (\hX_{jh}^{h,\delta},\xi_{jh})_{j\le T/h}\bigr) =\mathcal{L}\bigl((\tiX_{jh}^{h,\delta},\xi_{jh})_{j\le T/h}\bigr)$.
\end{lemma}
\begin{proof}
    We show by induction on $j$ that $\mathcal{L}\bigl((\hX^{h,\delta}_{ih})_{i\le j},(B_t)_{t\le jh}\bigr)=\mathcal{L}\bigl((\tiX_{ih}^{h,\delta})_{i\le j},(B_t)_{t\le jh}\bigr)$ for all $j\in \{0,\dots,\frac Th\}$. 
    For $j=0$ this statement is clear. For the induction step note first that by the induction hypothesis and by independence of $(B_t-B_{jh})_{t\ge jh}$ from ${\cal F}_{jh}$ it already follows for the joint laws that
    $\mathcal{L}\bigl((\hX^{h,\delta}_{ih})_{i\le j},(B_t)_{t\le jh+h}\bigr)=\mathcal{L}\bigl((\tiX_{ih}^{h,\delta})_{i\le j},(B_t)_{t\le jh+h}\bigr)$. We recall that $\xi_{jh}$ is measurable w.r.t.\ $(B_t)_{t\le jh}$. Independence of $B$ and $\bB$ as well as $Z$ imply that $(\bB_{jh+h}-\bB_{jh})$ and $(Z_{jh+h}-Z_{jh})$ are independent of each other and independent of $\bigl((\hX^{h,\delta}_{ih})_{i\le j},(B_t)_{t\le jh+h}\bigr)$ as well as $\bigl((\tiX_{ih}^{h,\delta})_{i\le j},(B_t)_{t\le jh+h}\bigr)$. By calculus for independent Gaussian random variables we deduce that the conditional law of $\tiX_{jh+h}^{h,\delta}$ given $\bigl((\tiX_{ih}^{h,\delta})_{i\le j},(B_t)_{t\le jh+h}\bigr)$ is Gaussian with expectation $\widetilde m$ and variance $\widetilde\Sigma$ given by 
    \begin{align*}
        \widetilde m&=\tiX^{h,\delta}_{jh}+\frac{\xi_{jh}\sigma\bigl(jh,\tiX_{jh}^{h,\delta}\bigr)}{\sqrt{\hE_\delta\bigl[\xi_{jh}^2\big|\tiX_{jh}^{h,\delta}\bigr]}}\rho \bigl( B_{jh+h}-B_{jh} \bigr)
        \\
        \widetilde \Sigma&= \biggl(\frac{\xi_{jh}^2\sigma^2\bigl(jh,\tiX_{jh}^{h,\delta}\bigr)}{{\hE_\delta \bigl[\xi_{jh}^2\big|\tiX_{jh}^{h,\delta}\bigr]}}
     -{c_{min}^2}
     \biggr)\rhobar^2 h
     +c_{min}^2\rhobar^2h
     =\frac{\xi_{jh}^2\sigma^2\bigl(jh,\tiX_{jh}^{h,\delta}\bigr)}{{\hE_\delta \bigl[\xi_{jh}^2\big|\tiX_{jh}^{h,\delta}\bigr]}}\rhobar^2h 
    \end{align*}
    where $\hE_\delta \bigl[\xi_{jh}^2\big|\tiX_{jh}^{h,\delta}\bigr]={\Psi_j\bigl(\tiX_{jh}^{h,\delta}, \mathcal{L}(\xi_{jh},\tiX_{jh-h/2}^{h,\delta})\bigr)}$.
    By similar calculations we again obtain that the  conditional distribution of $\hX_{jh+h}^{h,\delta}$ given $\bigl((\hX^{h,\delta}_{ih})_{i\le j},(B_t)_{t\le jh+h}\bigr)$ is Gaussian with mean $\widehat m$ and variance $\widehat{\Sigma}$ defined like $\widetilde m$ and $\widetilde{\Sigma}$ but with $\hX_{jh}^{h,\delta}$ replacing $\tiX_{jh}^{h,\delta}$. Therefore it must hold that  $\mathcal{L}\bigl((\hX^{h,\delta}_{ih})_{i\le j+1},(B_t)_{t\le jh+h}\bigr)=\mathcal{L}\bigl((\tiX_{ih}^{h,\delta})_{i\le j+1},(B_t)_{t\le jh+h}\bigr)$ and hence the induction is complete.
    
    By measurability of $\xi$ w.r.t.\ the filtration generated by $B$ it then follows that 
    \begin{align*}
        \mathcal{L}\bigl( (\hX_{jh}^{h,\delta},\xi_{jh})_{j\le T/h}\bigr) =\mathcal{L}\bigl((\tiX_{jh}^{h,\delta},\xi_{jh})_{j\le T/h}\bigr).
    \end{align*}
\end{proof}

The following lemma, taken from \cite[Proposition 3.3.1]{Zhou2018} shows the strength of the half-step scheme; it gives us a direct and exact way to write the conditional expectation, similar to the non-parametric estimator in Section \ref{Sec_Particles}.

\begin{lemma}\label{Lem_ExactCondexp}
    Let $\varphi_\lambda(x)=\frac{1}{\sqrt{2\pi\lambda}}e^{-\frac{x^2}{2\lambda}}$ be the density of a Gaussian with mean $0$ and variance 
    \begin{align}\label{Eq_lambda}
        \lambda=c_{min}^2\rhobar^2h.
    \end{align}
    Let Assumptions \ref{Ass:Half_Step} and \ref{Aspt_new} hold (guaranteeing that $\lambda >0$).
    Consider $\tiX^h$ given by \eqref{Eq:Half_step}. Then for $j \in \{1,\dots,\frac Th\}$
    \begin{align*}
        \E\bigl[ \xi_{jh}^2|\tiX^{h,\delta}_{jh}\bigr]=\frac{\E\Bigl[\xi_{jh}^2\varphi_{\lambda}\bigl(x-\tiX_{jh-h/2}^{h,\delta}
        \bigr)\Bigr]}{\E\Bigl[
        \varphi_{\lambda}
        \bigl(x-\tiX_{jh-h/2}^{h,\delta}
        \bigr)
        \Bigr]}\Bigg|_{x=\tiX^{h,\delta}_{jh}}.
    \end{align*}
    In particular if $\delta=0$, then for $\hE_\delta$ from Equation \eqref{eq_hatEdelta}
    \begin{align*}
        \hE_0\bigl[ \xi_{jh}^2|\tiX^{h,0}_{jh}\bigr]=\E\bigl[ \xi_{jh}^2|\tiX^{h,0}_{jh}\bigr].
    \end{align*}
    and in this case 
    $\mathcal{L}\bigl( (\tiX_{jh}^{h,0},\xi_{jh})_{j\le T/h}\bigr) =\mathcal{L}\bigl((X_{jh}^{h},\xi_{jh})_{j\le T/h}\bigr)$, with $X^{h}$
    from Equation \eqref{Eq:Euler_Scheme}. 
\end{lemma}
\begin{proof}
    This follows from Lemma \ref{Lem_Cond_exp_gauss} in the Appendix with $\zeta=\hX^h_{jh}$, $X=\hX^h_{jh-h/2}$, $Y=\xi^2_{jh}$,  and $Z= c_{min} \rhobar(Z_{jh}-Z_{jh-h})$. Note that for these parameters $f_Z(x)=\varphi_\lambda(x)$.
By the definition of $\hE_\delta$, Equation \eqref{eq_hatEdelta}, it follows that $(X^{h}_{jh})_{j\le T/h}=(\hX^{h,0}_{jh})_{j\le T/h}$. With Lemma \ref{lem_laws}, we conclude that $ \mathcal{L}\bigl( (\tiX_{jh}^{h,0},\xi_{jh})_{j\le T/h}\bigr) =\mathcal{L}\bigl((\hX_{jh}^{h,0},\xi_{jh})_{j\le T/h}\bigr)=\mathcal{L}\bigl((X_{jh}^{h},\xi_{jh})_{j\le T/h}\bigr)$.
\end{proof}

\subsection{Regularisation Error Estimates}

Although the exact formula for conditional expectation is immensely helpful, working with $\hE_\delta$ as an estimator for conditional expectation has a big advantage in providing a Lipschitz-property, both in $x$ and the joint law of $(\xi,X)$. 

In this section we show that the contribution to the weak error from this additional $\delta$-parameter can be well-quantified. 

    \begin{theorem}\label{thm_HScondExp_err} Let $h,T$ be such that $\frac Th\in\mathbb{N}$ and $f:\R\to\R$ be three times weakly differentiable with derivatives of polynomial growth. Under Assumptions \ref{Ass:Half_Step} and \ref{Aspt_new}, for $\tiX^{h,\delta}$ and $\hX^{h,\delta}$ respectively defined in \eqref{Eq:Half_step} and \eqref{eq_Xhdelta} it holds for any $p>1$ and $\delta \in (0,1/2)$ that
        \begin{align*}
            \Bigl|\E\bigl[ f(Y_T)\bigr]-\E\bigl[f\bigl(\tiX_T^{h,\delta}\bigr)\bigr]\Bigr|=\Bigl|\E\bigl[ f(Y_T)\bigr]-\E\bigl[f\bigl(\hX_T^{h,\delta}\bigr)\bigr]\Bigr|\lesssim h + \delta^{\frac1p}
        \end{align*}
        Here the hidden constant also depends on $p$.
    \end{theorem}

The proof of this theorem relies on the following estimation of the  $L^p$-error arising from the estimation of the conditional expectation. We postpone the proof of that lemma after that of the theorem.

\begin{remark}
    It follows from the proof of Theorem \ref{thm_HScondExp_err}, and therefore also the proof of Lemma \ref{Lem_regularisation} that the hidden constants do not depend on $\lambda$, i.e.\ the constants from Assumptions \ref{Ass:Half_Step} and \ref{Aspt_new}. These assumptions are primarily needed to guarantee specific properties of the density of the second half-step as the calculation in the proof of Lemma \ref{Lem_regularisation} shows.
\end{remark}

\begin{lemma}\label{Lem_regularisation} Recall $\hE_\delta$ from Equation \eqref{eq_hatEdelta}. Under Assumptions \ref{Ass:Half_Step} and \ref{Aspt_new}, for any $M>0$ and $p\ge 1$ it holds that 
   for $j\in\{1,\dots,\frac Th\}$, \begin{align*}
         \Bigl\|\hE_\delta\bigl[ \xi_{jh}^2|\tiX^{h,\delta}_{jh}\bigr]-\E\bigl[ \xi_{jh}^2|\tiX^{h,\delta}_{jh}\bigr] \Bigr\|_{L^p} \lesssim 
        \Bigl(\delta \sqrt{-\log(\delta)}\Bigr)^{\frac1p}.
    \end{align*}
\end{lemma}
\begin{proof}[Proof of Theorem \ref{thm_HScondExp_err}]
        We first choose some $q$ such that $1 < q<p$. The equality $\hE_\delta\bigl[\xi_{0}^2\big|\hX_{0}^{h,\delta}\bigr]=\E[\xi_{0}^2]=\E\bigl[\xi_{0}^2\big|\hX_{0}^{h,\delta}\bigr]$ and 
        Theorem \ref{Thm:Rate_CondExp} tell us that 
        \begin{align*}
        \Bigl|\E\bigl[ f(Y_T)\bigr]-\E\bigl[f(\hX_T^{h,\delta})\bigr]\Bigr|\lesssim h+\sup_{j\in\{1,\dots,\frac Th-1\}} \bigl\| {\hE_\delta\bigl[\xi_{jh}^2\big|\hX_{jh}^{h,\delta}\bigr]}
        -{\E\bigl[\xi_{jh}^2\big|\hX_{jh}^{h,\delta}\bigr]}\bigr\|_{L^q}
    \end{align*}
    for all $j \in \{1,\dots,\frac Th-1\}$, since $\hE_\delta\bigl[\xi_{jh}^2\big|\tiX_{jh}^{h,\delta}\bigr]=\Psi_j\bigl(\tiX_{jh}^{h,\delta}, \mathcal{L}(\xi_{jh},\tiX_{jh-h/2}^{h,\delta})\bigr)$, $\hE_\delta\bigl[\xi_{jh}^2\big|\hX_{jh}^{h,\delta}\bigr]=\Psi_j\bigl(\hX_{jh}^{h,\delta}, \mathcal{L}(\xi_{jh},\tiX_{jh-h/2}^{h,\delta})\bigr)$ and, by Lemma \ref{lem_laws}, $\mathcal{L}(\hX_{jh}^{h,\delta},\xi_{jh}) =\mathcal{L}(\tiX_{jh}^{h,\delta},\xi_{jh})$, we have
$$\bigl\| {\hE_\delta\bigl[\xi_{jh}^2\big|\hX_{jh}^{h,\delta}\bigr]}
        -{\E\bigl[\xi_{jh}^2\big|\hX_{jh}^{h,\delta}\bigr]}\bigr\|_{L^q}
        =\bigl\| {\hE_\delta\bigl[\xi_{jh}^2\big|\tiX_{jh}^{h,\delta}\bigr]}
         -{\E\bigl[\xi_{jh}^2\big|\tiX_{jh}^{h,\delta}\bigr]}\bigr\|_{L^q}.$$
    By Lemma \ref{Lem_regularisation}, we deduce that uniformly for $j \in \{1,\dots,\frac Th-1\}$
        \begin{align*}
            \bigl\| {\hE_\delta\bigl[\xi_{jh}^2\big|\hX_{jh}^{h,\delta}\bigr]}
        -{\E\bigl[\xi_{jh}^2\big|\hX_{jh}^{h,\delta}\bigr]}\bigr\|_{L^q}
        &\lesssim \Bigl(\delta \sqrt{-\log(\delta)}\Bigr)^{\frac1q}.
        \end{align*}
        We conclude by noting that $\bigl(\delta \sqrt{-\log(\delta)}\bigr)^{\frac1q}\lesssim \delta^{\frac1p}$ and that,  by Lemma \ref{lem_laws} again, $\E\bigl[f\bigl(\tiX_T^{h,\delta}\bigr)\bigr]=\E\bigl[f\bigl(\hX_T^{h,\delta}\bigr)\bigr]$.
      \end{proof}

      \begin{proof}[Proof of Lemma \ref{Lem_regularisation}]
    We start by calculating the error in $L^1$. Let $\lambda=c_{min}^2\rhobar^2h$. First we note that by Lemma \ref{Lem_ExactCondexp} and the boundedness of $\xi$
    \begin{align*}
        \Bigl|\hE_\delta\bigl[ \xi_{jh}^2|\tiX^{h,\delta}_{jh}=x\bigr]-\E\bigl[ \xi_{jh}^2|\tiX^{h,\delta}_{jh}=x\bigr]
        \Bigr|
        &=\Biggl|\frac{\delta \E\Bigl[ \varphi_{\lambda}\bigl(x-\tiX_{jh-h/2}^{h,\delta}
        \bigr)\Bigr]-\delta\E\Bigl[\xi_{jh}^2\varphi_{\lambda}\bigl(x-\tiX_{jh-h/2}^{h,\delta}
        \bigr)\Bigr]}
        {\bigl(\E\Bigl[\varphi_{\lambda}\bigl(x-\tiX_{jh-h/2}^{h,\delta}
        \bigr)\Bigr]+\delta\bigr)\E\Bigl[ \varphi_{\lambda}\bigl(x-\tiX_{jh-h/2}^{h,\delta}
        \bigr)\Bigr]}\Biggr|
        \\
        &\le (1+\|\xi\|_{L^\infty}^2) \frac{\delta}{\E\Bigl[ \varphi_{\lambda}\bigl(x-\tiX_{jh-h/2}^{h,\delta}
        \bigr)\Bigr]+\delta}.
    \end{align*}
    Let $f$ denote the density of $\tiX_{jh-h/2}^{h,\delta}$, which exists due to Lemma \ref{Lem:Density}. By the fact that $\varphi_1(x)\gtrsim 1$ for $x \in (-1/2,1/2)$ we see that
    \begin{align*}
        \E\Bigl[ \varphi_{\lambda}\bigl(x-\tiX_{jh-h/2}^{h,\delta}
        \bigr)\Bigr]&= \int_{\mathbb{R}} \varphi_{\lambda}({x-y}) f(y)\,dy
        \\
        &=\int_{\mathbb{R}} \varphi_{1}({y}) f(x -\lambda \sqrt{y})\,dy
        \\
        &\gtrsim \int_{-\frac12}^{\frac12} f(x -\sqrt{\lambda} y)\,dy
        = \frac1{\sqrt{\lambda}}\int_{x-\frac{\sqrt{\lambda}}2}^{x + \frac
        {\sqrt{\lambda}}2} f(y)\,dy.
    \end{align*}

    Let 
    $A_k=\bigl[k\sqrt{\lambda}/2,k\sqrt{\lambda}/2+\sqrt{\lambda}/2\bigr)$ and set $k(x)=n\in \mathbb{N}$ if $x\in \bigl[n\sqrt{\lambda}/2,n\sqrt{\lambda}/2+\sqrt{\lambda}/2\bigr)$. Note that $k(x)=n$ for $x \in A_n$  and also by construction $x \in A_{k(x)}$.
    Then it follows that
    \begin{align*}
        \E\bigl[\varphi_{\lambda}\bigl(x-\tiX_{jh-h/2}^{h,\delta}
        \bigr) \bigr]
        &\gtrsim \frac1{\sqrt{\lambda}} \int_{A_{k(x)}} f(y)\,dy.
    \end{align*}
    In the following we denote by $\lceil x\rceil$ the smallest integer larger than $x \in \mathbb{R}$. We now fix a constant $M>0$.
    By boundedness of $\xi$ 
    \begin{align*}
        \E\Bigl[\Bigl|\hE_\delta\bigl[ \xi_{jh}^2|&\tiX^{h,\delta}_{jh}\bigr]-\E\bigl[ \xi_{jh}^2|\tiX^{h,\delta}_{jh}\bigr]
        \Bigr|\Bigr] 
        \\
        &\le
        2\|\xi\|_{L^\infty} \Prob(|\tiX_{jh-h/2}^{h,\delta}|>M)+\int_{-M}^M (1+\|\xi\|_{L^\infty})^2  \frac{\delta}{\E\bigl[\varphi_{\lambda}\bigl(x-\tiX_{jh-h/2}^{h,\delta}
        \bigr)\bigr]+\delta} f(x)\,dx
        \\
        &\lesssim  \Prob(|\tiX_{jh-h/2}^{h,\delta}|>M)+\sum_{k=-\bigl\lceil\frac M{\sqrt{\lambda}}\bigr\rceil}^{\bigl\lceil\frac M {\sqrt{\lambda}}\bigr\rceil} \int_{A_k}  \frac{\delta}{\E\bigl[\varphi_{\lambda}\bigl(x-\tiX_{jh-h/2}^{h,\delta}
        \bigr)\bigr]+\delta} f(x)\,dx
        \\
        &\lesssim \Prob(|\tiX_{jh-h/2}^{h,\delta}|>M)+\sum_{k=-\bigl\lceil\frac M{\sqrt{\lambda}}\bigr\rceil}^{\bigl\lceil\frac M{\sqrt{\lambda}}\bigr\rceil} \int_{A_k}  \frac{\delta}{\frac1{\sqrt{\lambda}} \int_{A_{k(x)}} f(y)\,dy } f(x)\,dx
        \\
        &\lesssim \Prob(|\tiX_{jh-h/2}^{h,\delta}|>M)+\delta \frac{M}{\sqrt{\lambda}}\sqrt{\lambda}. 
    \end{align*}
    The fact that for any $Z \in L^\infty$
    \begin{align*}
        \mathbb{E}\bigl[|Z|^p\bigr]\le \|Z\|_{L^\infty}^{p-1} \|Z\|_{L^1}
    \end{align*}
    now implies that
    \begin{align*}
        \Bigl\|\hE_\delta\bigl[ \xi_{jh}^2|\tiX^{h,\delta}_{jh}\bigr]-\E\bigl[ \xi_{jh}^2|\tiX^{h,\delta}_{jh}\bigr] \Bigr\|_{L^p} \lesssim {\bigl(\delta M\bigr)^{\frac1p}} + \Prob\bigl(|\tiX_{jh-h/2}^{h,\delta}|>M\bigr)^{\frac1p}.
    \end{align*}

    Lemma \ref{lem:Tail_estim}, the boundedness of $\sigma$ and Assumption \ref{Aspt_new} imply that
    \begin{align*}
        \Prob\bigl( |\tiX_{jh-h/2}^{h,\delta}|>M\bigr) 
        \lesssim \frac{ (b/a)\| \sigma\|_{L^\infty}\sqrt{T}}{M}    
        \exp\bigl( -\frac{M^2}{4(b/a)^2\|\sigma\|_{L^\infty}^2{T}}\bigr).
    \end{align*}
    Choosing $M=\sqrt{-4T (b/a)\|\sigma\|_{L^\infty}^2\log(\delta)}$ and recalling that $\delta<1/2$ we see that $\Prob\bigl( |\tiX_{jh-h/2}^{h,\delta}|>M\bigr) \lesssim \delta$ and by the previous part
    \begin{align*}
        \Bigl\| \hE_\delta\bigl[ \xi_{jh}^2|\tiX^{h,\delta}_{jh}\bigr]-\E\bigl[ \xi_{jh}^2|&\tiX^{h,\delta}_{jh}\bigr] \Bigr\|_{L^p} \lesssim \Bigl(\delta \sqrt{-\log(\delta)}\Bigr)^{\frac1p}.
    \end{align*}
\end{proof}

\section{Propagation of Chaos}\label{Sec_PoC}

In this section we fix the stepsize $h$ such that $\frac Th \in \mathbb{N}$, the lower bound $c_{min}$ from Assumption \ref{Aspt_new} as well as the regularisation parameter $\delta\in (0,1/2)$. We also work in the non-fully correlated setting, i.e.\ under Assumption \ref{Ass:Half_Step}. The main goal of this section is to introduce an interacting particle system with $N$ particles and to study the limit $N \rightarrow \infty$. For $i \in \{1,\dots,N\}$ the $i$-th particle will be denoted by the superscript $(i)$. 
Let $\bigl( B^{(i)},W^{(i)},Z^{(i)}\bigr)_{i=1}^N$ be independent Brownian motions (independent of each other and in $(i)$). Let $\Delta B_{jh}^{(i)}=B_{jh+h}^{(i)}-B_{jh}^{(i)}$ and similarly $\Delta W^{(i)}$ as well as $\Delta Z^{(i)}$. Recall that $\lambda=c_{min}^2\rhobar^2h$ and assume that $\frac Th\in \mathbb{N}$. 

For $j\ge 1$, let $\mu_{jh}=\mathcal{L}\bigl({\xi_{jh},\tiX_{jh-h/2}^{h,\delta}}\bigr)$ and set $\mu_0=\mathcal{L}(\xi_0,x_0)$. Recall the notation  
$\Psi_j(x,\mu_{jh})=\hE_\delta\bigl[ \xi_{jh}^2|\tiX^{h,\delta}_{jh}=x\bigr]$ from \eqref{Eq:Psi_Def}, with the special case that $\Psi_0(x,\mu_0)=\E[\xi_0^2]$.  We consider independent copies of the half-step scheme for $j\in \{0,\dots,\frac Th-1\}$
\begin{align*}
     \tiX^{h,(i),\delta}_{jh+h/2}&=\tiX^{h,(i),\delta}_{jh}+\sigma(jh,\tiX^{h,(i),\delta}_{jh}) \xi_{jh}^{(i)}\Psi_j^{-\frac12}\bigl(\tiX^{h,(i),\delta}_{jh},\mu_{jh}\bigr)\rho \Delta B_{jh}^{(i)}
     \\
     &\qquad+
     \biggl(\sigma^2(jh,\tiX^{h,(i),\delta}_{jh})\bigl(\xi_{jh}^{(i)}\bigr)^2\Psi_j^{-1}\bigl(\tiX^{h,(i),\delta}_{jh},\mu_{jh}\bigr)-c_{min}^2\biggr)^{1/2}
     \rhobar 
     \Delta W_{jh}^{(i)},
     \\
     \tiX^{h,(i),\delta}_{jh+h}&=\tiX_{jh+h/2}^{h,(i),\delta} + c_{min}\rhobar \Delta Z_{jh}^{(i)},
     \\
     \tiX^{h,(i),\delta}_{0}&=x_0.
\end{align*}
Furthermore we introduce the interacting particle scheme
\begin{gather}\label{Eq:Particle_System}
\begin{aligned}
   \tiX^{h,(i),N,\delta}_{jh+h/2}&=\tiX^{h,(i),N,\delta}_{jh}+\sigma(jh,\tiX^{h,(i),N,\delta}_{jh})\xi_{jh}^{(i)}\Psi_j^{-\frac12} \bigl(\tiX^{h,(i),N,\delta}_{jh},\mu_{jh}^N\bigr)\rho \Delta B_{jh}^{(i)}
     \\
     &\qquad+
     \biggl(\sigma^2(t,\tiX^{h,(i),N,\delta}_{jh})\bigl(\xi_{jh}^{(i)}\bigr)^2\Psi_j^{-1} \bigl(\tiX^{h,(i),N,\delta}_{jh},\mu_{jh}^N\bigr)-c_{min}^2\biggr)^{1/2}
     \rhobar 
     \Delta W_{jh}^{(i)},
     \\
     \tiX^{h,(i),N,\delta}_{jh+h}&=\tiX_{jh+h/2}^{h,(i),N,\delta} + {c_{min}}\rhobar \Delta Z_{jh}^{(i)},
     \\
      \tiX^{h,(i),N,\delta}_{0}&=x_0.
\end{aligned}
\end{gather}
Here we denote for $j\ge1$ by $\mu_{jh}^N$ the empirical measure $\mu_{jh}^N(dx,dy)=\frac1N \sum_{i=1}^N \delta_{\xi_{jh}^{(i)},\tiX_{jh-h/2}^{h,(i),N,\delta}}(dx,dy)$ and for consistency $\mu_0^N=\mu_0$ such that $\Psi_0(x_0,\mu_0^N)=\E[\xi_0^2]$.
By 
the recursive structure the measure $\mu_{jh}^N$ is well-defined. 
Recall that by Lemma \ref{lem_laws} for $j\ge 0$ and any $i \in \{1,\dots,N\}$ it holds that $\tiX_{jh}^{h,(i),\delta}\stackrel{d}{=}\hX_{jh}^{h,\delta}$ with the latter defined in Equation \eqref{eq_Xhdelta}. 
\begin{remark}\label{Rem:Psi}
    Recall that for $j\ge 1$ by definition $\Psi_j$ depends on the parameters $\lambda$ and $\delta$ via
    \begin{align*}
        \Psi_j^{-1}\bigl(y,\mu^N_{jh} \bigr)&=\frac
        {
        \frac1N \sum_{k=1}^N   \varphi_\lambda\bigl( y-\tiX_{jh-h/2}^{h,(k),N,\delta} \bigr)+\delta
        }
        { {\frac1N \sum_{k=1}^N \bigl(\xi_{jh}^{(k)}\bigr)^2 \varphi_\lambda\bigl( y-\tiX_{jh-h/2}^{h,(k),N,\delta}\bigr)}+\delta},
        \\
        \Psi_j^{-1}\bigl(y,\mu_{jh}\bigr)&= \frac{
        \E\bigl[ \varphi_\lambda\bigl(y-\tiX_{jh-h/2}^{h,(1),\delta}\bigr)\bigr]+\delta
        }{{\E\bigl[ \bigl(\xi_{jh}^{(1)}\bigr)^2\varphi_\lambda\bigl(y-\tiX_{jh-h/2}^{h,(1),\delta}\bigr)\bigr]}+\delta }.
    \end{align*}
    For $j=0$ we simply have 
    \begin{align*}
        \Psi_0^{-1}\bigl(y,\mu_{0}\bigr)=\Psi_0^{-1}\bigl(y,\mu_{0}^N\bigr)=\E\bigl[\xi_0^2\bigr].
    \end{align*}

    A direct application of Assumption \ref{Aspt_new} shows that $\Psi_j^{-1}\le a^{-1}$, which shows that all terms in the sqare roots are bounded from below by $3 c_{min}^2$.
\end{remark}

\begin{theorem}\label{Thm:HS_Prop_of_chaos}
    Let Assumptions \ref{Ass:Half_Step} and \ref{Aspt_new} hold. Let $\delta \in (0,1/2)$ and let $T,h$ be such that $\frac Th\in \mathbb{N}$. Then there are constants $C,c>0$ such that
    \begin{align*}
        \Bigl\| \tiX_T^{h,(1),\delta} - \tiX_T^{h,(1),N,\delta} \Bigr\|_{L^2}^2
        \le 
         \frac{ c}{N} h^{1/2} \delta^{-1} 
        \exp\Bigl(T
        h^{-1} \log\bigl( C  \delta^{-1}h^{-1/2}\bigr)\Bigr)
    \end{align*}
    \end{theorem}
\begin{remark}\label{Rmk:HS_Prop_of_chaos}
    Recall that $\lambda=c_{min}^2 \rhobar^2 h$. 
    The proof below shows that $c,C$ do not depend on $h,\delta$ or $N$. Furthermore $c,C\rightarrow \infty$ as $c_{min},\rhobar \rightarrow 0$. 
\end{remark}

Before we provide the proof of Theorem \ref{Thm_Particle_Rate_Intro} we will give its precise statement
\begin{theorem}\label{Thm_Particle_Rate}
    Let $f$ be three times weakly differentiable with derivatives of polynomial growth. Let $\sigma$ and $\xi$ satisfy Assumption \ref{Aspt_new}. Let Assumption \ref{Ass:Half_Step} hold and let $\delta,h \in (0,1/2)$ such that $\frac Th \in \mathbb{N}$. Then there is a constant $C>0$ such that
    \begin{align*}
         \Bigl|\E\bigl[ f(Y_T)\bigr]-\E\bigl[f\bigl(\tiX_T^{h,(1),N,\delta}\bigr)\bigr]\Bigr|
        \lesssim 
        h + \delta^{1-}
        + \frac1{\sqrt{N}} \delta^{-1} 
        \exp\Bigl(\frac T2
        h^{-1} \log\bigl( C  \delta^{-1}h^{-1/2}\bigr)\Bigr).
    \end{align*}
\end{theorem}

\begin{proof}[Proof of Theorem \ref{Thm_Particle_Rate}]
    By Theorem \ref{thm_HScondExp_err} we see that for any $p>1$, with hidden constants depending on $p$,
    \begin{align*}
        \Bigl|\E\bigl[ f(Y_T)\bigr]-\E\bigl[f\bigl(\tiX_T^{h,\delta}\bigr)\bigr]\Bigr|\lesssim h + \delta^{1/p}.
    \end{align*}
     Recall that $\tiX_T^{h,\delta}\stackrel{d}{=}\tiX_T^{h,(1),\delta}$.  
    Theorem \ref{Thm:HS_Prop_of_chaos} together with Remark \ref{Rmk:HS_Prop_of_chaos} and the local Lipschitz-property of $f$ together with the polynomial growth of the derivative imply that  
    there is a constant $C>0$ such that
     \begin{align*}
        \Bigl|\E\bigl[f\bigl(\tiX_T^{h,(1),\delta}\bigr)\bigr]-\E\bigl[f\bigl(\tiX_T^{h,(1),N,\delta}\bigr)\bigr]\Bigr|
        &\lesssim \Bigl\|
        \tiX_T^{h,(1),\delta}-\tiX_T^{h,(1),N,\delta}
        \Bigr\|_{L^2}
        \\
        &\lesssim\frac1{\sqrt{N}}
        h^{1/4} \delta^{-1/2}\exp\Bigl(\frac T2
        h^{-1} \log\bigl( C\delta^{-1}h^{-3/2}\bigr)
        \Bigr)
        \\
        &\lesssim\frac1{\sqrt{N}}
         \delta^{-1/2}\exp\Bigl(\frac T2
        h^{-1} \log\bigl( C\delta^{-1}h^{-3/2}\bigr)
        \Bigr).
    \end{align*}
    For the last inequality we used the fact that $h\lesssim 1$.

    \end{proof}
    As $x\mapsto \log(x)$ grows slower than $|x|^\alpha$ with any positive power $\alpha>0$, so that 
    \begin{align*}
        \Bigl|\E\bigl[f\bigl(\tiX_T^{h,(1),\delta}\bigr)\bigr]-\E\bigl[f\bigl(\tiX_T^{h,(1),N,\delta}\bigr)\bigr]\Bigr|
        &\lesssim\frac1{\sqrt{N}} \delta^{-1/2}
        \exp\Bigl( Ch^{-1-} \delta^{0-}\Bigr).
    \end{align*}
    which yields the statement of Theorem \ref{Thm_Particle_Rate_Intro}.

For the proof of propagation of chaos we use Lipschitz properties of the map $\Psi$ established in the following lemma. 
\begin{lemma}\label{rem_G_deriv_improv}
Let $\lambda,\delta \in (0,1)$ and $\mu\in \mathcal{P}_2(\mathbb{R}\times\mathbb{R})$  induced by some random variables $(\xi,X)$ such that $\xi$ is bounded from above and below as in Assumption \ref{Aspt_new}. For $j\ge 0$ consider $\Psi_j$ with its dependence on $\delta$ and $\lambda=c_{min}^2\rhobar^2h$ as defined in \eqref{Eq:Psi_Def}. Then the map $y \mapsto \Psi_j^{-\frac12}(y,\mu)$ is Lipschitz and for all $\gamma \in (0,1)$
\begin{align*}
        \| \Psi_j^{-\frac12}\|_{Lip} \lesssim \gamma^{-\frac12} \lambda^{-\frac{1+\gamma}2} \delta^{-\gamma}.
    \end{align*}
\end{lemma}
\begin{proof}
     Recall that $\xi$ is bounded from below and above. For $j=0$ the map $\Psi_0$ does not depend on the first parameter, hence this statement is clear. Assume now that $j\ge1$.
     
     Chain rule implies that
    \begin{align*}
        \frac{\partial}{\partial y} \Psi_j^{-\frac12}(y,\mu)&=\frac12\sqrt{ \frac{{\E\bigl[ \xi^2
        \varphi_{\lambda}\bigl(y-X\bigr)\bigr]}+\delta }{
        \E\bigl[ \varphi_{\lambda}\bigl(y-X\bigr)\bigr]+\delta
        }
        }
        \Biggl(
        \frac{
        \E\bigl[ \varphi_{\lambda}'\bigl(y-X\bigr)\bigr] 
        \Bigl(
        \E\bigl[ \xi^2\varphi_{\lambda}\bigl(y-X\bigr)\bigr]+\delta
        \Bigr)      }
        {
        \Bigl({{\E\bigl[ \xi^2\varphi_{\lambda}\bigl(y-X\bigr)\bigr]}+\delta }\Bigr)^2}
        \\
        & \qquad \qquad \qquad 
        - 
        \frac{
        \Bigl(\E\bigl[ \varphi_{\lambda}\bigl(y-X\bigr)\bigr] +\delta\Bigr)
        \E\bigl[ \xi^2\varphi_{\lambda}'\bigl(y-X\bigr)\bigr]
        }
        {
        \Bigl({{\E\bigl[ \xi^2\varphi_{\lambda}\bigl(y-X\bigr)\bigr]}+\delta }\Bigr)^2}
        \Biggr)
    \end{align*}
    The term in the square-root is bounded. By boundedness of $\xi$ we only need to consider the expression
    \begin{align*}
        \frac{
        \E\bigl[ \bigl| \varphi_{\lambda}'\bigl(y-X\bigr)\bigr| \bigr] +\delta
       }
        {
        {{\E\bigl[ \varphi_{\lambda}\bigl(y-X\bigr)\bigr]}+\delta }}.
    \end{align*}
    By Lemma \ref{Lem:Gauss_facts}
    \begin{align*}
        \frac{  \E\bigl[ \bigl| \varphi_{\lambda}'\bigl(y-X\bigr)\bigr| \bigr] +\delta
       }
        {
        {{\E\bigl[ \varphi_{\lambda}\bigl(y-X\bigr)\bigr]}+\delta }}
        &=
        \frac{\frac1\lambda\E\Bigl[ \bigl|y-X\bigr| \varphi_{\lambda}\bigl(y-X\bigr) \Bigr] +\delta 
       }
        {
        {\E\bigl[ \varphi_{\lambda}\bigl(y-X\bigr)\bigr]+\delta }}
        \lesssim
        \frac{\frac{\lambda^{\frac{1-\gamma}2}}{\lambda\sqrt{\gamma}}\E\Bigl[ \varphi_{\lambda}\bigl(y-X\bigr)^{1-\gamma} \Bigr] +\delta 
       }
        {
        {\E\bigl[ \varphi_{\lambda}\bigl(y-X\bigr)\bigr]+\delta }}.
    \end{align*}
    Jensen's inequality now implies that
    \begin{align*}
        \frac{  \E\bigl[ \bigl| \varphi_{\lambda}'\bigl(y-X\bigr)\bigr| \bigr] +\delta
       }
        {
        {{\E\bigl[ \varphi_{\lambda}\bigl(y-X\bigr)\bigr]}+\delta }}
        \lesssim 
        \frac{{\lambda}^{-\frac{1+\gamma}2}\gamma^{-\frac12}\E\Bigl[ \varphi_{\lambda}\bigl(y-X\bigr)\Bigr]^{{1-\gamma}} +\delta 
       }
        {
        {\E\bigl[ \varphi_{\lambda}\bigl(y-X\bigr)\bigr]+\delta }} 
        \lesssim 1+ 
        \frac{{\lambda}^{-\frac{1+\gamma}2}\gamma^{-\frac12}\E\Bigl[ \varphi_{\lambda}\bigl(y-X\bigr)\Bigr]^{{1-\gamma}}    }
        {
        {\E\bigl[ \varphi_{\lambda}\bigl(y-X\bigr)\bigr]+\delta }}.
    \end{align*}
    We now consider the function 
    \begin{align*}
        g_\delta:y\mapsto \frac{y^{1-\gamma}}{y+\delta} 
    \end{align*}
    and observe that since $g_\delta(y)=\delta^{-\gamma}g_1(y/\delta)$, $\sup_y |g_\delta(y)|= \delta^{-\gamma} \sup_y |g_1(y)|\le \delta^{-\gamma}$.
    Plugging this in we see that
    \begin{align*}
         \frac{  \E\bigl[ \bigl| \varphi_{\lambda}'\bigl(y-X\bigr)\bigr| \bigr] +\delta
       }
        {
        {{\E\bigl[ \varphi_{\lambda}\bigl(y-X\bigr)\bigr]}+\delta }}
        \lesssim
        1+ \Bigl( \lambda^{\frac{1+\gamma}2}\sqrt{\gamma} \delta^\gamma\Bigr)^{-1}\lesssim \gamma^{-\frac12} \lambda^{-\frac{1+\gamma}2} \delta^{-\gamma}.
    \end{align*}
\end{proof}

We now are able to show propagation of chaos as a last ingredient for the proof of Theorem \ref{Thm_Particle_Rate}.

\begin{proof}[Proof of Theorem \ref{Thm:HS_Prop_of_chaos}]
    For this proof we suppress $h$ and $\delta$ in the superscripts for readability, meaning that we set $\tiX_{jh}^{(i)}=\tiX_{jh}^{h,(i),\delta}$ as well as $\tiX_{jh}^{(i),N}=\tiX_{jh}^{h,(i),N,\delta}$.
    
    For $j=1$ we actually have that $\tiX_{h/2}^{(i),N}=\tiX_{h/2}^{(i)}$ due to the last equation in Remark \ref{Rem:Psi} and the definition of the recursion schemes.
    Therefore the first Euler-step coincides, i.e.\ $\tiX_{h}^{(i),N}=\tiX_{h}^{(i)}$.    
    We now fix some index $i\in \{1,\dots,N\}$, some grid point $j \in \{1,\dots,\frac Th-1\}$ and study $\tiX_{jh+h}^{(i)}-\tiX_{jh+h}^{(i),N}$. Then
    \begin{align*}
        \tiX_{jh+h}^{(i)}-&\tiX_{jh+h}^{(i),N}=\tiX_{jh}^{(i)}-\tiX_{jh}^{(i),N} \\
        &\qquad + \sigma(jh,\tiX_{jh}^{(i)})\xi_{jh}^{(i)} \Bigl(
        \Psi_j^{-\frac12}\bigl(\tiX_{jh}^{(i)},\mu_{jh}\bigr)
        -
        \Psi_j^{-\frac12}\bigl(\tiX_{jh}^{(i)},\mu^N_{jh} \bigr) \Bigr)\rho \Delta B_{jh}^{(i)}
         &\eqqcolon (1)
        \\
        &\qquad + \sigma(jh,\tiX_{jh}^{(i)})\xi_{jh}^{(i)} \Bigl( \Psi_j^{-\frac12}\bigl(\tiX_{jh}^{(i)},\mu^N_{jh} \bigr) - \Psi_j^{-\frac12}\bigl(\tiX_{jh}^{(i),N},\mu^N_{jh}\bigr) \Bigr)\rho \Delta B_{jh}^{(i)} &\eqqcolon (2)
        \\ 
        &\qquad + \Bigl(\sigma(jh,\tiX_{jh}^{(i)})-\sigma(jh,\tiX_{jh}^{(i),N})\Bigr)\xi_{jh}^{(i)}\Psi_j^{-\frac12}\bigl(\tiX_{jh}^{(i),N},\mu^N_{jh}
        \bigr)\rho\,\Delta B_{jh}^{(i)} &\eqqcolon (3)
        \\ 
        &\qquad + \biggl( \biggl(\sigma^2(jh,\tiX^{(i)}_{jh})\bigl(\xi_{jh}^{(i)}\bigr)^2\Psi_j^{-1}\bigl(\tiX^{(i)}_{jh},\mu_{jh}\bigr)-c_{min}^2\biggr)^{1/2}
        \\
        &\qquad \qquad \qquad -\biggl(\sigma^2(jh,\tiX^{(i),N}_{jh})\bigl(\xi_{jh}^{(i)}\bigr)^2\Psi_j^{-1}\bigl(\tiX^{(i),N}_{jh},\mu_{jh}^N\bigr)-c_{min}^2\biggr)^{1/2}
        \biggr)
     \rhobar 
     \Delta W_{jh}^{(i)}. &\eqqcolon (4)
    \end{align*}

    \paragraph{Estimating expression (2) and (3)}
    By Assumption \ref{Aspt_new} and  Lemma \ref{rem_G_deriv_improv}
    we get the estimate
    \begin{gather}\label{Eq:estimTwoThree}
    \begin{aligned}
        \E\Bigl[(2)^2+(3)^2 \Bigr]&\lesssim \E\Bigl[\bigl(1+\|\Psi_j^{-\frac12}\|_{Lip}^2\bigr) \Bigl(\Delta B_{jh}^{(i)} \Bigr)^2 \Bigl( \tiX_{jh}^{(i)}-\tiX_{jh}^{(i),N}\Bigr)^2\Bigr]
        \\
        &\lesssim h {\gamma}^{-1} 
        \delta^{-2\gamma}  \lambda^{-1-\gamma}
        \E\Bigl[ \Bigl( \tiX_{jh}^{(i)}-\tiX_{jh}^{(i),N}\Bigr)^2\Bigr].
    \end{aligned}
    \end{gather}

    \paragraph{Estimating the term (1)} 
    We start by rewriting
    \begin{align*}
        \Psi_j^{-\frac12}&\bigl(\tiX_{jh}^{(i)},\mu_{jh}\bigr)
        -
        \Psi_j^{-\frac12}\bigl(\tiX_{jh}^{(i)},\mu^N_{jh} \bigr) 
        \\
        &=
        \frac{
        \sqrt{
        \E\bigl[ \varphi_\lambda\bigl(y-\tiX_{jh-h/2}^{(1)}\bigr) \bigr]+\delta
        }
        }{
        \sqrt{\E\bigl[ \bigl(\xi_{jh}^{(1)}\bigr)^2\varphi_\lambda\bigl(y-\tiX_{jh-h/2}^{(1)}\bigr) \bigr]+\delta} }\Biggl|_{y=\tiX_{jh}^{(i)}}
        - \frac
        {
        \sqrt{\frac1N \sum_{k=1}^N   \varphi_\lambda\bigl( \tiX_{jh}^{(i)}-\tiX_{jh-h/2}^{(k)} \bigr)+\delta
        }
        }
        { \sqrt{\frac1N \sum_{k=1}^N \bigl(\xi_{jh}^{(k)}\bigr)^2 \varphi_\lambda\bigl( \tiX_{jh}^{(i)}-\tiX_{jh-h/2}^{(k)}\bigr)+\delta}}
        \\
        &
        +
        \frac
        {
        \sqrt{\frac1N \sum_{k=1}^N   \varphi_\lambda\bigl( \tiX_{jh}^{(i)}-\tiX_{jh-h/2}^{(k)} \bigr)+\delta}
        }
        { \sqrt{\frac1N \sum_{k=1}^N \bigl(\xi_{jh}^{(k)}\bigr)^2 \varphi_\lambda\bigl( \tiX_{jh}^{(i)}-\tiX_{jh-h/2}^{(k)}\bigr)+\delta}}
        -
        \frac{
        \sqrt{\frac1N \sum_{k=1}^N   \varphi_\lambda\bigl( \tiX_{jh}^{(i)}-\tiX_{jh-h/2}^{(k),N} \bigr)+\delta}
        }
        { \sqrt{\frac1N \sum_{k=1}^N \bigl(\xi_{jh}^{(k)}\bigr)^2 \varphi_\lambda\bigl( \tiX_{jh}^{(i)}-\tiX_{jh-h/2}^{(k),N}\bigr)+\delta}}
        \\
        &\eqqcolon \frac{\sqrt{\mathcal{A}_0^i+\delta}}{\sqrt{\mathcal{B}_0^i+\delta}}-\frac{\sqrt{\mathcal{A}^i_1+\delta}}{\sqrt{{\mathcal{B}^i_1}+\delta}}+\frac{\sqrt{\mathcal{A}^i_1+\delta}}{\sqrt{\mathcal{B}^i_1+\delta}}-\frac{\sqrt{\mathcal{A}^i_2+\delta}}{\sqrt{\mathcal{B}^i_2+\delta}}.
    \end{align*}
    By the assumptions on $\xi$ we can apply the first part of Lemma \ref{lem_sqt_estim} to the first two terms and the second part of Lemma \ref{lem_sqt_estim} to the last two terms. We therefore observe
    \begin{align*}
        \E\Bigl[ \Bigl( \Psi_j^{-\frac12}\bigl(\tiX_{jh}^{(i)},\mu_{jh}\bigr)
        -
        \Psi_j^{-\frac12}\bigl(\tiX_{jh}^{(i)},\mu^N_{jh} \bigr) \Bigr)^2\Bigr] \lesssim& \frac1{\delta^2}\Bigl(
        \E\bigl[ \bigl(\mathcal{A}_0^i-\mathcal{A}_1^i \bigr)^2\bigr]+\E\bigl[ \bigl(\mathcal{B}_0^i-\mathcal{B}_1^i \bigr)^2\bigr]\Bigr)
        \\
        &+\E\Bigl[ \Bigl(\frac{{\mathcal{A}^i_1+\delta}}{{\mathcal{B}^i_1+\delta}}-\frac{{\mathcal{A}^i_2+\delta}}{{\mathcal{B}^i_2+\delta}} \Bigr)^2\Bigr].
    \end{align*}
    Note that by the i.i.d.\ structure $\E\bigl[\bigl(\xi_{jh}^{(i)}\bigr)^2\varphi_\lambda\bigl(y-\tiX_{jh-h/2}^{(i)}\bigr)  \bigr]= \frac1N\sum_{k=1}^N \E\bigl[\bigl(\xi_{jh}^{(k)}\bigr)^2 \varphi_\lambda\bigl(y-\tiX_{jh-h/2}^{(k)}\bigr)  \bigr]$. Thus  we see, by independence conditional on $\tiX_{jh}^{(i)}$
    \begin{align*}
        &\E\bigl[ \bigl(\mathcal{B}_0^i-\mathcal{B}_1^i \bigr)^2\bigr] 
        \\
        &=\E\biggl[ \Biggl(\E\Bigl[ \frac1N\sum_{k=1}^N \bigl(\xi_{jh}^{(k)}\bigr)^2 \varphi_\lambda\bigl(y-\tiX_{jh-h/2}^{(k)}\bigr)  \Bigr]\bigg|_{y=\tiX_{jh}^{(i)}}-\frac1N\sum_{k=1}^N \bigl(\xi_{jh}^{(k)}\bigr)^2 \varphi_\lambda\bigl(\tiX_{jh}^{(i)}-\tiX_{jh-h/2}^{(k)}\bigr) 
        \Biggr)^2\biggr]
        \\
        &\lesssim
        \frac{\|\varphi_\lambda\|_{L^\infty}^2}{N^2}+ \E\biggl[ \Biggl(\hspace{-1pt}\E\Bigl[ \frac1N\sum_{\substack{k=1 \\ k \not = i}
        }^N \bigl(\xi_{jh}^{(k)}\bigr)^2 \varphi_\lambda\bigl(y-\tiX_{jh-h/2}^{(k)}\bigr)  \Bigr]\bigg|_{y=\tiX_{jh}^{(i)}}\hspace{-8pt}-\frac1N\sum_{\substack{k=1 \\ k \not = i}}^N \bigl(\xi_{jh}^{(k)}\bigr)^2 \varphi_\lambda\bigl(\tiX_{jh}^{(i)}-\tiX_{jh-h/2}^{(k)}\bigr) \hspace{-1pt}
        \Biggr)^2\biggr]
        \\
        &\lesssim \frac {\|\varphi_\lambda\|_{L^\infty}^2} {N^2}+ \frac1N\E\biggl[    \frac1N \sum_{\substack{k=1\\ k \not = i}
        }^N \E\Bigl[  \biggl( \bigl(\xi_{jh}^{(k)}\bigr)^2 \varphi_\lambda\bigl(y-\tiX_{jh-h/2}^{(k)}\bigr) -\E\bigl[\bigl(\xi_{jh}^{(k)}\bigr)^2 \varphi_\lambda\bigl(y-\tiX_{jh-h/2}^{(k)}\bigr) \bigr]\biggr)^2   \Bigr]\bigg|_{y=\tiX_{jh}^{(i)}}      \biggr]
        \\
        &\lesssim \frac1N \|\varphi_\lambda\|_{L^\infty}^2\lesssim \frac1N \lambda^{-1}.
    \end{align*}
    With similar arguments it follows that $\E\bigl[ \bigl(\mathcal{A}_0^i-\mathcal{A}_1^i \bigr)^2\bigr]\lesssim \frac1N \|\varphi_\lambda\|_{L^\infty}^2\lesssim \frac1N \lambda^{-1}$.

    We now consider $\Bigl(\frac{{\mathcal{A}^i_1+\delta}}{{\mathcal{B}^i_1+\delta}}-\frac{{\mathcal{A}^i_2+\delta}}{{\mathcal{B}^i_2+\delta}} \Bigr)^2$.
    For this term we want to use Lemma \ref{lem:estim_AB} for $\gamma_k=(\xi_{jh}^{(k)})^2$, $\alpha_k=\varphi_\lambda\bigl(\tiX_{jh}^{(i)}-\tiX_{jh-h/2}^{(k)}\bigr)$, $\beta_k=\varphi_\lambda\bigl(\tiX_{jh}^{(i)}-\tiX_{jh-h/2}^{(k),N}\bigr)$ and $\Delta_k=\bigl|\tiX_{jh-h/2}^{(k)}-\tiX_{jh-h/2}^{(k),N}\bigr|$.
    We see for any $x,y\in \mathbb{R}$ with $|x|\le |y|$ by Lemma \ref{Lem:Gauss_facts} and by monotonicity for any $\gamma \in (0,1/2)$
    \begin{align*}
        \bigl| \varphi_\lambda(x)-\varphi_\lambda(y)\bigr| &= \bigl| \varphi_\lambda(|x|)-\varphi_\lambda(|y|)\bigr| = \int_{|x|}^{|y|} \varphi_\lambda'(z)\,dz \le \int_{|x|}^{|y|} \frac z\lambda \varphi_\lambda(z)\,dz
        \\
        &\lesssim \gamma^{-1/2} \lambda^{-\frac{1+\gamma}2} \int_{|x|}^{|y|} \varphi_\lambda(z)^{1-\gamma} \,dz
        \lesssim \gamma^{-1/2} \lambda^{-\frac{1+\gamma}2}  \varphi_\lambda(x)^{1-\gamma} \bigl|y-x\bigr|.
    \end{align*}
    Therefore the assumptions of Lemma \ref{lem:estim_AB} are satisfied with $C=\gamma^{-1/2} \lambda^{-\frac{1+\gamma}2}$ and $p=1-\gamma$. {Using $\max\limits_{k=1,\dots,N}\max\{\alpha_{k},\beta_k\}\lesssim \lambda^{-1/2}$, independently of $N$,} this lemma now tells us that
    \begin{align*}
        \Bigl(\frac{{\mathcal{A}^i_1+\delta}}{{\mathcal{B}^i_1+\delta}}-\frac{{\mathcal{A}^i_2+\delta}}{{\mathcal{B}^i_2+\delta}} \Bigr)^2 & \lesssim \gamma^{-1} \lambda^{-1-\gamma} \lambda^{-\frac12(1-\gamma-1/2)\cdot 2} \delta^{-1}\frac1N\sum_{k=1}^N \Bigl(\tiX_{jh-h/2}^{(k),N}-\tiX_{jh-h/2}^{(k)}\Bigr)^2
        \\
        &\lesssim  \gamma^{-1} \delta^{-1} \lambda^{-3/2}
        \frac1N\sum_{k=1}^N \Bigl(\tiX_{jh-h/2}^{(k),N}-\tiX_{jh-h/2}^{(k)}\Bigr)^2
    \end{align*}
    Recall that Lemma \ref{lem:estim_AB} requires that $1-\gamma=p\in (1/2,1)$, therefore the calculation above works for any $\gamma\in (0,1/2)$.

    By definition we have $\tiX_{jh-h/2}^{(k),N}-\tiX_{jh-h/2}^{(k)}=\tiX_{jh}^{(k),N}-\tiX_{jh}^{(k)}$ for all $k\in \{1,\dots,N\}$ and $ j \in \{1,\dots,T/h\}$. Putting these estimates together, we see that
    \begin{gather}\label{Eq:estimOne}
    \begin{aligned}
        \E\Bigl[ \Bigl(\sigma(jh,\tiX_{jh}^{(i)})\xi_{jh}^{(i)} \Bigl(
        \Psi_j^{-\frac12}\bigl(&\tiX_{jh}^{(i)},\mu_{jh}\bigr)
        -
        \Psi_j^{-\frac12}\bigl(\tiX_{jh}^{(i)},\mu^N_{jh} \bigr) \Bigr) \Delta B_{jh}^{(i)}\Bigr)^2\Bigr]
        \\
        &\lesssim
        \E\Bigl[ \Bigl(
        \Psi_j^{-\frac12}\bigl(\tiX_{jh}^{(i)},\mu_{jh}\bigr)
        -
        \Psi_j^{-\frac12}\bigl(\tiX_{jh}^{(i)},\mu^N_{jh} \bigr) \Bigr)^2\Bigr] h
        \\
        &\lesssim \delta^{-2}\frac{h}{N} \lambda^{-1}+ 
        h\gamma^{-1} \delta^{-1}\lambda^{-3/2}
        \E\Bigl[\frac1N\sum_{k=1}^N \Bigl(\tiX_{jh}^{(k),N}-\tiX_{jh}^{(k)}\Bigr)^2 \Bigr]
        \end{aligned}
    \end{gather}
    \paragraph{Estimating the term (4)}
    By Assumption \ref{Aspt_new} and Remark \ref{Rem:Psi}, the terms in the square-root are bounded from below by $c_{min}^2$, therefore 
    \begin{align*}
        (4)\lesssim &\biggl| \biggl(\sigma^2(jh,\tiX^{(i)}_{jh})\bigl(\xi_{jh}^{(i)}\bigr)^2\Psi_j^{-1}\bigl(\tiX^{(i)}_{jh},\mu_{jh}\bigr)- \sigma^2(jh,\tiX^{(i),N}_{jh})\bigl(\xi_{jh}^{(i)}\bigr)^2\Psi_j^{-1}\bigl(\tiX^{(i),N}_{jh},\mu_{jh}^N\bigr)
        \biggr)
     \rhobar 
     \Delta W_{jh}^{(i)}\biggr| \hspace{-20pt}& 
        \\
        &\lesssim\Bigl| \sigma^2(jh,\tiX_{jh}^{(i)})\bigl(\xi_{jh}^{(i)}\bigr)^2 \Bigl(
        \Psi_j^{-1}\bigl(\tiX_{jh}^{(i)},\mu_{jh}\bigr)
        -
        \Psi_j^{-1}\bigl(\tiX_{jh}^{(i)},\mu^N_{jh} \bigr) \Bigr)\rhobar \Delta W_{jh}^{(i)}\Bigr| & 
        \eqqcolon (4.1) 
        \\
        &\qquad + \Bigl| \sigma^2(jh,\tiX_{jh}^{(i)})\bigl(\xi_{jh}^{(i)}\bigr)^2 \Bigl( \Psi_j^{-1}\bigl(\tiX_{jh}^{(i)},\mu^N_{jh} \bigr) - \Psi_j^{-1}\bigl(\tiX_{jh}^{(i),N},\mu^N_{jh}\bigr) \Bigr)\rhobar \Delta W_{jh}^{(i)} \Bigr|
        &\eqqcolon (4.2)
        \\ 
        &\qquad + \Bigl| \Bigl(\sigma^2(jh,\tiX_{jh}^{(i)})-\sigma^2(jh,\tiX_{jh}^{(i),N})\Bigr)\bigl(\xi_{jh}^i\bigr)^2\Psi_j^{-1}\bigl(\tiX_{jh}^{(i),N},\mu^N_{jh}
        \bigr)\rhobar \Delta W_{jh}^{(i)} \Bigr|
        &\eqqcolon (4.3)
    \end{align*}
    Recalling that $\Psi_j^{-1/2}$ is bounded due to the boundedness of $\xi$ we see that $\|\Psi_j^{-1} \|_{Lip}\lesssim \|\Psi_j^{-\frac12} \|_{Lip}$. By boundedness of $\sigma$ and  $\xi$, we now can repeat the arguments for the terms (2) and (3) from above to see that 
    \begin{align}\label{Eq:estimOneBar}
        \E\Bigl[(4.2)^2+(4.3)^2 \Bigr]
        &\lesssim h  {\gamma}^{-1} 
        \delta^{-2\gamma}
        \lambda^{-1-\gamma}
        \E\Bigl[ \Bigl( \tiX_{jh}^{(i)}-\tiX_{jh}^{(i),N}\Bigr)^2\Bigr]
    \end{align}
    as well as
    \begin{align}\label{Eq:estimTwoThreeBar}
        \E\Bigl[(4.1)^2\Bigr]\lesssim   \frac{h}{ N} \delta^{-2}\lambda^{-1}+ 
        {h}\gamma^{-1} \delta^{-1}\lambda^{-3/2}
        \E\Bigl[\frac1N\sum_{k=1}^N \Bigl(\tiX_{jh-h/2}^{(k),N}-\tiX_{jh-h/2}^{(k)}\Bigr)^2 \Bigr].
    \end{align}

    Recall that by construction $\tiX_{jh-h/2}^{(k),N}-\tiX_{jh-h/2}^{(k)}=\tiX_{jh}^{(k),N}-\tiX_{jh}^{(k)}$.
    Hence, when combining \eqref{Eq:estimTwoThree}, \eqref{Eq:estimOne}, \eqref{Eq:estimOneBar} and \eqref{Eq:estimTwoThreeBar} we see that there is a constant $C>0$ such that
    \begin{align*}
        \E\Bigl[ \Bigl( \tiX_{jh+h}^{(i)}-\tiX_{jh+h}^{(i),N}  \Bigr)^2   \Bigr]\le \Bigl(1 + &Ch\gamma^{-1} 
        \delta^{-2\gamma} \lambda^{-{1-\gamma}}   \Bigr)\E\Bigl[ \Bigl( \tiX_{jh}^{(i)}-\tiX_{jh}^{(i),N}\Bigr)^2\Bigr] 
         \\
         &+ C{h}{\gamma^{-1} \delta^{-1}}\lambda^{-3/2}
         \E\Bigl[\frac1N\sum_{k=1}^N \Bigl(\tiX_{jh}^{(k),N}-\tiX_{jh}^{(k)}\Bigr)^2 \Bigr]
         + C \frac{h}{ N}\delta^{-2} \lambda^{-1}.
    \end{align*}
    Summing up over all particles, choosing some arbitrary $\gamma\in(0,1/2)$ (for example $\gamma=1/3$) and by changing the constant $C$ we see that
    \begin{align*}
       \frac1N \sum_{i=1}^N \E\Bigl[ \Bigl( \tiX_{jh+h}^{(i)}-\tiX_{jh+h}^{(i),N}  \Bigr)^2   \Bigr]
       &\le
       \Bigl(1+C h \delta^{-1}\lambda^{-3/2}
       \Bigr) 
       \E\Bigl[  \frac1N\sum_{i =1}^N \Bigl( \tiX_{jh}^{(i)}-\tiX_{jh}^{(i),N}  \Bigr)^2   \Bigr] 
       \\
       &\qquad       +C \frac{h}{ N} \delta^{-2}\lambda^{-1}.
    \end{align*}
    A discrete Gronwall-type argument, see  Lemma \ref{lem:recurr_relation}, implies that
    \begin{align*}
        \frac1N\sum_{i =1}^N \E\Bigl[ \Bigl( \tiX_{jh+h}^{(i)}-\tiX_{jh+h}^{(i),N}  \Bigr)^2   \Bigr] 
        &\lesssim 
        \frac{h}{ N} \delta^{-2}\lambda^{-1}
        \sum_{j=1}^{T/h}  \Bigl(1+C h \delta^{-1} \lambda^{-3/2}\Bigr)^{j-1}
        \\
        &\lesssim  \frac{h}{ N} \delta^{-2}\lambda^{-1}
        \frac{\Bigl(1+C h \delta^{-1} \lambda^{-3/2} \Bigr)^{T/h}}{C h \delta^{-1} \lambda^{-3/2}}
        \\
        &\lesssim \frac1N \delta^{-1} \lambda^{1/2} 
        \Bigl(1+C h \delta^{-1} \lambda^{-3/2}\Bigr)^{T/h}.
    \end{align*}
    As $\lambda=\rhobar^2c_{min}h\lesssim h$ it follows from changing the constant $C$ that
    \begin{align*}
         \frac1N\sum_{i =1}^N \E\Bigl[ \Bigl( \tiX_{jh+h}^{(i)}-\tiX_{jh+h}^{(i),N}  \Bigr)^2   \Bigr] 
        &\lesssim \frac1{N} \lambda^{1/2} \delta^{-1} 
        \exp\Bigl(
       Th^{-1} \log\bigl( C\delta^{-1} h^{-1/2}\bigr)\Bigr).
    \end{align*}
\end{proof}

\section{Numerical Experiments}\label{Sec:Numerics}

Numerical experimenting can already be found in the literature, see for example \cite{GuyonHL12} or \cite{ReisingerTsianni2023}, where the authors used non-parametric methods for the conditional expectation or \cite{BayerEtal2022}, where kernel regression was used.
 Recall the \emph{half-step} scheme \eqref{Eq:Half_step}. 
The parameters used are 
\begin{multicols}{2}
\begin{itemize}
    \item $\xi$ for the stochastic volatility,
    \item $\sigma$ for the local volatility,
    \item $\delta$ for regularisation,
    \item $h$ for the discretisation length, 
    \item $\rho$ for the correlation between $B$ and $W$
    \item $c_{min}$ for the lower bound in Assumption \ref{Aspt_new}
    \item $N$ for the number of particles, which is equal to the number of Monte Carlo steps.
\end{itemize}
\end{multicols}
For illustration purposes we set $T=1$.
For the experiments we plotted the error of the numerical scheme against the number of simulation steps (corresponding to the reciprocal of step size). We used two test functions: first the smooth function $\cos(x)$ and secondly the function $x\mapsto (e^{x-1}-\frac12)^+$, which is not smooth but corresponds to the payoff of a call option for a stock market model where we model the log-stock $X$. We cannot simulate expectations directly, therefore we use Monte-Carlo estimation by just averaging over all our particles. 

Note that also the expectations are approximated via Monte-Carlo estimation. To avoid visual cluttering we did not plot confidence intervals; these would be of order $1/\sqrt{N}$.

\subsection{Fake Brownian motion}\label{Sec:Plots_FakeBM}

For this section we used $X_0=0$ and $\delta=0.001$. For the stochastic volatility we chose a rough-Bergomi-type model, i.e.\
\begin{align*}
    \xi_t=0.01 + 0.5 \exp\bigl( W_t^H-\frac{t^{2H}}2\bigr),
\end{align*}
with $c_{min}=0.05$\footnote{In theory there is no $c_{min}$ as in Assumption \ref{Aspt_new}, due to unbounded $\xi$. Using only finitely many paths in the algorithm we attempted to make $c_{min}$ as large as possible so that all terms in square roots in \eqref{Eq:Half_step} remained positive. 
Rounding down then led to the choice $c_{min}=0.05$.} 
and where $W_t^H$ denotes a Riemann-Liouville fractional Brownian motion with Hurst parameter $H=0.1$, fully correlated to $W$. Recall that 
$W_t^H=C_H\int_0^t (t-s)^{H-1/2}\,dW_s$, where $C_H$ is chosen such that $\E[(W_t^H)^2]=t^{2h}$ for all $t\in [0,T]$.

The correlation parameter $\rho$ was set as $\rho = -0.7$. We are interested in fake Brownian motion, therefore $\sigma(t,x)\equiv 1$ and $X_0=0$. With the mimicking process being a Brownian motion, we can calculate the reference values exactly: knowing the characteristic function of a standard normal distribution implies that $\E[\cos(W_1)]=e^{-1/2}$. The Black-Scholes formula, which is in closed form, tells us that $\E[(e^{W_1-1}-1/2)^+]\approx 0.269$.

\subsubsection[Error analysis in h]{Error analysis in $h$ and $N$}\label{sec:h_error}

We used stepsizes $h=1,\frac12,...,\frac1{50}$ and sampled $N\in \{1000,2000,3000,5000,8000\}$ paths for fixed regularisation parameter $\delta=0.001$.

Figure \ref{fig:coserror1} shows that the observed error is well within our theoretical error range. For this chosen function the error does not significantly change depending on the number of steps. It can also be seen that, as expected, a higher number of paths leads to a smaller error. In Figure \ref{fig:coserrorlogplot} we consider the logplot of the moving average over three data points and observe that up to a certain point the error behaves like $h+\delta$. The moving average was taken due to noisyness of the data.

When looking at the call-option-type payoff function the convergence looks much rougher, as can be observed in Figure \ref{fig:experror1}. Note that theoretically the payoff function is beyond the scope of this paper, due to unboundedness and by not being sufficiently smooth. Similar to before, in Figure \ref{fig:experrorlogplot} we plotted the moving average of the error on a log-scale. Note that due to the much rougher data the error no longer behaves like $h+\delta$, however the absolute error still becomes small relatively quick.

Lastly, we observe in Figure \ref{fig:var_quadvar1} that the variance of quadratic variation of $X$ does not seem to converge to $0$, independent of the number of paths. This is a good \lq\lq sanity check\rq\rq: our target process is a Brownian motion, and for that process the variance of the quadratic variation is equal to $0$. The previous error plots show, that the payoffs of $X$ get close to the payoffs of Brownian motion, indicating convergence of the respective marginal laws. However, quadratic variation seems to stay away from $0$, indicating that the path properties of $X$ differ from those of Brownian motion.

\begin{multicols}{2}

\begin{figure}[H]
    \centering
    \includegraphics[width=0.45\textwidth]{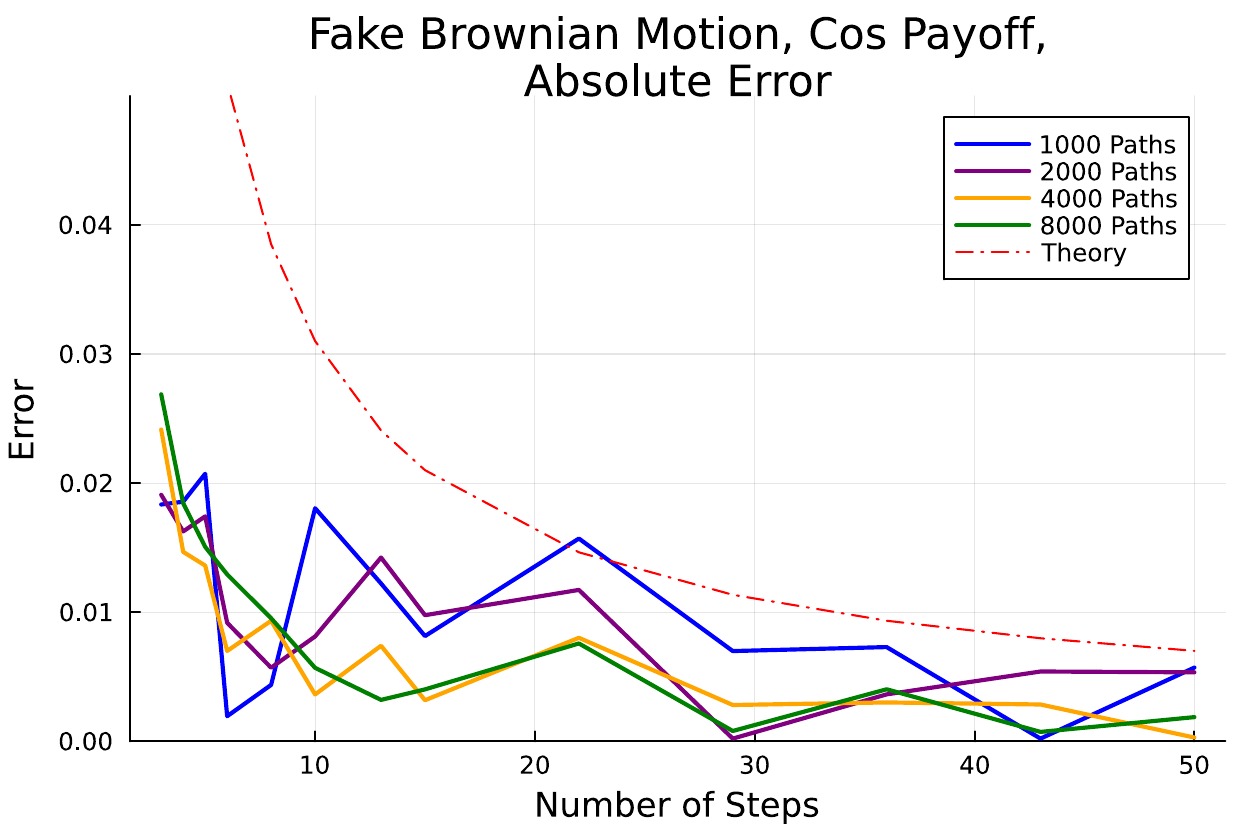}
    \caption{Error for payoff function $\cos(x)$, for a fake geometric Brownian motion. The dashed red curve corresponds to $h+\delta$.}
    \label{fig:coserror1}
\end{figure}

\begin{figure}[H]
    \centering
    \includegraphics[width=0.45\textwidth]{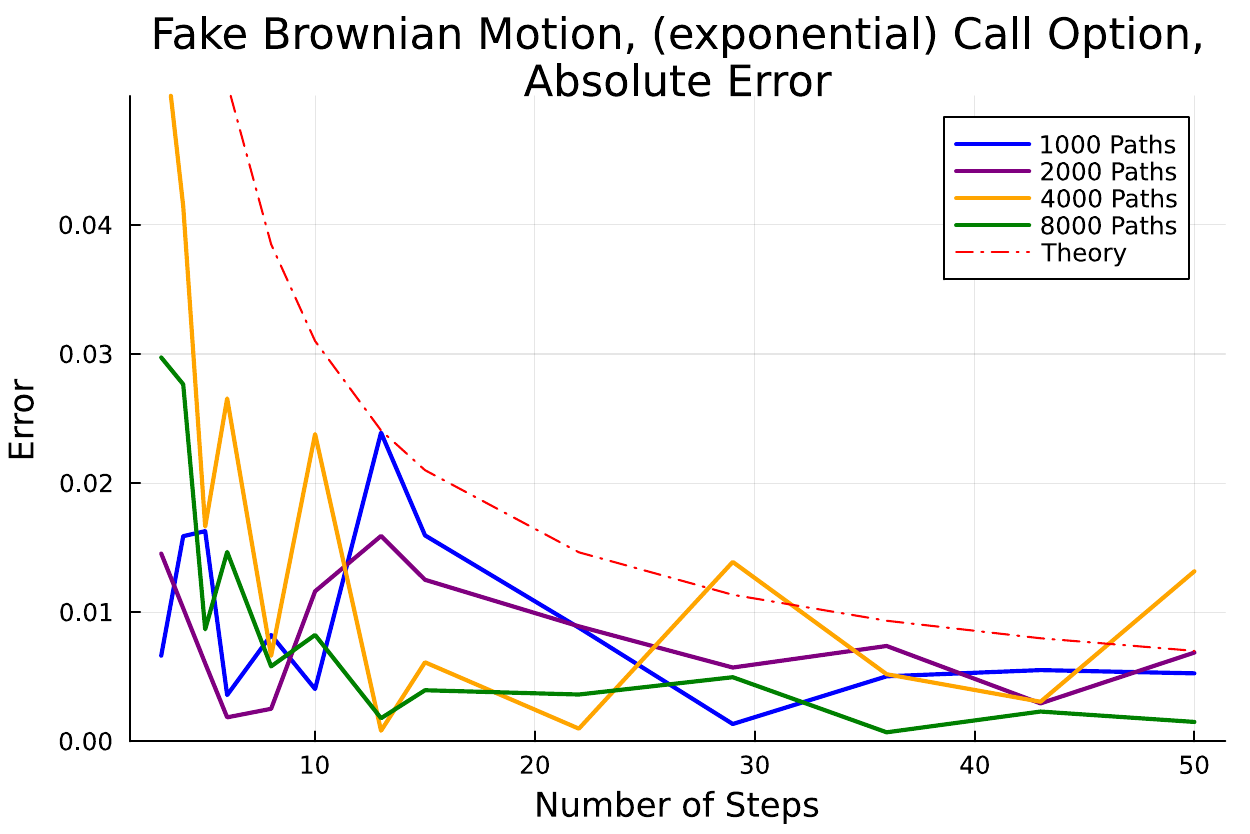}
    \caption{Error for payoff function $f(x)=(e^{x-1}-\frac12)^+$, for a fake geometric Brownian motion. $\,$}
     \label{fig:experror1}
\end{figure}
\end{multicols}


\begin{multicols}{2}

\begin{figure}[H]
    \centering
    \includegraphics[width=0.45\textwidth]{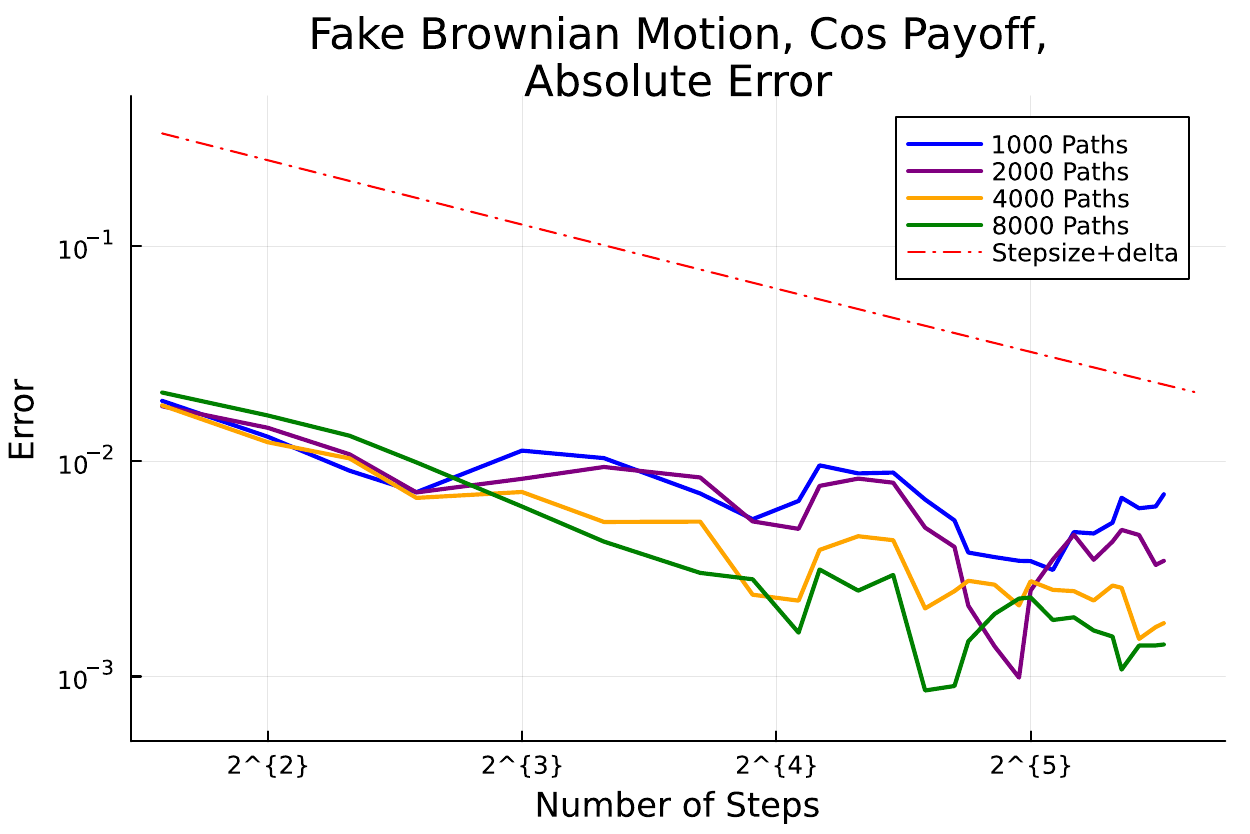}
    \caption{Logplot of error for  $\cos(x)$. }
    \label{fig:coserrorlogplot}
\end{figure}

\begin{figure}[H]
    \centering
    \includegraphics[width=0.45\textwidth]{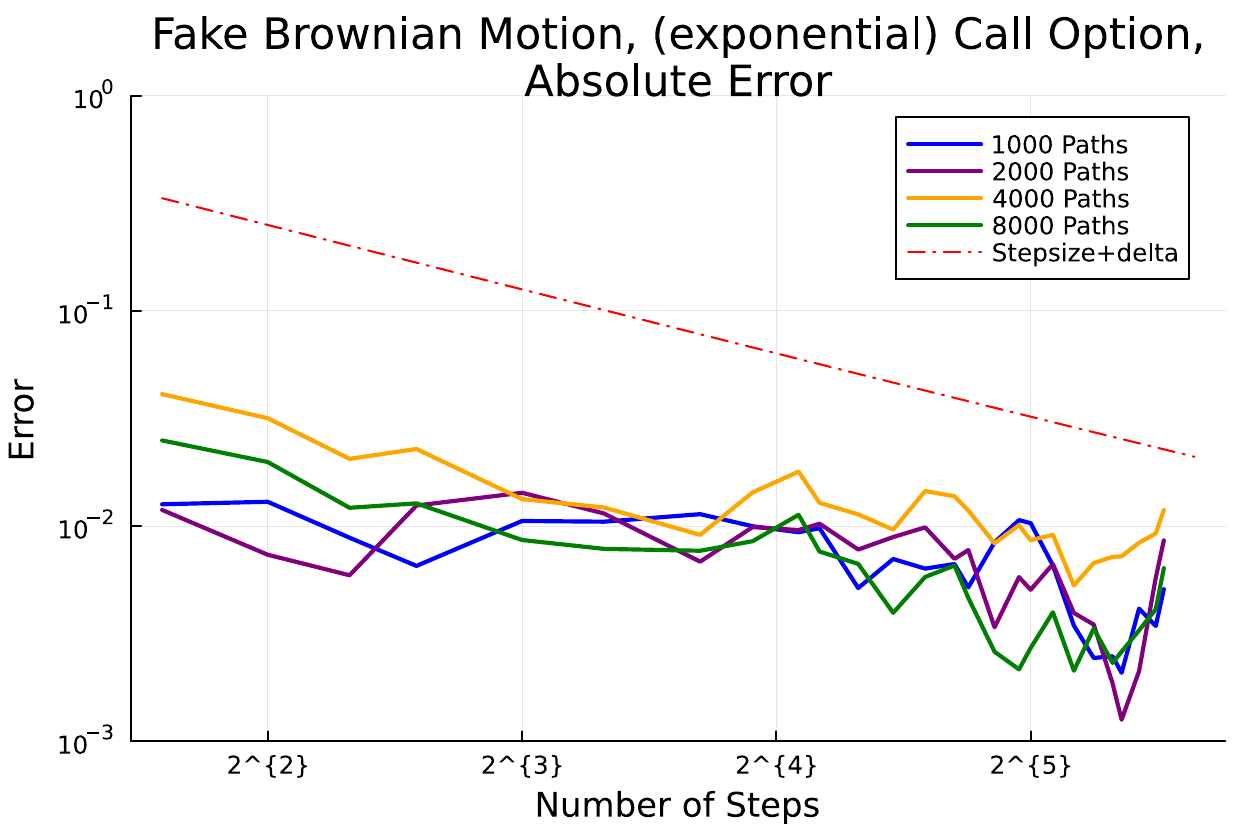}
    \caption{Logplot of error for payoff function $f(x)=(e^{x-1}-\frac12)^+$.}
     \label{fig:experrorlogplot}
\end{figure}

\end{multicols}

\begin{figure}[H]
    \centering
    \includegraphics[width=0.45\textwidth]{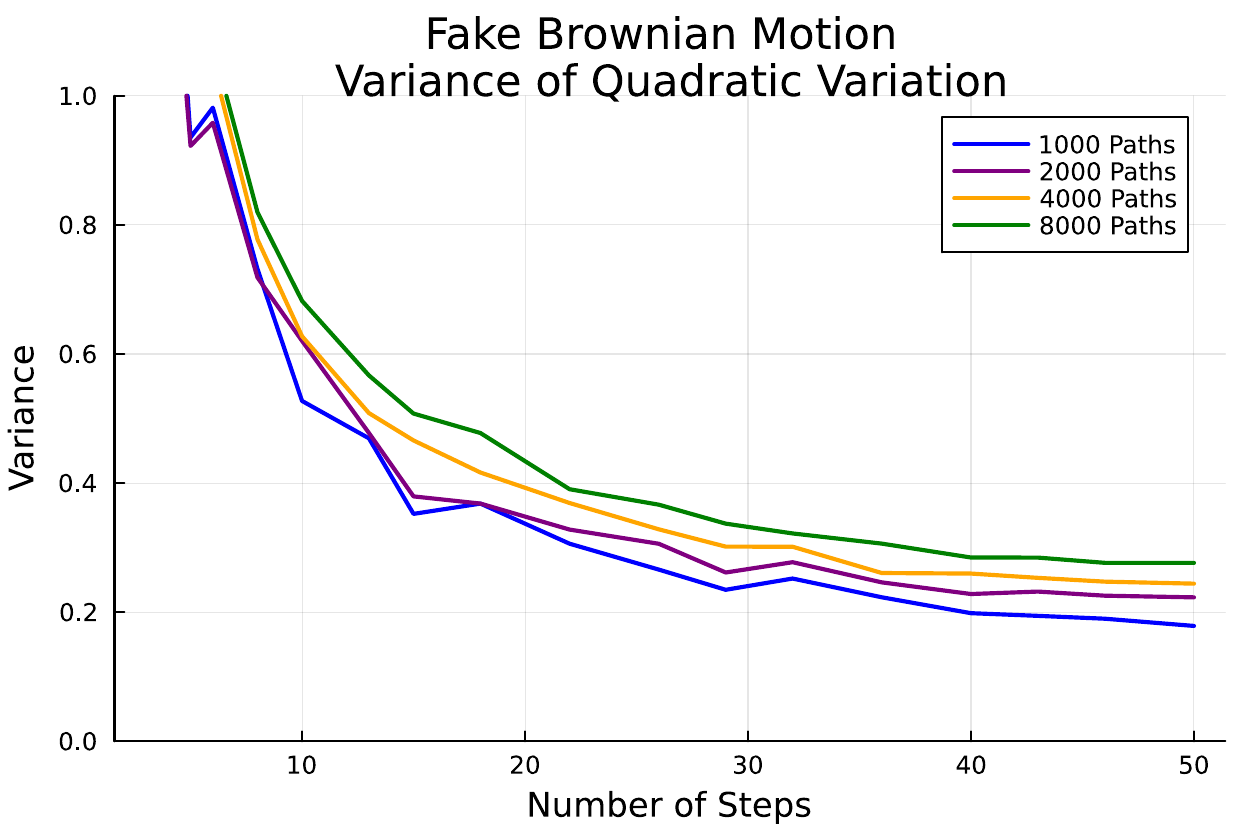}
    \caption{Plot of the variance of the quadratic variation.}
    \label{fig:var_quadvar1}
\end{figure}

\subsubsection[The role of delta]{The role of $\delta$}

Our numerical experimentation showed that there is a strong dependence on the parameter $\delta$, which we want to illustrate in this subsection. We use $N=6000$ and $\delta\in \{0.1,0.05,0.01,0.001\}$. The other parameters remain unchanged.

From Figure \ref{Fig:deltaplot} it can be seen that there is a substantial difference in the error for different values of $\delta$. This probably comes from the fact that the values of 
$\E\bigl[ \varphi_\lambda( y-\tiX_{jh-h/2}^{h,\delta})\xi_{jh}^2\bigr]$ might be small for some $y$ with non-vanishing probability, according to $\mathcal{L}(\tiX_{jh}^{h,\delta})$, 
leading to significant differences coming from the $\delta$-regularisation of the approximation of the conditional expectation as in Equation \eqref{eq_hatEdelta}.
It is also very remarkable that the error does not really depend on the step size. Stability  was already achieved at a step size of $h \approx 0.1$. Note that due to the relatively big number of particles and the smooth test function, the absolute value of the error is quite small. 

\begin{figure}[H]
    \centering
    \includegraphics[width= 0.45\textwidth]{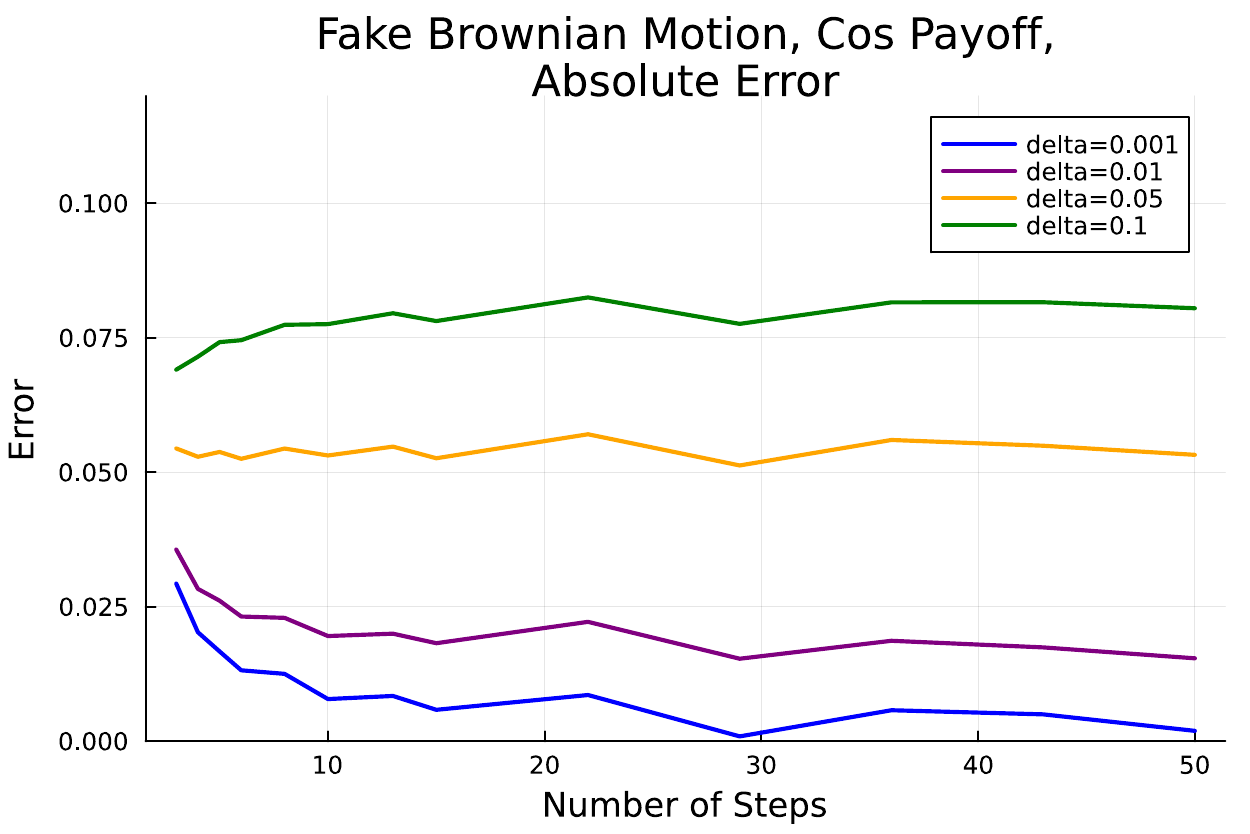}
    \caption{Error plot for differente $\delta$.}
    \label{Fig:deltaplot}
\end{figure}

\subsection{Tanh model}

In this section we use a non-trivial volatility, namely we choose $\sigma(x)=\frac34+\frac{\tanh(x)}4$. Such a toy model was previously introduced by \cite{FordeJacquier11} and has similar properties as required in Assumption \ref{Aspt_new}, due to boundedness and smoothness of $\tanh$. 
For this section we again used the same parameters as in Section \ref{sec:h_error}.

To the best of our knowledge, there does not exist a closed call-price formula for the tanh model. Therefore we used yet another Monte-Carlo simulation with $100$ Euler steps and $90^2$ Monte Carlo steps, but now for the target process $Y$ to approximate $\E\bigl[\cos(Y_1)\bigr]$ as well as $\E\bigl[ f(Y_1)\bigr]$ for the function $f(x)=(e^{x-1}-\frac12)^+$.
For this model we observe that for the smooth payoff function the simulated model is close to the \lq\lq real \rq\rq price and the error plot of Figure \ref{fig:coserror_tanh} has a much smaller magnitude than in Figure \ref{fig:coserror1}. We omit a log-plot due to the fast convergence speed. The same can be said for Figure \ref{fig:experror_tanh}.

\begin{multicols}{2}

\begin{figure}[H]
    \centering
    \includegraphics[width=0.45\textwidth]{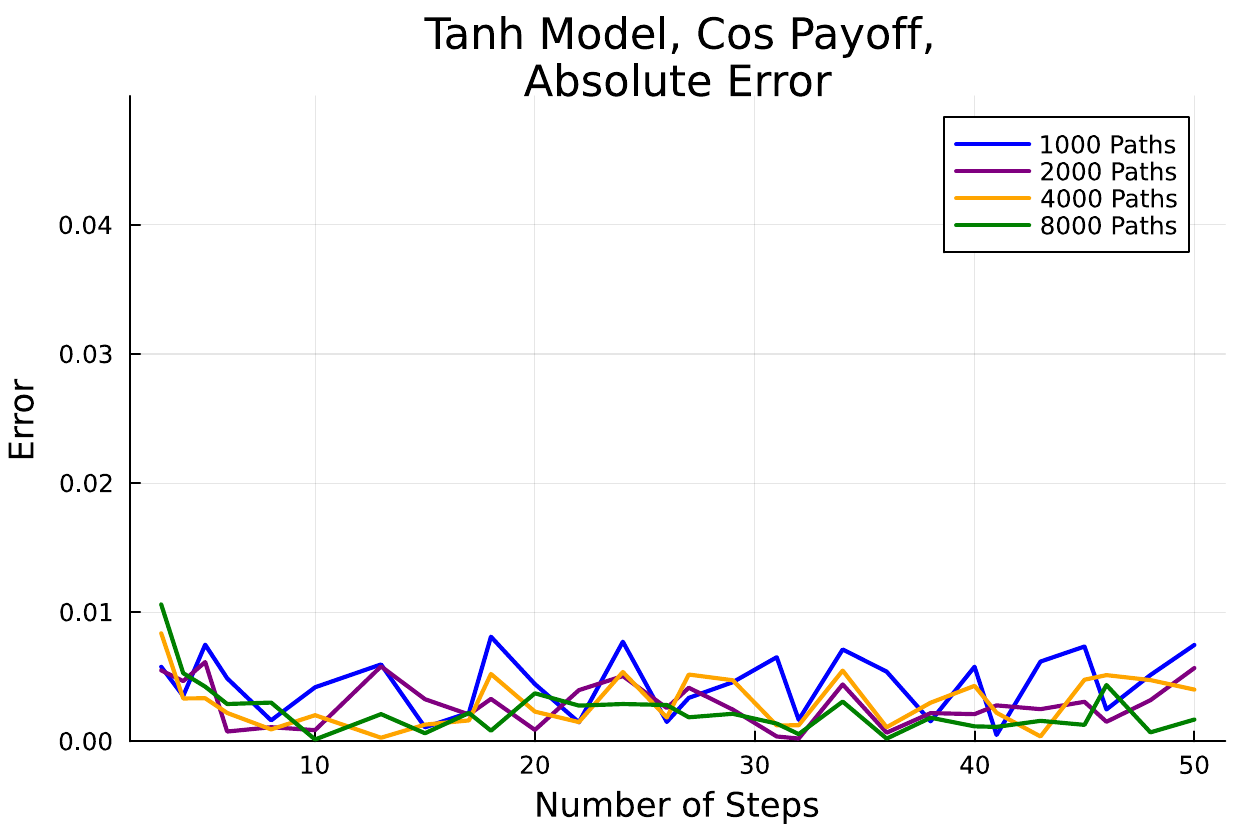}
    \caption{Error for payoff function $\cos(x)$. }
    \label{fig:coserror_tanh}
\end{figure}

\begin{figure}[H]
    \centering
    \includegraphics[width=0.45\textwidth]{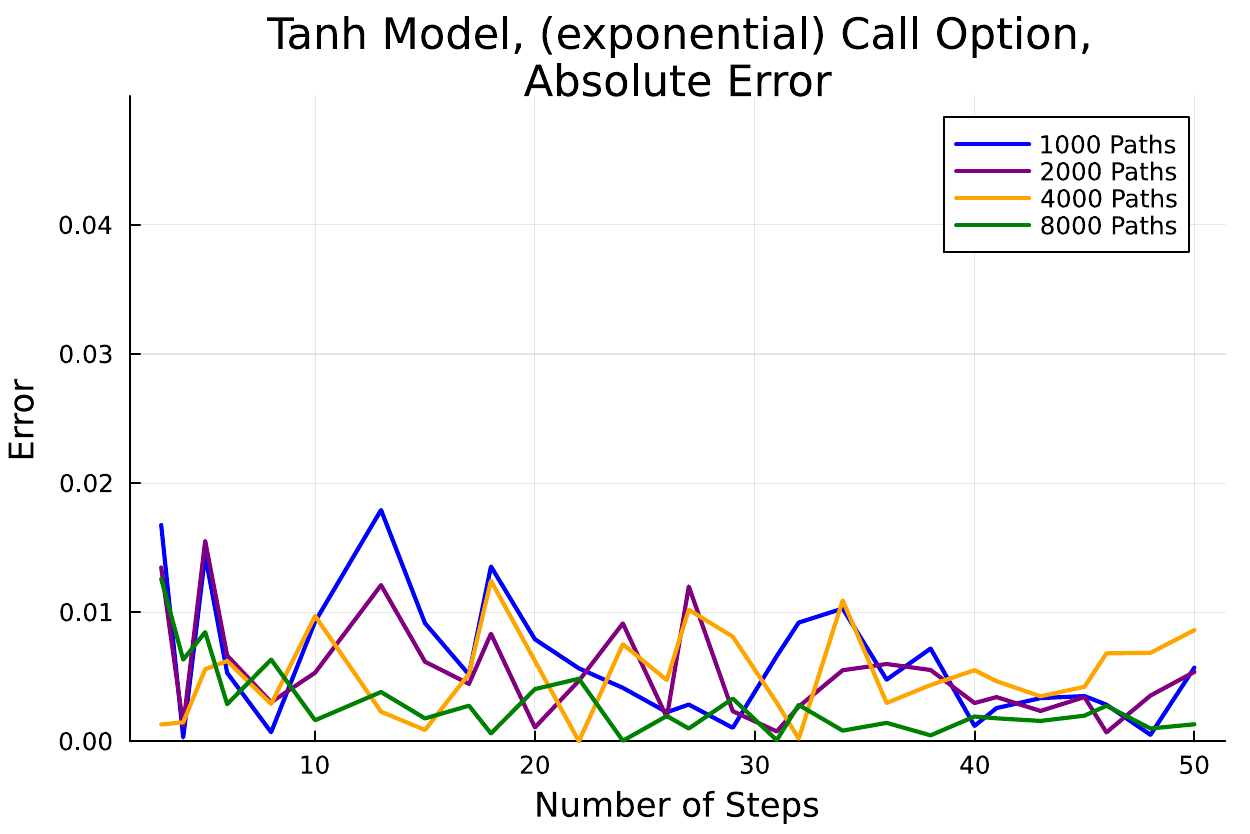}
    \caption{Error for payoff function $f(x)=(e^{x-1}-\frac12)^+$.}
    \label{fig:experror_tanh}
\end{figure}

\end{multicols}

\section{Weak Error for Particle 
Schemes with Nonparametric Estimators for Conditional Expectation}\label{Sec_Particles}

In this section we
still work with an approximation of conditional expectation, but this time there is no closed formula available. This leads us to studying non-parametric methods and to quantify the weak approximation error when using these methods.

For a parameter $\varepsilon \in (0,1)$ we introduce a Kernel function $K_\eps$ for our non-parametric estimation. Similar as above we introduce a parameter $\delta\in (0,1)$ to define the regularised Nadaraya-Watson estimator for conditional expectation:
\begin{definition}\label{Def_hEepsdelta}
    Let $X,\xi$ be two random variables, such that $\xi$ is square integrable. Let $\eps,\delta\in(0,1/2)$ and $K_\eps$ some Kernel satisfying Assumption \ref{Ass:kernel} below. 
    If $\mathbb{V}(X)>0$ define the estimator $\hE_\delta^\eps$ of the conditional expectation as
    \begin{align}\label{Eq:EhatEpsDelta}
        \hE_\delta^\eps\bigl[\xi^2\big|X\bigr]=\frac{
        \E\bigl[ \xi^2 K_\eps\bigl(x-X\bigr)\bigr]\big|_{x=X}+\delta
        }{{\E\bigl[ K_\eps\bigl(x-X\bigr)\bigr]\big|_{x=X}}+\delta }.
    \end{align}
    as well as $ \hE^\eps\bigl[\xi^2\big|X\bigr]= \hE^\eps_0\bigl[\xi^2\big|X\bigr]$.
    In case $\mathbb{V}(X)=0$ set $\hE_\delta^\eps\bigl[\xi^2\big|X\bigr]=\hE^\eps\bigl[\xi^2\big|X\bigr]=\E[\xi^2]$.
\end{definition}
    As we saw in previous sections, this estimator of the conditional expectation 
    can again be understood as a function $ \hE_\delta^\eps \bigl[ \xi_{jh}^2|\tiX^{h,\eps,\delta}_{jh}=x\bigr]\equiv \hPsi_j(x,\mathcal{L}(\xi_{jh},\tiX_{jh}^{h,\eps,\delta})) $, where for measures $\mu$ on $\mathbb{R}^2$ and $j \ge 1$
    the function $\hPsi_j$ is defined via 
    \begin{align}\label{Eq:PsiHat_Def}
        \hPsi_j(x,\mu)= \frac{\int_{\mathbb{R}\times\mathbb{R}}y^2K_{\eps}(x-z)
        \,d\mu(y,z)+\delta}{\int_{\mathbb{R}\times\mathbb{R}}K_{\eps}(x-z)
        \,d\mu(y,z)+\delta}.
    \end{align}
    For $j=0$ we set $\hPsi_0(x,\mu)=\int_{\mathbb{R}^2} y^2 \mu(dy,dz)$.

As done previously, we introduce an interacting particle system.
For this we denote by $\bigl( W^{(i)})_{i=1}^N$ independent Brownian motions. 
For $j\in \{0,\dots,\frac Th-1\}$,  we consider the following Euler scheme:
\begin{align}\label{Eq:Sec7}
    \hX_{jh+h}^{h,(i),N,\eps,\delta}&=\hX_{jh}^{h,(i),N,\eps,\delta}+ \sigma(jh,\hX_{jh}^{h,(i),N,\eps,\delta})\xi_{jh}^{(i)} \hPsi_j^{-\frac12}\bigl(\hX_{jh}^{h,(i),N,\eps,\delta},\mu^N_{jh}\bigr)\bigl( W_{jh+h}^{(i)}- W_{jh}^{(i)}\bigr),
    \\
    \hX_0^{h,(i),N,\eps,\delta}&=x_0,\nonumber
\end{align}
where, for $j\ge 0$ we denote by $\mu_{jh}^N$  the empirical measure of the interacting particle system. More precisely $\mu_{jh}^N(dx,dy)=\frac1N \sum_{j=1}^N \delta_{\xi_{jh}^{(i)},\hX_{jh}^{h,(i),N,\eps,\delta}}(dy,dx)$. 
Then we can show the following.

\begin{theorem}\label{Thm_Particle_Rate2}
    Let Assumption  \ref{Ass:Half_Step} as well as Assumption \ref{Aspt_new} hold and $h,T$ such that $\frac Th \in \mathbb{N}$. Let $K$  satisfy Assumption \ref{Ass:kernel} below and $f$ be three times weakly differentiable with derivatives of polynomial growth. Then for $h,\eps,\delta\in(0,1/2)$ 
    satisfying $\eps \lesssim h$ there is some $C>0$, such that
    \begin{align*}
         \Bigl|\E\bigl[ f(Y_T)\bigr]-\E\bigl[f\bigl(\hX_T^{h,(1),N,\eps,\delta}\bigr)\bigr]\Bigr|
        &\lesssim 
        h + \eps h^{-1} + \delta^{1-}  
        + \frac1{\sqrt{N}}
        \eps
        \exp\Bigl(
        \frac T{2}h^{-1} \log\bigl( C h \delta^{-2}\eps^{-4}\bigr)\Bigr).
    \end{align*}
\end{theorem}

Note  that the weak error rate observed in Theorem \ref{Thm_Particle_Rate2} is worse than the one observed in Theorem \ref{Thm_Particle_Rate}.
Indeed, choosing $p\approx 1$, $\delta =h$ and  $ \varepsilon = h^2 $ in the latter (so that $h + \varepsilon h^{- 1} \sim h$) we see $h^{- 3 / 2}$ in the ``log-exponent'' for the half-step scheme, compared with $h^{- 9}$ without half-step.

The remaining sections focus on proving Theorem \ref{Thm_Particle_Rate2}: 
Section \ref{Sec:Cond_exp_estim} establishes the convergence speed of the Nadaraya-Watson estimator for conditional expectation. This is then extended by showing convergence with the additional regularisation parameter. Finally, in Section \ref{Sec:POC} we show propagation of chaos and convergence to the target process, using results from previous sections.

\subsection{Nonparametric Estimator for Conditional Expectation}\label{Sec:Cond_Exp_Estimators}

Following up on the discussion in Section \ref{Approximation} above we now propose a method for a suitable estimator $\hE$ for conditional expectation. 
Let $\bigl(W^{(i)},\xi^{(i)}\bigr)_{i=1,\dots,N}$ be independent copies of $(W,\xi)$. Assume that, under a simplified notation, $\hX_{jh}^{h,(i)}$ is already defined via some interacting particle system, as in Equation \eqref{Eq:Sec7}. Let $K_\varepsilon:\mathbb{R}\rightarrow\mathbb{R}$ be a kernel function parameterized by $\varepsilon$ such that $K_\eps$ converges to the Dirac distribution $\delta_0$ as $\eps\rightarrow 0$. 
Then we define for $j\ge 1$
\begin{align}\label{eq_hEeps}
    \hE^{\eps,N}\bigl[\bigl(\xi_{jh}^{(i)}\bigr)^2\big|\hX_{jh}^{h,(i)}\bigr]=
    \begin{cases}
        \frac{\sum_{k=1}^N \bigl(\xi_{jh}^{(k)}\bigr)^2 K_\eps\bigl(\hX_{jh}^{h,(i)}-\hX_{jh}^{h,(k)}\bigr)}{\sum_{k=1}^N  K_\eps\bigl(\hX_{jh}^{h,(i)}-\hX_{jh}^{h,(k)}\bigr)} & \text{for } j \ge 1,
        \\
        \E[\bigl(\xi_{0}^{(i)}\bigr)^2\bigr]
        & \text{for } j=0
    \end{cases}.
\end{align}
This estimator is also referred to as Nadaraya-Watson estimator. It has the great advantage that the approximation of the conditional expectation depends smoothly on the condition $\hX_{jh}^{h,(i)}.$ To balance the singularities of the fraction, similarly to Section \ref{Sec:Half_Step}, we reintroduce the regularisation parameter $\delta>0$ for the estimator 
\begin{align}\label{eq_hEepsdelta}
    \hE_{\delta}^{\eps,N}\bigl[\bigl(\xi^{(i)}_{jh}\bigr)^2\big|\hX_{jh}^{h,(i)}\bigr]=
    \begin{cases}
        \frac{\frac1N\sum_{k=1}^N \bigl(\xi_{jh}^{(k)}\bigr)^2 K_\eps\bigl(\hX_{jh}^{h,(i)}-\hX_{jh}^{h,(k)}\bigr)+\delta}{\frac1N\sum_{k=1}^N  K_\eps\bigl(\hX_{jh}^{h,(i)}-\hX_{t}^{h,(k)}\bigr)+\delta} & \text{for } j \ge 1,
        \\
         \E[\bigl(\xi_{0}^{(i)}\bigr)^2\bigr]
        & \text{for } j=0
    \end{cases}.
\end{align}
Note that for $j=0$ we have the equality $  \hE^{\eps,N}\bigl[\bigl(\xi_{0}^{(i)}\bigr)^2\big|\hX_{0}^{h,(i)}\bigr]= \hE^{\eps,N}_\delta\bigl[\bigl(\xi_{0}^{(i)}\bigr)^2\big|\hX_{0}^{h,(i)}\bigr]=\hE[\bigl(\xi_{0}^{(i)}\bigr)^2\bigr]$.
This method was also proposed in \cite{GuyonHL2013}, where the authors introduced this method for a practical audience and has been further developed in \cite{Muguruza2019}, where the author proposes a similar algorithm but the kernels $K_\eps$ also depend on the particles.

The computation time is of order $N^2$, which makes good simulations computationally expensive. In other applications, comparable to estimating at just one single time point the so-called \emph{one-fifth-rule}, see for example \cite[Section 2.2.1]{LiScott07}, is used, i.e.\ one chooses $\eps \approx N^{-1/5}$.
 For the Kernel $K_\eps$ we have the following
 \pagebreak[2]
\begin{assumption}\label{Ass:kernel}
    For $\eps>0$ we define a function $K_\eps: \mathbb{R}\rightarrow [0,\infty)$ satisfying the following properties:
    \begin{enumerate}
        \item \emph{$L^1$-Scaling:} $K_\eps(\cdot)=\frac1\eps K_1\bigl(\frac\cdot\eps\bigr)$,
        \item \emph{Compact support:} $\operatorname{supp} K_1 \subseteq [-1,1]$,
        \item \emph{Regularity:}
        $K_1 \in C_1$,
        \item \emph{Stability near zero:} There is a constant $c$ such that $K_1 \ge c\one_{(-\frac12,\frac12)}$.
    \end{enumerate}
\end{assumption}
\begin{example}
    The function 
    \begin{align*}
        K_1(x)=\max\Bigl\{\bigl( 1- x^2\bigr),0\Bigr\}^2
    \end{align*}
satisfies Assumption \ref{Ass:kernel}, 2.-4.\end{example}

\begin{remark}[Kernel Method]
    In  \cite{BayerEtal2022} the authors introduced a kernel method for approximating conditional expectation: similar to the particle method the idea is to approximate the function determined by the conditional expectation and not the (stochastic) random variable by itself. 
Recall that for $\xi,X \in L^1$ the conditional expectation is given by $\E\bigl[\xi\big|X\bigr]=f\bigl(X\bigr)$ where
\begin{align*}
    f=\operatorname*{argmin}_{g \in \mathcal{L}^0} \E\bigl[\bigl(g(X)-\xi\bigr)^2 \bigr].
\end{align*}
Here $\mathcal{L}^0$ denotes the class of measurable functions $g:\mathbb{R}\rightarrow \mathbb{R}$. The authors then introduced a class of functions $\mathcal{H}$, corresponding to a reproducing kernel Hilbert space endowed with a scalar product $\langle .,. \rangle_{\mathcal{H}}$ as well as a regularizing parameter $\theta>0$. They then proposed the ridge regression
\begin{align*}
    f\approx \operatorname*{argmin}_{g \in \mathcal{H}} \Bigl\{ \E\bigl[\bigl(g(X)-\xi\bigr)^2 \bigr] + \theta \|g\|_{\mathcal{H}}\Bigr\}.
\end{align*}
If the function of the conditional expectation is already in the reproducing kernel Hilbert space, then this yields a very good approximation. For more general functions, that are not element in the RKHS, one can obtain good approximations if the function is in the completion of the RKHS w.r.t.\ a suitable metric.
\end{remark}

\subsection{Approximation: Kernel Method and Regularisation}\label{Sec:Cond_exp_estim}

We start this section by showing $L^p$ convergence for the Nadaraya-Watson estimator.

We define the regularized Euler scheme for some arbitrary $\eps,\delta \in (0,1)$
\begin{gather}\label{eq_Xhepsdelta} 
\begin{aligned}
    \hX_{jh+h}^{h,\eps,\delta}&=\hX_{jh}^{h,\eps,\delta}+\frac{\xi_{jh}\sigma\bigl(jh,\hX_{jh}^{h,\eps,\delta}\bigr)}{\sqrt{\hE_\delta^\eps\bigl[\xi_{jh}^2\big|\hX_{jh}^{h,\eps,\delta}\bigr]}} \bigl(W_{jh+h}-W_{jh}\bigr),
    \\
    \hX_0^{h,\eps,\delta}&=x_0.
\end{aligned}
\end{gather}
\begin{remark}
    Due to the ellipticity assumption the condition $\mathbb{V}(\hX_{jh}^{h,\eps,\delta})=0$ in Definition \ref{Def_hEepsdelta} is satisfied if and only if $j=0$, making the choice of the estimator similar to the one for the particle system in Equation \eqref{eq_hEepsdelta} above.
\end{remark}
Similar to Theorem \ref{thm_HScondExp_err} above we can now quantify the error introduced by the estimator $\hE_\delta^\eps$ in the following theorem.
\begin{theorem}\label{thm_condExp_err}
        Suppose Assumptions \ref{Ass:Half_Step} and \ref{Aspt_new} hold. Let $K_\eps$ satisfy Assumption \ref{Ass:kernel}. Let $h,T>0$ such that $T/h\in{\mathbb N}$, $\eps,\delta \in (0,1)$ and $ \hX_{jh}^{h,\eps,\delta}$ be defined in Equation \eqref{eq_Xhepsdelta} for $j\in \{0,\dots,\frac Th\}$. Recall $\lambda$ from Equation \eqref{Eq_lambda}.
        Let $f$ be three times differentiable with derivatives of polynomial growth, then it holds for any $p>1$ that
        \begin{align}\label{Eq:CondExp_Err_2}
            \Bigl|\E\bigl[ f(Y_T)\bigr]-\E\bigl[f\bigl(\hX_T^{h,\eps,\delta}\bigr)\bigr]\Bigr|\lesssim h + \eps h^{-1} e^{\frac {\eps^2}{\lambda}} + \delta^{\frac1p}.
        \end{align}
         Here the hidden constant explodes as $p\rightarrow 1$ or $p \rightarrow \infty$.
\end{theorem}

\begin{remark}
    The proof of Lemma \ref{Lem:FrakCBound} and the definition of $\lambda$ shows that the second term explodes as $\rhobar \rightarrow 0$. As $\lambda$ linearly depends on $h$, the exponential in the second term is no issue as one would choose $\eps\lesssim h$ anyway to have small contribution from the term $\eps h^{-1}$. 
\end{remark}

\begin{lemma}\label{Lem_regular_kernel} Assume the setting of Theorem \ref{thm_condExp_err} and recall $\hE^\eps_\delta$ from Definition \ref{Def_hEepsdelta}. 
    Then for any $p\ge 1$ it holds that 
    \begin{align*}
        \Bigl\| \hE_\delta^\eps\bigl[\xi_{jh}^2| \hX_{jh}^{h,\eps,\delta} \bigr]-\hE^\eps\bigl[\xi_{jh}^2| \hX_{jh}^{h,\eps,\delta} \bigr] \Bigr\|_{L^p} \lesssim 
        \Bigl(\delta \sqrt{-\log(\delta)}\Bigr)^{\frac1p}.
    \end{align*}
\end{lemma}
\begin{proof}
    The proof is the same as the proof as Lemma \ref{Lem_regularisation}, with $K_\eps=\varphi_\lambda$ and by replacing $\sqrt{\lambda}$ by $\varepsilon$. 
    Note that  Assumption \ref{Ass:kernel} tells us, that for any positive function $f$ it holds that 
    \begin{align*}
         \int_{-1}^{1} K_1({y}) f(x+\eps y)\,dy 
       \gtrsim \int_{-\frac12}^{\frac12} f(x+\eps y)\,dy.
    \end{align*}
\end{proof}

      \begin{lemma}\label{Lem:FrakCBound}
          Assume the setting of Theorem \ref{thm_condExp_err} and recall $\hE^\eps_\delta$ from Definition \ref{Def_hEepsdelta}. It holds for  $j \in \{1,\dots,\frac Th\}$ that
          \begin{align*}
               \Bigl\| \E\bigl[\xi_{jh}^2| \hX_{jh}^{h,\eps,\delta} \bigr]-\hE^\eps\bigl[\xi_{jh}^2| \hX_{jh}^{h,\eps,\delta} \bigr] \Bigr\|_{L^p}
               \lesssim \eps h^{-1} e^{\frac {\eps^2}{\lambda}}
          \end{align*}
      \end{lemma}
      \begin{proof}
            We fix some index $j \in \{1,\dots,T/h\}$. 
            Under Assumption \ref{Ass:Half_Step} we can write 
            \begin{align*}
                \hX_{jh}^{h,\eps,\delta}&=\hX_{jh-h}^{h,\eps,\delta}+\frac{\xi_{jh-h}\sigma\bigl(jh-h,\hX_{jh-h}^{h,\eps,\delta}\bigr)}{\sqrt{\hE_\delta^\eps\bigl[\xi_{jh-h}^2\big|\hX_{jh-h}^{h,\eps,\delta}\bigr]}} 
                \rho
                \bigl(B_{jh}-B_{jh-h}\bigr)
                \\
                &\qquad \qquad + \frac{\xi_{jh-h}\sigma\bigl(jh-h,\hX_{jh-h}^{h,\eps,\delta}\bigr)}{\sqrt{\hE_\delta^\eps\bigl[\xi_{jh-h}^2\big|\hX_{jh-h}^{h,\eps,\delta}\bigr]}} 
                \rhobar
                \bigl(\bB_{jh}-\bB_{jh-h}\bigr),
            \end{align*}
            for two independent Brownian motions $B$ and $\bB$ such that $\xi$ and $\bB$ are independent.

           We can apply Lemma \ref{Lem:Density} with $Z_1= \hX_{jh}^{h,\eps,\delta}$, $\eta=\rhobar \frac{\xi_{jh-h}\sigma(jh-h,\hX_{jh-h}^{h,\eps,\delta})}{\sqrt{\hE_\delta^\eps[\xi_{jh-h}^2|\hX_{jh-h}^{h,\eps,\delta}]}}$ 
           and $Z_0=\hX_{jh-h}^{h,\eps,\delta}+\rho \frac{\xi_{jh-h}\sigma(jh-h,\hX_{jh-h}^{h,\eps,\delta})}{\sqrt{\hE_\delta^\eps[\xi_{jh-h}^2|\hX_{jh-h}^{h,\eps,\delta}]}}(B_{jh}-B_{jh-h})$. 
           Note that $\eta >\rhobar c_{min}$.
           We denote $Z_1=\hX_{jh}^{h,\eps,\delta}$ and $\xi\coloneqq \xi_{jh-h}$.
           By independence of $(Z_0,\xi)$ and $\Delta \bB_{jh-h}\coloneqq (\bB_{jh}-\bB_{jh-h}\bigr)$, the density of $Z_1$ with respect to the Lebesgue measure is 
           $f_{Z_1}(x)=\E\Bigl[\varphi_{h\eta^2}\bigl(x-Z_0\bigr) \Bigr]$ and for $\psi:\R\to\R$ measurable and bounded,
           \begin{align*}
               \E\bigl[ \xi^2\psi(Z_1)\bigr]=\E\Bigl[\int_\R\xi^2\psi(Z_0+x)\varphi_{h\eta^2}\bigl(x\bigr)dx\Bigr]=\int_\R \E\Bigl[\xi^2\varphi_{h\eta^2}\bigl(x-Z_0\bigr) \Bigr]\psi(x)dx
           \end{align*}
           so that
           \begin{align*}\E\bigl[ \xi^2|Z_1\bigr]&=m(Z_1)\mbox{ with}\\
        m(x)&:= \frac1{f_{Z_1}(x)}\E\Bigl[\xi^2 \varphi_{h\eta^2}\bigl(x-Z_0\bigr) \Bigr].
    \end{align*}

        We use the tower property of conditional expectation and see by the support properties in Assumption \ref{Ass:kernel} that there is some ($\mathcal{F}_{jh}$-measurable) $x^*$ taking values in $(x-\eps,x+\eps)$ such that
\begin{align*}
        \E\bigl[\xi^2\big|Z_1\bigr]-\frac{\E\bigl[K_\eps(Z_1-x)\xi^2]}{\E\bigl[K_\eps(Z_1-x)]}\bigg|_{x=Z_1}
        \\&=
        \frac{\E\bigl[m(x)K_\eps(Z_1-x)\bigr]-\E\Bigl[K_\eps(Z_1-x)\E\bigl[\xi^2\big|Z_1\bigr]\Bigr]}{\E\bigl[K_\eps(Z_1-x)]}\bigg|_{x=Z_1}
        \\&=
        \frac{\E\bigl[m(x)K_\eps(Z_1-x)\bigr]-\E\bigl[K_\eps(Z_1-x)m(Z_1)\bigr]}{\E\bigl[K_\eps(Z_1-x)]}\bigg|_{x=Z_1}
        \\
        &=\frac{\E\bigl[\bigl(Z_1-x\bigr)m'\bigl(x^*\bigr)K_\eps(Z_1-x)\bigr]}{\E\bigl[K_\eps(Z_1-x)]}\bigg|_{x=Z_1}
        \\
        &\le \eps\sup_{z \in (Z_1-\eps,Z_1+\eps)}|m'(z)|.
    \end{align*}
    By the chain rule the derivative is given by
    \begin{align*}
        m'(x)=\frac1{f_{Z_1}(x)} \E\Bigl[\xi^2 \partial_x \varphi_{h\eta^2}\bigl(x-Z_0\bigr) \Bigr] - \frac1{f_{Z_1}^2(x)}\E\Bigl[\xi^2 \varphi_{h\eta^2}\bigl(x-Z_0\bigr) \Bigr]\E\Bigl[\partial_x\varphi_{h\eta^2}\bigl(x-Z_0\bigr)\Bigr]. 
    \end{align*}
    Recall that $\xi^2$ is bounded and that for $\lambda>0$, $\partial_x \varphi_\lambda(x)=-\frac x\lambda \varphi_\lambda(x)$.  The fact that $\eta$ is bounded from below and the boundedness of $m(x)=\frac1{f_{Z_1}(x)}\E\bigl[\xi^2 \varphi_{h\eta^2}\bigl(x-Z_0\bigr) \bigr]$ imply that
    \begin{align*}
        \bigl|m'(x)\bigr|\lesssim \frac{2}{h} \frac{\E\Bigl[\varphi_{h\eta^2}\bigl(x-Z_0\bigr)\bigl|x-Z_0\bigr|\Bigr]}{\E\Bigl[\varphi_{h\eta^2}\bigl(x-Z_0\bigr)\Bigr]}.
    \end{align*}

    We now want to find an upper bound of the r.h.s.\
    whose $p$-th power is integrable w.r.t.\ the density of $Z_1$. We introduce an auxiliary parameter $\gamma\in(0,1)$ such that $\gamma \le \frac1{2p}$.
    By Lemma \ref{Lem:Gauss_facts} and Jensen's inequality, as well as the boundedness of $\eta$ from below,
    \begin{align*}
        \bigl|m'(x)\bigr|\lesssim h^{-1+1/2-\gamma/2} \gamma^{-1/2} \frac{\E\Bigl[\varphi_{h\eta^2}\bigl(x-Z_0\bigr)^{1-\gamma}\Bigr]}{\E\Bigl[\varphi_{h\eta^2}\bigl(x-Z_0\bigr)\Bigr]}
        \lesssim h^{
        -\frac{1+\gamma}2} \gamma^{-\frac12} \E\Bigl[\varphi_{h\eta^2}\bigl(x-Z_0\bigr)\Bigr]^{-\gamma}.
    \end{align*}
    Lemma \ref{Lem:Gauss_facts} furthermore tells us that, recalling $h\eta^2 \ge \lambda$ with $\lambda$ defined in \eqref{Eq_lambda},
    \begin{align*}
        \sup_{z \in (x-\eps,x+\eps)}|m'(z)| &\lesssim  h^{
        -\frac{1+\gamma}2} \gamma^{-\frac12}\Bigl(e^{\frac{\eps^2}{\gamma \lambda}} \sqrt{h}^{-\gamma}\Bigr)^{\gamma}\E\Bigl[\varphi_{h\eta^2}\bigl(x-Z_0\bigr)^{1+\gamma}\Bigr]^{-\gamma}
        \\
        &\lesssim  \underbrace{h^{
        -\frac{1+\gamma+\gamma^2}2} \gamma^{-\frac12}e^{\frac{\eps^2}{ \lambda}}}_{C_{h,\gamma}} \E\Bigl[\varphi_{h\eta^2}\bigl(x-Z_0\bigr)^{1+\gamma}\Bigr]^{-\gamma}.
    \end{align*}
    Using our representation of $f_{Z_1}$, Jensen's inequality and once more Lemma \ref{Lem:Gauss_facts} imply
    \begin{align*}
    \Bigl\| \E\bigl[\xi_{}^2| Z_1 \bigr]-\hE^\eps\bigl[\xi_{}^2| Z_1 \bigr] \Bigr\|_{L^p}^p
    &\lesssim
    \eps^p
        \Bigl\|
        \sup_{z \in (Z_1-\eps,Z_1+\eps)}m'(z)
        \Bigr\|_{L^p}^p 
        \\
        &\lesssim \eps^p C_{h,\gamma}^p \int_{\mathbb{R}} \E\Bigl[\varphi_{h\eta^2}\bigl(x-Z_0\bigr)^{1+\gamma}\Bigr]^{-\gamma p} \E\Bigl[\varphi_{h\eta^2}\bigl(x-Z_0\bigr)\Bigr] \,dx
        \\
        &\lesssim \eps^p
        C_{h,\gamma}^p \int_{\mathbb{R}} \E\Bigl[\varphi_{h\eta^2}\bigl(x-Z_0\bigr)^{1+\gamma}\Bigr]^{\frac{1}{1+\gamma}-\gamma p}  \,dx
        \\
        &\lesssim \eps^p
        C_{h,\gamma}^p {h}^{-\frac\gamma2\bigl(\frac{1}{1+\gamma}-\gamma p\bigr)}\int_{\mathbb{R}} \E\Bigl[\varphi_{h\eta^2}\bigl(x-Z_0\bigr)\Bigr]^{\frac{1}{1+\gamma}-\gamma p}  \,dx.
    \end{align*}
    We now try to find a value for $\gamma$, that can depend on $p$, such that the last integral remains finite. 
    Denoting $\beta_{\gamma,p}=\frac{1}{1+\gamma}-\gamma p \in (0,1),$ it follows from H\"older's inequality that
    \begin{align*}
        \int_{\mathbb{R}} \E\Bigl[\varphi_{h\eta^2}\bigl(x-Z_0\bigr)&\Bigr]^{\beta_{\gamma,p}}  \,dx =\int_{\mathbb{R}} \E\Bigl[\varphi_{h\eta^2}\bigl(x-Z_0\bigr)\Bigr]^{\beta_{\gamma,p}} \frac{1+x^2}{1+x^2} \,dx 
        \\
        &\le \biggl(\int_{\mathbb{R}} \E\Bigl[\varphi_{h\eta^2}\bigl(x-Z_0\bigr)\Bigr] \bigl({1+x^2}\bigr)^{\frac{1}{\beta_{\gamma,p}}} \,dx\biggr)^{\beta_{\gamma,p}} \biggl(\int_{\mathbb{R}}\Bigl(\frac{1}{1+x^2}\Bigr)^{\frac1{1-\beta_{\gamma,p}}} \,dx \biggr)^{1-\beta_{\gamma,p}}.
    \end{align*}
    We choose $\gamma = \frac 1{4p}$, ensuring that $\frac14\le \beta_{1/(4p),p} \le \frac34$, meaning that we have bounds that do not depend on $p$.
    Lemma \ref{lem:Tail_estim} guarantees the boundedness of
    \begin{align*}
       \biggl(\int_{\mathbb{R}} &\E\Bigl[\varphi_{h\eta^2}\bigl(x-Z_0\bigr)\Bigr] \bigl({1+x^2}\bigr)^{\frac{1}{\beta_{1/(4p),p}}} \,dx\biggr)^{\frac{\beta_{1/(4p),p}}{p}} \biggl(\int_{\mathbb{R}}\Bigl(\frac{1}{1+x^2}\Bigr)^{\frac1{1-\beta_{1/(4p),p}}} \,dx \biggr)^{\frac{1-\beta_{1/(4p),p}}{p}}
       \\
        &\lesssim \E\Bigl[ \bigl(1+Z_1^2\bigr)^4 \Bigr].
    \end{align*}
    Basic algebra tells us that
    \begin{align*}
        C_{h,\gamma}{h}^{-\frac\gamma{2p}\bigl(\frac{1}{1+\gamma}-\gamma p\bigr)}&=
        C_{h,1/(4p)}{h}^{-\frac1{8p^2}\beta_{1/(4p),p}}
        \\
        &=
        \exp\Bigl(\log(h)\Bigl( 
        -\frac{1+\frac1{4p}+\frac1{16p^2}}2 -\frac1{8p^2}\beta_{1/(4p),p}\Bigr)\Bigr) (4p)^{\frac12}e^{\frac{\eps^2}{ \lambda}}
        \\
        &\lesssim \exp\Bigl(\log(h)\Bigl( 
        -\frac12-\frac1{p}\Bigl(\frac18+\frac1{32p}+\frac1{8p}\times\frac34
        \Bigr)
        \Bigr)\Bigr) \sqrt{p} e^{\frac{\eps^2}{ \lambda}}
        \\
        & \lesssim h^{-1} e^{\frac {\eps^2}{\lambda}}.
    \end{align*}
    Combining everything and substituting back we see that
    \begin{align*}
         \Bigl\| \E\bigl[\xi_{jh}^2| \hX_{jh}^{h,\eps,\delta} \bigr]-\hE^\eps\bigl[\xi_{jh}^2| \hX_{jh}^{h,\eps,\delta} \bigr] \Bigr\|_{L^p}
         \lesssim \eps h^{-1} e^{\frac {\eps^2}{\lambda}}.
    \end{align*}

      \end{proof}

\begin{proof}[Proof of Theorem \ref{thm_condExp_err}]
        Theorem \ref{Thm:Rate_CondExp} tells us that
        \begin{align*}
        \Bigl|\E\bigl[ f(Y_T)\bigr]-\E\bigl[f(\hX_T^{h,\eps,\delta})\bigr]\Bigr|\lesssim h+\sup_{j\in\{0,\dots,\frac Th\}} \bigl\| {\hE_\eps^\delta\bigl[\xi_{jh}^2\big|\hX_{jh}^{h,\eps,\delta}\bigr]}
        -{\E\bigl[\xi_{jh}^2\big|\hX_{jh}^{h,\eps,\delta}\bigr]}\bigr\|_{L^p}.
    \end{align*}
    For $j=0$, by definition ${\hE_\eps^\delta\bigl[\xi_{0}^2\big|\hX_{0}^{h,\eps,\delta}\bigr]}
        ={\E\bigl[\xi_{0}^2\big|\hX_{0}^{h,\eps,\delta}\bigr]}=\E[\xi_0^2]$, hence the difference above is equal to $0$.
    Let $q$ such that $1< q<p$. Basic triangle inequality tells us that
    \begin{align*}
            \bigl\| {\hE^\eps_\delta\bigl[\xi_{jh}^2\big|\hX_{jh}^{h,\eps,\delta}\bigr]}
        -{\E\bigl[\xi_{jh}^2\big|\hX_{jh}^{h,\eps,\delta}\bigr]}\bigr\|_{L^q}
        &\le  
         \bigl\| {\hE^\eps\bigl[\xi_{jh}^2\big|\hX_{jh}^{h,\eps,\delta}\bigr]}
        -{\hE_\delta^\eps\bigl[\xi_{jh}^2\big|\hX_{jh}^{h,\eps,\delta}\bigr]}\bigr\|_{L^q}
        \\
        &\qquad + 
        \bigl\| {\hE^\eps\bigl[\xi_{jh}^2\big|\hX_{jh}^{h,\eps,\delta}\bigr]}
        -{\E\bigl[\xi_{jh}^2\big|\hX_{jh}^{h,\eps,\delta}\bigr]}\bigr\|_{L^q} .
        \end{align*}
        Lemma  \ref{Lem_regular_kernel} tells us, uniformly for all $j \in \{1,\dots,\frac Th\}$
        \begin{align*}
             \bigl\| {\hE^\eps\bigl[\xi_{jh}^2\big|\hX_{jh}^{h,\eps,\delta}\bigr]}
        -{\hE_\delta^\eps\bigl[\xi_{jh}^2\big|\hX_{jh}^{h,\eps,\delta}\bigr]}\bigr\|_{L^q}
        \lesssim
        \Bigl(\delta \sqrt{-\log(\delta)}\Bigr)^{\frac1 q}.
        \end{align*}
        By Lemma  \ref{Lem:FrakCBound} we observe that, again uniformly for all $j \in \{1,\dots,\frac Th\}$
        \begin{align*}
             \bigl\| {\hE^\eps\bigl[\xi_{jh}^2\big|\hX_{jh}^{h,\eps,\delta}\bigr]}
        -{\E\bigl[\xi_{jh}^2\big|\hX_{jh}^{h,\eps,\delta}\bigr]}\bigr\|_{L^q} 
        \lesssim 
        \eps h^{-1} e^{\frac {\eps^2}{\lambda}}
        .
        \end{align*}
        The claim then follows by combining these estimates and from the fact that $\bigl(\delta \sqrt{-\log(\delta)}\bigr)^{\frac1q}\lesssim \delta^{\frac1p}$ as well as the monotonicity of $L^p$ norms.
    \end{proof}
    
\subsection{Propagation of Chaos for Euler Scheme}\label{Sec:POC}
In this section we fix the stepsize $h$, as well as the regularisation parameters $\eps,\delta\in(0,1)$. For $i \in \{1,\dots,N\}$ we will denote the $i$-th particle by the superscript $(i)$. 
Let $\bigl( W^{(i)})_{i=1}^N$ be independent Brownian motions. Let $\Delta W^{(i)}_{jh}=W^{(i)}_{jh+h}-W^{(i)}_{jh}$ and assume that $\frac Th\in \mathbb{N}$. Let $\mu_{jh}=\mathcal{L}\bigl({\xi_{jh},\hX_{jh}^{h,(1),\eps,\delta}}\bigr)$ and recall the notation $\hPsi_j$ from \eqref{Eq:PsiHat_Def}. For $j\in \{0,\dots,\frac Th-1\}$,  we consider independent copies of the following non-interacting particle scheme 
\begin{align*}
    \hX_{jh+h}^{h,(i),\eps,\delta}&=\hX_{jh}^{h,(i),\eps,\delta}+ \sigma(jh,\hX_{jh}^{h,(i),\eps,\delta})\xi_{jh}^{(i)}  \hPsi_j^{-\frac12}\bigl(\hX_{jh}^{h,(i),\eps,\delta},\mu_{jh}\bigr)\Delta W_{jh}^{(i)},
    \\
    \hX_0^{h,(i),\eps,\delta}&=x_0,
    \\
    \mu_{jh}&=\mathcal{L}\bigl({\xi_{jh},\hX_{jh}^{h,(1),\eps,\delta}}\bigr).
\end{align*}
Furthermore we recall the interacting particle scheme
\begin{align}\label{Eq:XhatEpsN}
    \hX_{jh+h}^{h,(i),N,\eps,\delta}&=\hX_{jh}^{h,(i),N,\eps,\delta}+ \sigma(jh,\hX_{jh}^{h,(i),N,\eps,\delta})\xi_{jh}^{(i)} \hPsi_j^{-\frac12}\bigl(\hX_{jh}^{h,(i),N,\eps,\delta},\mu^N_{jh}\bigr)\Delta W_{jh}^{(i)},
    \\
    \hX_0^{h,(i),N,\eps,\delta}&=x_0,\nonumber
\end{align}
where, for $j\ge 1$ we denote by $\mu_{jh}^N$  the empirical measure of the interacting particle system, $\mu_{jh}^N(dx,dy)=\frac1N \sum_{j=1}^N \delta_{\xi_{jh}^{(i)},\hX_{jh}^{h,(i),N,\eps,\delta}}(dy,dx)$ and for consistency we set $\mu_0^N=\mu_0$. 
\begin{remark}\label{Rem_PsiHat}
    Recall that, for $j\ge 1$, by definition the function $\hPsi_j$ depends on $\eps$ and $\delta$ via
    \begin{align*}
        \hPsi_j^{-1}\bigl(y,\mu^N_{jh} \bigr)&=\frac
        {
        \frac1N \sum_{k=1}^N   K_\eps\bigl( y-\hX_{jh}^{h,(k),N,\eps,\delta} \bigr)+\delta
        }
        { {\frac1N \sum_{k=1}^N \bigl(\xi_{jh}^{(k)}\bigr)^2 K_\eps\bigl( y-\hX_{jh}^{h,(k),N,\eps,\delta}\bigr)}+\delta},        
        \\
        \hPsi_j^{-1}\bigl(y,\mu_{jh}\bigr)&= \frac{
        \E\bigl[ K_\eps\bigl(y-\hX_{jh}^{h,(1),\eps,\delta}\bigr)\bigr]+\delta
        }{{\E\bigl[ \bigl(\xi_{jh}^{(1)}\bigr)^2K_\eps\bigl(y-\hX_{jh}^{h,(1),\eps,\delta}\bigr)\bigr]}+\delta }.
    \end{align*}
    Also, by convention
    \begin{align*}
        \hPsi_0^{-1}\bigl(y,\mu_{0} \bigr)=\hPsi_0^{-1}\bigl(y,\mu^N_{0} \bigr)=\frac1{\E[\xi_0^2]}.
    \end{align*}
    In particular the function $\hPsi_j$ satisfies $\hPsi_j( \hX_{jh}^{h,(i),N,\eps,\delta},\mu_{jh}^N)=\hE_\delta^{\eps,N}[(\xi_{jh}^{(i)})^2| \hX_{jh}^{h,(i),N,\eps,\delta}]$ for $\hE_\delta^{\eps,N}$ defined in  Equation \eqref{eq_hEepsdelta} above.
\end{remark}
\begin{theorem}\label{Lem:Prop_of_chaos}
    Under Assumption \ref{Aspt_new} and \ref{Ass:kernel} the following hold:
    \begin{enumerate}[(i)]
        \item   There is a constant $C>0$ such that
        \begin{align*}
        \Bigl\| \hX_T^{h,(1),\eps,\delta} - \hX_T^{h,(1),N,\eps,\delta} \Bigr\|_{L^2} \lesssim \eps
        \exp\Bigl(C \delta^{-2} \eps^{-4}\Bigr)  N^{-1/2}.
    \end{align*}
    \item  If $\eps\lesssim h$ there is a constant $\tilde C$ such that
    \begin{align*}
         \Bigl\| \hX_T^{h,(1),\eps,\delta} - \hX_T^{h,(1),N,\eps,\delta} \Bigr\|_{L^2} 
        &\lesssim \eps
        \exp\Bigl(\frac T2h^{-1}
        \log\bigl( \tilde C h \delta^{-2}\eps^{-4}\bigr)\Bigr)N^{-1/2}.
    \end{align*}
    \end{enumerate}

\end{theorem}

\begin{proof}[Proof of Theorem \ref{Thm_Particle_Rate2}]
    We distribute the error
    \begin{align*}
        \Bigl|\E\bigl[ f(Y_T)\bigr]-\E\bigl[f\bigl(\hX_T^{h,(1),N,\eps,\delta}\bigr)\bigr]\Bigr| 
        &\le 
        \Bigl|\E\bigl[ f(Y_T)\bigr]-\E\bigl[f\bigl(\hX_T^{h,\eps,\delta}\bigr)\bigr]\Bigr|
        \\&+
        \Bigl|\E\bigl[ f(\hX_T^{h,\eps,\delta})\bigr]-\E\bigl[f\bigl(\hX_T^{h,(1),N,\eps,\delta}\bigr)\bigr]\Bigr|.
    \end{align*}
    There now exists for all $p>1$ a constant $C>0$ such that
    \begin{align*}
        \Bigl|\E\bigl[ f(Y_T)\bigr]-\E\bigl[f\bigl(\hX_T^{h,\eps,\delta}\bigr)\bigr]\Bigr|
        &\lesssim  h +  \eps h^{-1} e^{\frac {\eps^2}{\lambda}}
        + \delta^{\frac1p} &
        \text{ due to Theorem \ref{thm_condExp_err},}
          \\
        \Bigl|\E\bigl[ f(\hX_T^{h,\eps,\delta})\bigr]-\E\bigl[f\bigl(\hX_T^{h,(1),N,\eps,\delta}\bigr)\bigr]\Bigr| &\lesssim  \eps
        \exp\Bigl(
        \frac T{2}h^{-1} \log\bigl( C h \delta^{-2}\eps^{-4}\bigr)\Bigr)N^{-1/2} \hspace{-5pt}&
        \text{ due to Theorem \ref{Lem:Prop_of_chaos}.}
    \end{align*} 
    For the second inequality we also used the fact that $f$ is locally Lipschitz-continuous with derivative of polynomial growth.
    Recalling $\lambda=h\rhobar^2 c_{min}^2$ and using the assumption that  $\eps\lesssim h$ we observe that $e^{\eps^2/\lambda}\lesssim 1$.
    Summing up these terms and choosing $p$ arbitrarily close to $1$ completes the proof.
\end{proof}

\begin{remark}\label{Rem:Rate}
    Theorem \ref{Thm_Particle_Rate2} above provides a manual how to choose the parameters in order to achieve any desired error.
    Let $\gamma>0$ and suppose we want an error of order $\gamma$, i.e. 
    \begin{align*}
        \Bigl|\E\bigl[ f(Y_T)\bigr]-\E\bigl[f\bigl(\hX_T^{h,(i),N,\eps,\delta}\bigr)\bigr]\Bigr|
        \lesssim 
        \gamma.
    \end{align*}
    One would choose stepsize $h \approx \gamma$, kernel bandwidth $\eps \approx \gamma^{2}$, regularisation $\delta \approx \gamma$ and particle size $N \approx \gamma^{2}\exp\bigl(T\gamma^{-1} \log(C \gamma^{-9} )\bigr)$. 
\end{remark}

\begin{remark}\label{Remark_Correlation}
Let us briefly sketch how to derive a weak error rate in the style of Theorem \ref{Thm_Particle_Rate2} without Assumption \ref{Ass:Half_Step}, under the (admittedly hard to verify) assumption that the function
$x\mapsto \E[ \xi_{jh}| \hX_{jh}^{h,\eps,\delta}=x]$ is $\alpha$-Hölder continuous uniformly for all $j \in \{1,\dots, \frac Th\}$ as well as $\eps, \delta \in (0,1/2)$. In this case
    \begin{align*}
        \Bigl|\E\bigl[ f(Y_T)\bigr]-\E\bigl[f\bigl(\hX_T^{h,(1),N,\eps,\delta}\bigr)\bigr]\Bigr|
        &\lesssim 
        h +  \eps^\alpha  + \delta^{\frac1p}  
        + \frac1{\sqrt{N}}\eps^{-1}
        \exp\Bigl(\frac T2
        h^{-1} \log\bigl(  C h \delta^{-2}\eps^{-4}\bigr)\Bigr).
    \end{align*}
Indeed, using H\"older regularity, it is not hard to
show (essentially using compact support of $K_\eps$) that
    \begin{align*}
         \Bigl\| \E\bigl[\xi_{jh}^2| \hX_{jh}^{h,\eps,\delta} \bigr]-\hE^\eps\bigl[\xi_{jh}^2| \hX_{jh}^{h,\eps,\delta} \bigr] \Bigr\|_{L^\infty} \lesssim 
        \eps^\alpha,
    \end{align*}
    which should be compared with Lemma \ref{Lem:FrakCBound} (which also provides an estimate for the left-hand side above). One can then copy the proof of Theorem \ref{Thm_Particle_Rate2} by replacing  the right-hand side provided in the said lemma with $\eps^\alpha$ above.
    \end{remark}

\begin{lemma}\label{Lem:G_deriv}
    Let $\eps,\delta \in (0,1)$ and $\mu\in \mathcal{P}_2(\mathbb{R}\times\mathbb{R})$  induced by some random variables $(\xi,X)$ such that $\xi$ is bounded from above and below, similar to Assumption \ref{Aspt_new}.  Let $K_\eps$ satisfy Assumption \ref{Ass:kernel}, $j \ge 0$ and $\hPsi_j$ with its dependence on $\eps$ and $\delta$ be as in Equation \eqref{Eq:PsiHat_Def}. Then the map $y \mapsto \hPsi_j^{-\frac12}(y,\mu)$ is Lipschitz with constant
    \begin{align*}
        \|\hPsi^{-\frac12}\|_{Lip}\le \frac{1}{\delta} \|K_\eps'\|_{\infty}.
    \end{align*}
\end{lemma}
\begin{proof}
    Recall that that $\xi$ is bounded and bounded away from $0$. For $j=0$ the function $\hPsi_0$ does not depend on $y$, hence the statement is correct.

    Assume now $j \ge 1$.
      Chain rule implies that
    \begin{align*}
        \frac{\partial}{\partial y} \hPsi_j^{-\frac12}(y,\mu)&=\frac12\sqrt{ \frac{{\E\bigl[ \xi^2K_\eps\bigl(y-X\bigr)\bigr]}+\delta }{
        \E\bigl[ K_\eps\bigl(y-X\bigr)\bigr]+\delta
        }
        }
        \Biggl(
        \frac{
        \E\bigl[ K_\eps'\bigl(y-X\bigr)\bigr] 
        \Bigl(
        \E\bigl[ \xi^2K_\eps\bigl(y-X\bigr)\bigr]+\delta
        \Bigr)      }
        {
        \Bigl({{\E\bigl[ \xi^2K_\eps\bigl(y-X\bigr)\bigr]}+\delta }\Bigr)^2}
        \\
        & \qquad \qquad \qquad 
        - 
        \frac{
        \Bigl(\E\bigl[ K_\eps\bigl(y-X\bigr)\bigr] +\delta\Bigr)
        \E\bigl[ \xi^2K_\eps'\bigl(y-X\bigr)\bigr]
        }
        {
        \Bigl({{\E\bigl[ \xi^2K_\eps\bigl(y-X\bigr)\bigr]}+\delta }\Bigr)^2}
        \Biggr)
        \\
        &\lesssim \frac{\|K_\eps\|_{C_1}}{\delta}.
    \end{align*}
    For the last inequality we used that $\xi$ is bounded from above by assumption.
\end{proof}
Being provided with Lipschitz continuity of $\hPsi$ we are now able to show propagation of chaos.

\begin{proof}[Proof of Theorem \ref{Lem:Prop_of_chaos}]

For the rest of this proof we simplify the notation by writing $X_{jh}^{(i)}=\hX_{jh}^{h,(i),\eps,\delta}$ and $X_{jh}^{(i),N}=\hX_{jh}^{h,(i),N,\eps,\delta}$.
First note that by definition $X_0^{(i)}=X_0^{(i),N}$ and by Remark \ref{Rem_PsiHat} as well as the definition of the Euler schemes $X_h^{(i)}=X_h^{(i),N}$ for all $i \in \{1,\dots,N\}$.

    We now fix some index $i \in \{1,\dots,N\}$ and some index $j\in \{1,\dots,\frac Th-1\}$ and study the difference $X_{jh+h}^{(i)}-X_{jh+h}^{(i),N}$. Then
    \begin{align*}
        X_{jh+h}^{(i)}-X_{jh+h}^{(i),N}&=X_{jh}^{(i)}-X_{jh}^{(i),N} \\
        &\qquad + \Bigl(\sigma(jh,X_{jh}^{(i)})-\sigma(jh,X_{jh}^{(i),N})\Bigr)\xi_{jh}^i\hPsi_j^{-\frac12}\bigl(X_{jh}^{(i),N},\mu^N_{jh}
        \bigr)\,\Delta W_{jh}^{(i)} &\eqqcolon (1)
        \\ 
        &\qquad + \sigma(jh,X_{jh}^{(i)})\xi_{jh}^{(i)} \Bigl( \hPsi_j^{-\frac12}\bigl(X_{jh}^{(i)},\mu^N_{jh} \bigr) - \hPsi_j^{-\frac12}\bigl(X_{jh}^{(i),N},\mu^N_{jh}\bigr) \Bigr) \Delta W_{jh}^{(i)} &\eqqcolon (2)
        \\ 
        &\qquad + \sigma(jh,X_{jh}^{(i)})\xi_{jh}^{(i)} \Bigl(
        \hPsi_j^{-\frac12}\bigl(X_{jh}^{(i)},\mu_{jh}\bigr)
        -
        \hPsi_j^{-\frac12}\bigl(X_{jh}^{(i)},\mu^N_{jh} \bigr) \Bigr) \Delta W_{jh}^{(i)}
        . &\eqqcolon (3)
    \end{align*}   
    By Lemma \ref{Lem:G_deriv}, we get the estimate
    \begin{gather}
    \begin{aligned}\label{eq:chaosprop_1}
        \E\Bigl[(1)^2+(2)^2 \Bigr]&\lesssim \E\Bigl[(\|\partial_x\sigma\|^2_{L^\infty}+\|\hPsi_j^{-\frac12}\|_{Lip}^2) \Bigl(\Delta W_{jh}^{(i)} \Bigr)^2 \Bigl( X_{jh}^{(i)}-X_{jh}^{(i),N}\Bigr)^2\Bigr]
        \\
        &\lesssim \frac1{\delta^2} \|K_\eps\|_{C_1}^2h \E\Bigl[ \Bigl( X_{jh}^{(i)}-X_{jh}^{(i),N}\Bigr)^2\Bigr].
    \end{aligned}
    \end{gather}
    We now want to estimate the term (3). 
    We start by rewriting
    \begin{align*}
        \hPsi_j^{-\frac12}\bigl(&X_{jh}^{(i)},\mu_{jh}\bigr)
        -
        \hPsi_j^{-\frac12}\bigl(X_{jh}^{(i)},\mu^N_{jh} \bigr) 
        \\
        &=
        \frac{
        \sqrt{
        \E\bigl[ K_\eps\bigl(y-X_{jh}^{(1)}\bigr)\bigr]+\delta
        }
        }{
        \sqrt{\E\bigl[ \bigl(\xi_{jh}^{(1)}\bigr)^2K_\eps\bigl(y-X_{jh}^{(1)}\bigr)\bigr]+\delta} }\Biggl|_{y=X_{jh}^{(i)}}
        - \frac
        {
        \sqrt{\frac1N \sum_{k=1}^N   K_\eps\bigl( X_{jh}^{(i)}-X_{jh}^{(k)} \bigr)+\delta
        }
        }
        { \sqrt{\frac1N \sum_{k=1}^N \bigl(\xi_{jh}^{(k)}\bigr)^2 K_\eps\bigl( X_{jh}^{(i)}-X_{jh}^{(k)}\bigr)+\delta}}
        \\
        &\quad +
        \frac
        {
        \sqrt{\frac1N \sum_{k=1}^N   K_\eps\bigl( X_{jh}^{(i)}-X_{jh}^{(k)} \bigr)+\delta}
        }
        { \sqrt{\frac1N \sum_{k=1}^N \bigl(\xi_{jh}^{(k)}\bigr)^2 K_\eps\bigl( X_{jh}^{(i)}-X_{jh}^{(k)}\bigr)+\delta}}
        -
        \frac{
        \sqrt{\frac1N \sum_{k=1}^N   K_\eps\bigl( X_{jh}^{(i)}-X_{jh}^{(k),N} \bigr)+\delta}
        }
        { \sqrt{\frac1N \sum_{k=1}^N \bigl(\xi_{jh}^{(k)}\bigr)^2 K_\eps\bigl( X_{jh}^{(i)}-X_{jh}^{(k),N}\bigr)+\delta}}
        \\
        &\eqqcolon \frac{\sqrt{\mathcal{A}_0^i+\delta}}{\sqrt{\mathcal{B}_0^i+\delta}}-\frac{\sqrt{\mathcal{A}^i_1+\delta}}{\sqrt{{\mathcal{B}^i_1}+\delta}}+\frac{\sqrt{\mathcal{A}^i_1+\delta}}{\sqrt{\mathcal{B}^i_1+\delta}}-\frac{\sqrt{\mathcal{A}^i_2+\delta}}{\sqrt{\mathcal{B}^i_2+\delta}}
    \end{align*}
    We can apply Lemma \ref{lem_sqt_estim} and see that
    \begin{align*}
        \E\Bigl[ \Bigl( \hPsi_j^{-\frac12}\bigl(X_{jh}^{(i)},\mu_{jh}\bigr)
        -
        \hPsi_j^{-\frac12}\bigl(X_{jh}^{(i)},\mu^N_{jh} \bigr) \Bigr)^2\Bigr] \lesssim \frac{1}{\delta^2}& \Bigl( 
        \E\bigl[ \bigl(\mathcal{A}_0^i-\mathcal{A}_1^i \bigr)^2\bigr]+\E\bigl[ \bigl(\mathcal{B}_0^i-\mathcal{B}_1^i \bigr)^2\bigr]
        \\
        &+\E\bigl[ \bigl(\mathcal{A}_1^i-\mathcal{A}_2^i \bigr)^2\bigr]+\E\bigl[ \bigl(\mathcal{B}_1^i-\mathcal{B}_2^i \bigr)^2\bigr]
        \Bigr).
    \end{align*}
    Note that by exchangeability $\E\bigl[\bigl(\xi_{jh}^{(i)}\bigr)^2K_\eps\bigl(y-X_{jh}^{(i)}\bigr)  \bigr]= \frac1N\sum_{k=1}^N \E\bigl[\bigl(\xi_{jh}^{(k)}\bigr)^2 K_\eps\bigl(y-X_{jh}^{(k)}\bigr)  \bigr]$. We now consider the term $ \E[ (\mathcal{B}_0^i-\mathcal{B}_1^i )^2]$, exclude the $i$-th element which is of order $\|K_\eps\|_{L^\infty}^2N^{-2}$ in the sums that appear and use independence conditional on $X_{jh}^{(i)}$ to see that
    \begin{align*}
        \E\bigl[& \bigl(\mathcal{B}_0^i-\mathcal{B}_1^i \bigr)^2\bigr]=\E\Bigl[ \Bigl(\E\Bigl[ \frac1N\sum_{k=1}^N \bigl(\xi_{jh}^{(k)}\bigr)^2 K_\eps\bigl(y-X_{jh}^{(k)}\bigr)  \Bigr]\Big|_{y=X_{jh}^{(i)}}-\frac1N\sum_{k=1}^N \bigl(\xi_{jh}^{(k)}\bigr)^2 K_\eps\bigl(X_{jh}^{(i)}-X_{jh}^{(k)}\bigr) 
        \Bigr)^2\Bigr]
        \\
        &\lesssim
        \frac{\|K_\eps\|_{L^\infty}^2}{N^2}+ \E\Bigl[ \Bigl(\E\Bigl[ \frac1N\sum_{\substack{k=1 \\ k \not = i}
        }^N \bigl(\xi_{jh}^{(k)}\bigr)^2 K_\eps\bigl(y-X_{jh}^{(k)}\bigr)  \Bigr]\bigg|_{y=X_{jh}^{(i)}}-\frac1N\sum_{\substack{k=1 \\ k \not = i}}^N \bigl(\xi_{jh}^{(k)}\bigr)^2 K_\eps\bigl(X_{jh}^{(i)}-X_{jh}^{(k)}\bigr) 
        \Bigr)^2\Bigr]
        \\
        &\lesssim \frac {\|K_\eps\|_{L^\infty}^2} {N^2}+ \frac1N\E\Bigl[    \frac1N \sum_{\substack{k=1\\ k \not = i}
        }^N \E\Bigl[  \Bigl( \bigl(\xi_{jh}^{(k)}\bigr)^2 K_\eps\bigl(y-X_{jh}^{(k)}\bigr) -\E\bigl[\bigl(\xi_{jh}^{(k)}\bigr)^2 K_\eps\bigl(y-X_{jh}^{(k)}\bigr) \bigr]\Bigr)^2   \Bigr]\bigg|_{y=X_{jh}^{(i)}}      \Bigr]
        \\
        &\lesssim \frac{\|K_\eps\|_{L^\infty}^2}N.
    \end{align*}
    With similar arguments it follows that $\E\bigl[ \bigl(\mathcal{A}_0^i-\mathcal{A}_1^i \bigr)^2\bigr]\lesssim \frac1N \|K_\eps\|_{L^\infty}^2$.
    The function $K_\eps$ has bounded derivative, therefore
    \begin{align*}
        \bigl(\mathcal{B}_1^i-\mathcal{B}_2^i \bigr)^2&=\biggl(\frac1N\sum_{k=1}^N  \Bigl(\bigl(\xi_{jh}^{(k)}\bigr)^2 K_\eps\bigl( X_{jh}^{(i)}-X_{jh}^{(k)}\bigr)-\bigl(\xi_{jh}^{(k)}\bigr)^2 K_\eps\bigl( X_{jh}^{(i)}-X_{jh}^{(k),N}\bigr)\Bigr)
        \biggr)^2 
        \\
        &\lesssim \biggl(\frac1N\sum_{k=1}^N\|K_\eps\|_{C_1}\Bigl|X_{jh}^{(k),N}-X_{jh}^{(k)}\Bigr|\biggr)^2
        \\
        &\lesssim \frac1N\sum_{k=1}^N\|K_\eps\|_{C_1}^2\Bigl(X_{jh}^{(k),N}-X_{jh}^{(k)}\Bigr)^2.
    \end{align*}
    For the last inequality we used Jensen's inequality. Similarly it follows that $\bigl(\mathcal{A}_1^i-\mathcal{A}_2^i \bigr)^2\lesssim \frac1N\sum_{k=1}^N\|K_\eps\|_{C_1}^2\Bigl(X_{jh}^{(k),N}-X_{jh}^{(k)}\Bigr)^2$.

    Putting these estimates together we see by independence of the Gaussian increment and boundedness of $\sigma$ that
    \begin{gather}
    \begin{aligned}
        \E\Bigl[ \Bigl(\sigma(jh,X_{jh}^{(i)})\xi_{jh}^{(i)} \Bigl(
        \hPsi_j^{-\frac12}\bigl(X_{jh}^{(i)},\mu_{jh}\bigr)
        &-
        \hPsi_j^{-\frac12}\bigl(X_{jh}^{(i)},\mu^N_{jh} \bigr) \Bigr) \Delta W_{jh}^{(i)}\Bigr)^2\Bigr]
        \\
        &\lesssim
        \E\Bigl[ \Bigl(
        \hPsi_j^{-\frac12}\bigl(X_{jh}^{(i)},\mu_{jh}\bigr)
        -
        \hPsi_j^{-\frac12}\bigl(X_{jh}^{(i)},\mu^N_{jh} \bigr) \Bigr)^2\Bigr] h
        \\
        &\lesssim \frac1N h\delta^{-2}\|K_\eps\|_{L^\infty}^2 + 
        h\delta^{-2}\|K_\eps\|_{C_1}^2 \E\Bigl[\frac1N\sum_{k=1}^N \Bigl(X_{jh}^{(k),N}-X_{jh}^{(k)}\Bigr)^2 \Bigr]\label{eq:chaosprop_2}
        \end{aligned}
    \end{gather}

    Combining \eqref{eq:chaosprop_1} and \eqref{eq:chaosprop_2} we see that there is a constant $C>0$ such that
    \begin{align*}
        \E\Bigl[ \Bigl( X_{jh+h}^{(i)}-X_{jh+h}^{(i),N}  \Bigr)^2   \Bigr]\le \Bigl(1 + &Ch\delta^{-2}\|K_\eps\|_{C_1}^2 \Bigr)\E\Bigl[ \Bigl( X_{jh}^{(i)}-X_{jh}^{(i),N}\Bigr)^2\Bigr] 
         \\
         &+ Ch\delta^{-2}\|K_\eps\|_{C_1}^2
         \E\Bigl[\frac1N\sum_{k=1}^N \Bigl(X_{jh}^{(k),N}-X_{jh}^{(k)}\Bigr)^2 \Bigr]
         +C \frac1Nh\delta^{-2}\|K_\eps\|_{L^\infty}^2.
    \end{align*}
    Summing up over $i$ we conclude that
    \begin{align*}
       \frac1N \sum_{i=1}^N \E\Bigl[ \Bigl( X_{jh+h}^{(i)}-X_{jh+h}^{(i),N}  \Bigr)^2   \Bigr]
       &\le 
       \Bigl(1 + 2Ch\delta^{-2}\|K_\eps\|_{C_1}^2 \Bigr) \E\Bigl[  \frac1N\sum_{i =1}^N \Bigl( X_{jh}^{(i)}-X_{jh}^{(i),N}  \Bigr)^2   \Bigr] 
       \\
       &\qquad+C \frac1Nh\delta^{-2}\|K_\eps\|_{L^\infty}^2.
    \end{align*}
    By Lemma \ref{lem:recurr_relation}  it follows that
    \begin{align*}
         \frac1N\sum_{i =1}^N \E\Bigl[ \Bigl( X_{jh+h}^{(i)}-X_{jh+h}^{(i),N}  \Bigr)^2   \Bigr]&\lesssim \frac1Nh\delta^{-2}\|K_\eps\|_{L^\infty}^2 
         \frac{\exp\Bigl(C {\delta^{-2}} \|K_\eps\|_{C_1}^2\Bigr)}
         {2Ch\delta^{-2}\|K_\eps\|_{C_1}^2}
         \\
         &\lesssim \frac1N
         \frac{\|K_\eps\|_{L^\infty}^2}
         {\|K_\eps\|_{C_1}^2}\exp\Bigl(C {\delta^{-2}} \|K_\eps\|_{C_1}^2\Bigr).
    \end{align*}
    The scaling property from Assumption \ref{Ass:kernel} implies that $\|K_\eps\|_{C_1}\lesssim \eps^{-2}$ as well as $ \frac{\|K_\eps\|_{L^\infty}^2}{\|K_\eps\|_{C_1}^2} \lesssim \eps^{2}$,  therefore the first part of the theorem follows.
    Lemma \ref{lem:recurr_relation} also implies that
    \begin{align*}
        \frac1N\sum_{i =1}^N \E\Bigl[ \Bigl( X_{jh+h}^{(i)}-X_{jh+h}^{(i),N}  \Bigr)^2   \Bigr] 
        &\lesssim \frac1Nh\delta^{-2}\|K_\eps\|_{L^\infty}^2 \sum_{j=1}^{T/h}  
        \Bigl(1+2Ch\delta^{-2}\|K_\eps\|_{C_1}^2\Bigr)^{j-1}
        \\
        &\lesssim  \frac1Nh\delta^{-2}\|K_\eps\|_{L^\infty}^2
        \frac{\Bigl(1+2Ch\delta^{-2}\|K_\eps\|_{C_1}^2\Bigr)^{T/h}}{2Ch\delta^{-2}\|K_\eps\|_{C_1}^2}
        \\
        &\lesssim \frac1N \frac{\|K_\eps\|_{L^\infty}^2}{\|K_\eps\|_{C_1}^2} {\Bigl(1+2Ch\delta^{-2}\|K_\eps\|_{C_1}^2\Bigr)^{T/h}}.
    \end{align*}
    Assuming that $\eps\le h$ it follows from changing the constant $C$ that
    \begin{align*}
         \frac1N\sum_{i =1}^N \E\Bigl[ \Bigl( X_{jh+h}^{(i)}-X_{jh+h}^{(i),N}  \Bigr)^2   \Bigr] 
        &\lesssim \frac1{N} \eps^{2}
        \exp\Bigl(
        Th^{-1} \log\bigl( 2C h \delta^{-2}\eps^{-4}\bigr)\Bigr).
    \end{align*}
\end{proof}

\section{Bibliography}
\bibliographystyle{abbrvurl}
\bibliography{Bibliography}

\begin{thebibliography}{10}

\bibitem{AbergelTachet10}
F.~Abergel and R.~Tachet.
\newblock {A nonlinear partial integro-differential equation from mathematical finance}.
\newblock {\em {Discrete and Continuous Dynamical Systems - Series A}}, 27(3):907--917, 2010.
\newblock \href {https://doi.org/10.3934/dcds.2010.27.907} {\path{doi:10.3934/dcds.2010.27.907}}.

\bibitem{BayerEtal2022}
C.~Bayer, D.~Belomestny, O.~Butkovsky, and J.~Schoenmakers.
\newblock A reproducing kernel {Hilbert} space approach to singular local stochastic volatility {McKean–Vlasov} models, 2024.
\newblock \href {https://doi.org/10.1007/s00780-024-00541-5} {\path{doi:10.1007/s00780-024-00541-5}}.

\bibitem{BeiglböckLowther2021}
M.~Beiglb{\"o}ck, G.~Lowther, G.~Pammer, and W.~Schachermayer.
\newblock Faking {Brownian} motion with continuous {Markov} martingales.
\newblock {\em Finance and Stochastics}, 28(1):259--284, 2024.
\newblock \href {https://doi.org/10.1007/s00780-023-00526-w} {\path{doi:10.1007/s00780-023-00526-w}}.

\bibitem{BentataCont2012}
A.~Bentata and R.~Cont.
\newblock Mimicking the marginal distributions of a semimartingale, 2012.
\newblock {arXiv} preprint.
\newblock \href {https://arxiv.org/abs/0910.3992} {\path{arXiv:0910.3992}}.

\bibitem{BouchardTouzi04}
B.~Bouchard and N.~Touzi.
\newblock Discrete-time approximation and {Monte-Carlo} simulation of backward stochastic differential equations.
\newblock {\em Stochastic Processes and their Applications}, 111(2):175--206, 2004.
\newblock \href {https://doi.org/10.1016/j.spa.2004.01.001} {\path{doi:10.1016/j.spa.2004.01.001}}.

\bibitem{BrunickShreve2013}
G.~Brunick and S.~Shreve.
\newblock {Mimicking an Itô process by a solution of a stochastic differential equation}.
\newblock {\em The Annals of Applied Probability}, 23(4):1584 – 1628, 2013.
\newblock \href {https://doi.org/10.1214/12-AAP881} {\path{doi:10.1214/12-AAP881}}.

\bibitem{Djete22}
M.~F. Djete.
\newblock Non--regular {McKean}--{Vlasov} equations and calibration problem in local stochastic volatility models, 2022.
\newblock {arXiv} preprint.
\newblock \href {https://arxiv.org/abs/2208.09986} {\path{arXiv:2208.09986}}.

\bibitem{Dupire1994}
B.~Dupire.
\newblock Pricing with a smile.
\newblock {\em Risk}, 7:18--20, 1994.

\bibitem{Dupire1997}
B.~Dupire.
\newblock Pricing and hedging with smiles.
\newblock In M.~A. Dempster and S.~R. Pliska, editors, {\em Mathematics of Derivative Securities}, pages 103--111. Cambridge University Press, 1997.

\bibitem{FordeJacquier11}
M.~Forde and A.~Jacquier.
\newblock Small-time asymptotics for an uncorrelated local-stochastic volatility model.
\newblock {\em Appl. Math. Finance}, 18(6):517--535, 2011.
\newblock \href {https://doi.org/10.1080/1350486X.2011.591159} {\path{doi:10.1080/1350486X.2011.591159}}.

\bibitem{Friedman1964}
A.~Friedman.
\newblock {\em Partial differential equations of parabolic type}.
\newblock Prentice-Hall, Inc., Englewood Cliffs, NJ, 1964.

\bibitem{GuyonHL2013}
J.~Guyon and P.~Henry-Labordere.
\newblock {\em Nonlinear Option Pricing}.
\newblock Chapman and Hall/CRC, 1st edition, 2013.
\newblock \href {https://doi.org/10.1201/b16332} {\path{doi:10.1201/b16332}}.

\bibitem{GuyonHL12}
J.~Guyon and P.~Henry-Labordère.
\newblock Being particular about calibration.
\newblock {\em Risk}, 25(1):88--93, 01 2012.
\newblock URL: \url{https://www.risk.net/markets/2135540/being-particular-about-calibration}.

\bibitem{GuyonLekeufack2023}
J.~Guyon and J.~Lekeufack.
\newblock Volatility is (mostly) path-dependent.
\newblock {\em Quant. Finance}, 23(9):1221--1258, 2023.
\newblock \href {https://doi.org/10.1080/14697688.2023.2221281} {\path{doi:10.1080/14697688.2023.2221281}}.

\bibitem{Gyongy1986}
I.~Gy{\"o}ngy.
\newblock Mimicking the one-dimensional marginal distributions of processes having an {It\^o} differential.
\newblock {\em Probability Theory and Related Fields}, 71:501--516, 1986.
\newblock \href {https://doi.org/10.1007/BF00699039} {\path{doi:10.1007/BF00699039}}.

\bibitem{HamzaKlebaner2007}
K.~Hamza and F.~C. Klebaner.
\newblock A family of non-{Gaussian} martingales with {Gaussian} marginals.
\newblock {\em Journal of Applied Mathematics and Stochastic Analysis}, 2007:19 p., 2007.
\newblock \href {https://doi.org/10.1155/2007/92723} {\path{doi:10.1155/2007/92723}}.

\bibitem{HirschProfetaEtAl2011}
F.~Hirsch, C.~Profeta, B.~Roynette, and M.~Yor.
\newblock {\em Peacocks and Associated Martingales, with Explicit Constructions}.
\newblock Bocconi \& Springer Series. Springer Milano, 1 edition, 2011.
\newblock \href {https://doi.org/10.1007/978-88-470-1908-9} {\path{doi:10.1007/978-88-470-1908-9}}.

\bibitem{JourdainZhou2020}
B.~Jourdain and A.~Zhou.
\newblock Existence of a calibrated regime switching local volatility model.
\newblock {\em Math. Finance}, 30(2):501--546, 2020.
\newblock \href {https://doi.org/10.1111/mafi.12231} {\path{doi:10.1111/mafi.12231}}.

\bibitem{Kellerer1972}
H.~G. Kellerer.
\newblock Markov-{Komposition} und eine {Anwendung} auf {Martingale}.
\newblock {\em Mathematische Annalen}, 198:99--122, 1972.
\newblock \href {https://doi.org/10.1007/BF01432281} {\path{doi:10.1007/BF01432281}}.

\bibitem{KohatsuShigeyoshi1997}
A.~Kohatsu-Higa and S.~Ogawa.
\newblock Weak rate of convergence for an {E}uler scheme of nonlinear {SDE}'s.
\newblock {\em Monte Carlo Methods Appl.}, 3(4):327--345, 1997.
\newblock \href {https://doi.org/10.1515/mcma.1997.3.4.327} {\path{doi:10.1515/mcma.1997.3.4.327}}.

\bibitem{KonakovMammen2002}
V.~Konakov and E.~Mammen.
\newblock Edgeworth type expansions for {E}uler schemes for stochastic differential equations.
\newblock {\em Monte Carlo Methods Appl.}, 8(3):271--285, 2002.
\newblock \href {https://doi.org/10.1515/mcma.2002.8.3.271} {\path{doi:10.1515/mcma.2002.8.3.271}}.

\bibitem{LackerShkolnikovZhang2019}
D.~Lacker, M.~Shkolnikov, and J.~Zhang.
\newblock {Inverting the Markovian projection, with an application to local stochastic volatility models}.
\newblock {\em The Annals of Probability}, 48(5):2189 – 2211, 2020.
\newblock \href {https://doi.org/10.1214/19-AOP1420} {\path{doi:10.1214/19-AOP1420}}.

\bibitem{LiScott07}
Q.~Li and J.~S. Racine.
\newblock {\em Nonparametric Econometrics: Theory and Practice}.
\newblock Princeton University Press, Princeton, NJ, 2007.

\bibitem{Muguruza2019}
A.~Muguruza.
\newblock Not so particular about calibration: Smile problem resolved, 2019.
\newblock {arXiv} preprint.
\newblock \href {https://arxiv.org/abs/1909.13366} {\path{arXiv:1909.13366}}.

\bibitem{ReisingerTsianni2023}
C.~Reisinger and M.~O. Tsianni.
\newblock Convergence of the {Euler–Maruyama} particle scheme for a regularised {McKean–Vlasov} equation arising from the calibration of local-stochastic volatility models, 2024.
\newblock \href {https://doi.org/10.1007/978-3-031-59762-6_28} {\path{doi:10.1007/978-3-031-59762-6_28}}.

\bibitem{RevuzYor99}
D.~Revuz and M.~Yor.
\newblock {\em Continuous Martingales and Brownian Motion}, volume 293 of {\em Grundlehren der mathematischen Wissenschaften}.
\newblock Springer-Verlag, Berlin, third edition, 1999.
\newblock \href {https://doi.org/10.1007/978-3-662-06400-9} {\path{doi:10.1007/978-3-662-06400-9}}.

\bibitem{TalayTubaro1990}
D.~Talay and L.~Tubaro.
\newblock Expansion of the global error for numerical schemes solving stochastic differential equations.
\newblock {\em Stochastic Anal. Appl.}, 8(4):483--509, 1990.
\newblock \href {https://doi.org/10.1080/07362999008809220} {\path{doi:10.1080/07362999008809220}}.

\bibitem{Zhou2018}
A.~Zhou.
\newblock {\em {Theoretical and numerical study of problems nonlinear in the sense of McKean in finance}}.
\newblock Phd thesis, {Universit{\'e} Paris Est, {\'E}cole des Ponts Paris Tech}, Oct. 2018.
\newblock URL: \url{https://inria.hal.science/tel-01957638}.

\end{thebibliography}

\pagebreak
\section{Appendix}
\label{Sec:Appendix}

\subsection[Proof of Lemma 2.5]{Proof of Lemma \ref{Lem_UEstim}}

The following lemma is taken from \cite[Theorem 7, Chapter 9, p. 260]{Friedman1964} and also appears in a more specific setting in \cite[Theorem 2.2]{KonakovMammen2002}.
\begin{lemma}\label{Lem_DensityRegularity}
    Let $t\in[0,T)$, $Y^{t,x}$ be the solution of \eqref{Eq:target} started at time $t$ with $Y_t^{t,x}=x$. Assume that $\sigma$ is four times continuously differentiable with bounded derivatives and uniformly elliptic. Then the law of $Y_T^{t,x}$ admits a density $p(y;t,x)$ with respect to the Lebesgue measure. Furthermore, if $0\le n+m \le 4$ there are constants $0<C<\infty$ such that
    \begin{align*}
       \frac{d^m}{dy^m}\frac{d^n}{dx^n} p(y+x;t,x) \le C \frac{1}{(T-t)^{(m+1)/2}} \exp\Bigl( - \frac{y^2}{C(T-t)}\Bigr).
    \end{align*}
\end{lemma}
\begin{remark}
    In the notation of \cite{Friedman1964} we have $p(y;t,x)=\Gamma^*(x,T;y,t)=\Gamma(y,t;x,T)$.
\end{remark}

\begin{proof}[Proof of Lemma \ref{Lem_UEstim}]
    By construction $u$ solves the Feynman-Kac PDE \eqref{Eq:FeynmanKac}, implying that
    \begin{align*}
        \partial_t\partial_x^j u(t,x) &= \frac12\partial_x^j \Bigl(\sigma^2(t,x) \partial_x^{2} u(t,x) \Bigr),
        \\
        \partial_t^2 u(t,x) &= \frac14\sigma^2(t,x) \partial_x^{2}(\sigma^2(t,x)\partial_x^{2} u(t,x)) + \sigma(t,x) \partial_t \sigma(t,x) \partial_x^{2} u(t,x).
    \end{align*}
    By Assumption \ref{Aspt_new} on $\sigma$ and the product rule it is sufficient to find bounds for $\partial_x^j u(t,x)$. We treat the case $j=4$, the other cases work similar.

    We can write $u(t,x)=\int_{\mathbb{R}} f(y)p(y;t,x)\,dy=\int_{\mathbb{R}} f(y+x)p(y+x;t,x)\,dy$. Assume first that $f$ is four times continuously differentiable with bounded derivatives, then by Leibniz rule
    \begin{align*}
        \partial_x^4 u(t,x)&=\int_\mathbb{R} \frac{d^4}{dx^4} \bigl( f(y+x) p(y+x;t,x)\bigr)\,dy= \sum_{k=0}^4 \binom{4}{k} \int_\mathbb{R} f^{(k)}(y+x) \frac{d^{4-k}}{dx^{4-k}}p(y+x;t,x)\,dy
    \end{align*}
    If $k\le l$ we can use Lemma \ref{Lem_DensityRegularity} and observe
    \begin{align*}
        \biggl|\int_\mathbb{R} f^{(k)}(y+x) \frac{d^{4-k}}{dx^{4-k}}p(y+x;t,x)\,dy \biggr| &\le \int_\mathbb{R} \bigl|f^{(k)}(x+y)\bigr| C (T-t)^{-1/2} \exp\Bigl(-\frac{y^2}{2C(T-t)} \Bigr)\,dy 
    \end{align*}
    If $k>l$ we use partial integration to see that
    \begin{align*}
        \int_\mathbb{R} f^{(k)}(y+x) \frac{d^{4-k}}{dx^{4-k}}p(y+x;t,x)\,dy= (-1)^{k-l} \int_\mathbb{R} f^{(l)}(y+x)\frac{d^{k-l}}{dy^{k-l}} \frac{d^{4-k}}{dx^{4-k}}p(y+x;t,x)\,dy.
    \end{align*}
    The estimate from Lemma \ref{Lem_DensityRegularity} tells us
    \begin{align*}
        \biggl|\int_\mathbb{R} f^{(l)}(y+x)^{k-l} \frac{d^{4-k}}{dx^{4-k}}p(y+x;t,x)\,dy \biggr|&\le \int_\mathbb{R} \bigl|f^{(l)}(y+x)\bigr| C (T-t)^{-\frac{(k-l+1)}2} \exp\Bigl(-\frac{y^2}{2C(T-t)} \Bigr)\,dy.
    \end{align*}
    Altogether there is a constant $\hat C$ such that
    \begin{align*}
        \partial_x^4 u(t,x)\le \hat C \sum_{k=l}^4 \int_\mathbb{R} \bigl|f^{(k)}(y+x)\bigr| C (T-t)^{-(k-l+1)/2} \exp\Bigl(-\frac{y^2}{2C(T-t)} \Bigr)\,dy.
    \end{align*}
    As the right hand side only depends on the $l$ weak derivatives of $f$, we can also plug in some $f$ satisfying the assumptions of the lemma. Thus it follows from the polynomial growth condition on $f$, the fact that $(k-l+1)\le (4-l+1)$, Gaussian integration (treating the term $(T-t)^{-1/2}$), that there is a constant $0<\tilde C < \infty$ and some $N \in \mathbb{N}$ such that
    \begin{align*}
        \partial_x^4 u(t,x) &\le 
            \tilde C(T-t)^{-(4-l)/2} \bigl(1+|x|^N\bigr)
    \end{align*}
    As we saw that all estimates eventually only depend on the polynomial growth of $f^{(k)}$ for $k\le l$ the claim follows.
\end{proof}

\subsection{Gaussian Calculations}
\begin{lemma}\label{lem:Tail_estim}
    Let $h,T>0$ such that $\frac Th\in \mathbb{N}$.  For $k\in \{1,\dots,\frac Th\}$ let $\tilde \beta_k$ and $\beta_k$ be $\mathcal{F}_{kh}$ adapted and uniformly bounded by some $b>0$.
    Assume that a recursion scheme is given by
    \begin{align*}
        X_{kh+t}=X_{kh}+\beta_k \bigl(B_{kh+t}-B_{kh}\bigr)+\tilde\beta_k \bigl(\bB_{kh+t}-\bB_{kh}\bigr) \qquad t \in [0,h].
    \end{align*}
    Then it holds that
    \begin{align*}
        \Prob\bigl( \sup_{t \in [0,T]} |X_t| >M \bigr) \le \exp\Bigl(-\frac{M^2}{4b^2T}\Bigr).
    \end{align*}
    In particular for any $\lambda>0$ 
    \begin{align*}
        \E\Bigl[\exp(\lambda |X_T|)\Bigr]<\infty.
    \end{align*}
\end{lemma}
\begin{proof}
    By assumption the quadratic variation of $X$ satisfies $\langle X\rangle_t\le 2b^2t$. Then the first part is a standard result from stochastic analysis, see for example \cite[Exercise 3.16, p. 153]{RevuzYor99}.
    
    The second part follows from the fact that $\int_0^\infty e^{c x-d x^2}\,dx<\infty$ for all $c,d>0$.
\end{proof}

\begin{lemma}\label{Lem:Density}
    Let $Z_0,\eta$ be random variables such that $\Prob(\eta=0)=0$. Let $\Delta W$ be Gaussian with mean $0$ and variance $h$ such that $\Delta W$ is independent of $Z_0$ and $\eta$. Assume that 
    \begin{align*}
        Z_1=Z_0+\eta \Delta W.
    \end{align*} 
    Then $Z_{1}$ has a density given by a convolution-type formula: 
    \begin{align*}
        f_{Z_{1}}(x)=\E\Bigl[\varphi_{h\eta^2}\bigl(x-Z_0\bigr) \Bigr]
    \end{align*}
    \end{lemma}

    \begin{proof}
        This follows immediately from the fact that $\Delta W$ is independent of $(Z_0,\eta)$ and has variance $h$.
    \end{proof}

\begin{lemma}\label{Lem_Cond_exp_gauss}
    Let $X,Y,Z$ be random variables and
    define
    $\zeta=X+ Z$.
    Assume that $Z$ has a strictly positive and bounded density $f_Z$ and that $Z$ is independent of $(X,Y)$.
\begin{enumerate}[(i)]
\item 
For any measurable and bounded function $\psi:\mathbb{R}\times\mathbb{R} \rightarrow\mathbb{R}$ such that $\psi(\zeta,Y)$ is integrable,
we have that 
\begin{align*}
    \E\bigl[\psi(\zeta,Y)\bigr]= \int_{\mathbb{R}} \E\bigl[ \psi(x,Y)f_Z(x-X)]\,dx
\end{align*}
\item Let $Y$ be integrable, then almost surely it holds that \begin{align*}
        \E\bigl[ Y \big|\zeta \bigr]=\frac{\E\bigl[
        Y f_{Z}( x-X)
        \bigr]\Big|_{x=\zeta}}
        {\E\bigl[
        f_{Z}( x-X)
        \bigr]\Big|_{x=\zeta}}.
    \end{align*}
    \end{enumerate}    
\end{lemma}

\begin{proof}
Let $\psi:\mathbb{R}\times\mathbb{R} \rightarrow\mathbb{R}$
be a measurable and bounded function. As $Z$ is independent of $\left(X,Y\right)$,
we have by Fubini's theorem
\begin{align*}
    \E\bigl[\psi(\zeta,Y)\bigr]&= \E\bigl[\psi(X+ Z,Y)\bigr]=\E\biggl[
    \int_{\mathbb{R}} \psi(X+z,Y)f_Z(z)\,dz
    \biggr]
    =\E\biggl[
    \int_{\mathbb{R}} \psi(x,Y)f_Z(x-X)\,dx
    \biggr]
\end{align*}
where we made the change of variable $x:=X+z$, and this proves
$(i)$.

Let $F$ be measurable, bounded and with compact support. Then by point $(i)$
\begin{align*}
    \E\biggl[F(\zeta) \frac{\E[Yf_Z(x-X)]\big|_{x=\zeta}}{\E[f_Z(x-X)]\big|_{x=\zeta}}
    \biggr]
     &=
     \int_{\mathbb{R}} 
    \E\biggl[F(x) \frac{\E[Yf_Z(x-X)]}{\E[f_Z(x-X)]}
    f_Z(x-X) \biggr]\,dx
    \\
    &=
    \int_{\mathbb{R}} 
    \E[F(x)Yf_Z(x-X)]
    \,dx
    \\
    &=\E\bigl[F(\zeta)Y\bigr].
\end{align*}
For the last equality we applied $(i)$ a second time. By definition of the conditional expectation, the claim follows. 
\end{proof}

\subsection{Auxiliary Lemmas}

\begin{lemma}\label{Lem:Gauss_facts}
    Let $\lambda>0$ and let $\varphi_\lambda(x)=\frac{1}{\sqrt{2\pi\lambda}}e^{-\frac{x^2}{2\lambda}}$ be the density of a centered Gaussian with variance $\lambda$. 
        \begin{enumerate}
        \item \label{Lem:Gauss_fact1}$\forall x\in\mathbb{R},\;\forall \gamma\in(0,1],\; |x|\varphi_\lambda(x)\le (e\gamma)^{-\frac12} \lambda^{\frac{1-\gamma}{2}} \varphi_\lambda^{1-\gamma}(x)$,
            \item $\forall x,y \in \mathbb{R}\mbox{ with }|x-y|<\eps,\;\forall \gamma\in(0,1],\;\varphi_\lambda(y) \ge (2\pi\lambda)^{\frac\gamma2}\varphi_\lambda(x)^{1+\gamma} e^{-\frac{\eps^2}{\gamma \lambda}} $,
        \item $\forall x\in\mathbb{R},\;\forall \gamma>0,\; \varphi_\lambda^{1+\gamma}(x)\le {(2\pi\lambda)}^{-\frac\gamma2} \varphi_\lambda(x).$
        \end{enumerate}
\end{lemma}

\begin{proof}
    \textbf{(1):} 
    We use the fact that
    $\mathbb{R}\ni y\mapsto |y|e^{-\frac{y^2}{2}}$ is bounded by $e^{-\frac12}$ and see that for $x>0$
    \begin{align*}
        x\varphi_\lambda(x)= x e^{-\frac{x^2\gamma}{2\lambda}} \Bigl(\frac{1}{\sqrt{2\pi \lambda}}\Bigr)^\gamma \varphi_\lambda^{1-\gamma} 
        = 
        x\frac{\sqrt{\gamma}}{\sqrt{\lambda}} e^{-\frac{x^2\gamma}{2\lambda}} {\lambda}^{\frac{1-\gamma}2}\gamma^{-\frac12}\frac{1}{\sqrt{2\pi}^{\gamma}}\varphi_\lambda^{1-\gamma}\le (e\gamma)^{-\frac12} {\lambda}^{\frac{1-\gamma}2} \varphi_\lambda^{1-\gamma}.  
    \end{align*}
    By symmetry the first claim follows.
    
    \textbf{(2):} Simple calculation shows
    \begin{align*}
        \varphi_\lambda(y)=\frac{1}{\sqrt{2\pi \lambda}}e^{-\frac{(x+y-x)^2}{2\lambda}}=\frac{1}{\sqrt{2\pi \lambda}}e^{-\frac{x^2(1+\gamma)}{2\lambda}}e^{\frac{x^2\gamma-2x(y-x)-(y-x)^2}{2\lambda}}.
    \end{align*}
    The claim then follows from the fact that 
    \begin{align*}
        x^2\gamma -2x(y-x)-(y-x)^2\ge x^2\gamma -2 |x|\eps-\eps^2 \ge -\frac{\eps^2}{\gamma}-\eps^2\ge -2\frac{\eps^2}{\gamma}.
    \end{align*}
    
    \textbf{(3):} Can be seen directly:
    \begin{align*}
        \varphi_\lambda^{1+\gamma}(x)={\bigl(2\pi\lambda\bigr)}^{-\frac\gamma2} \frac{1}{\sqrt{2\pi\lambda}} e^{-\frac{x^2(1+\gamma)}{2\lambda}}\le \bigl(2\pi\lambda\bigr)^{-\frac\gamma2} \frac{1}{\sqrt{2\pi\lambda}} e^{-\frac{x^2}{2\lambda}}.
    \end{align*}
    
\end{proof}

\begin{lemma}\label{lem:recurr_relation}
    Let $h,\gamma,\lambda>0$ and  $T>0$ such that $n\coloneqq T/h\in\mathbb{N}$. Assume that there is a recurrence relation of the form 
    \begin{align*}
        f_{j+1}\le \bigl(1+ h\gamma \bigr)f_j+\lambda \qquad j=0,\dots,Tn-1
    \end{align*}
    such that $f_0=0$. Then it follows that
    \begin{align*}
        f_j\le \lambda \sum_{k=1}^j (1+h \gamma )^{k-1} .
    \end{align*}
    In particular $f_j\le \frac\lambda{h\gamma}  e^{T\gamma}$.
\end{lemma}
\begin{proof}
    It can follows from induction that
    \begin{align*}
    g_j=\lambda \sum_{k=1}^j (1+h\gamma)^{k-1}.
    \end{align*}
    is a solution to $g_{j+1}=(1+h\gamma)g_j+\lambda$ with $g_0=0$. We can check by induction on $j$ that $f_j\le g_j$.
    Properties of the geometric series
    imply that
    \begin{align*}
        g_j=\lambda \frac{1}{h\gamma} \Bigl( \bigl(1+h\gamma\bigr)^j-1\Bigr)\le \frac{\lambda}{h\gamma} \Bigl( \bigl(1+h\gamma \bigr)^{T/h}-1\Bigr).
    \end{align*}
    By monotonicity and the definition of the Euler number $e$ it follows that
    \begin{align*}
        g_j \le \frac\lambda{h\gamma}e^{T\gamma}. 
    \end{align*}
\end{proof}

\begin{lemma}\label{lem_sqt_estim}
    Let $a_0,a_1,b_0,b_1>0$ be positive numbers and let $\delta \in (0,1)$. Assume that there is some $0<c\le 1\le C<\infty$ such that
    \begin{align*}
        c\le \frac{a_0}{b_0}\le C \text{ and } c\le \frac{a_1}{b_1}\le C.
    \end{align*}
    Then it follows that
    \begin{align*}
        \Bigl|\frac{\sqrt{a_0+\delta}}{\sqrt{b_0+\delta}}-\frac{\sqrt{a_1+\delta}}{\sqrt{b_1+\delta}}\Bigr|\lesssim \frac1\delta \Bigl(|a_0-a_1|+|b_0-b_1|\Bigr).
    \end{align*}
    Furthermore, it holds that
    \begin{align*}
        \Bigl|\frac{\sqrt{a_0+\delta}}{\sqrt{b_0+\delta}}-\frac{\sqrt{a_1+\delta}}{\sqrt{b_1+\delta}}\Bigr|^2 
        \le
        \frac{\sqrt{C}}{\sqrt{c}}\Bigl|\frac{{a_0+\delta}}{{b_0+\delta}}-\frac{{a_1+\delta}}{{b_1+\delta}}\Bigr|.
    \end{align*}
\end{lemma}
\begin{proof}
    We begin by rewriting
    \begin{align*}
        \frac{\sqrt{a_0+\delta}}{\sqrt{b_0+\delta}}-\frac{\sqrt{a_1+\delta}}{\sqrt{b_1+\delta}}
        =
        \frac{\sqrt{a_0+\delta}}{\sqrt{b_0+\delta}}-\frac{\sqrt{a_1+\delta}}{\sqrt{b_0+\delta}}+\frac{\sqrt{a_1+\delta}}{\sqrt{b_0+\delta}}-\frac{\sqrt{a_1+\delta}}{\sqrt{b_1+\delta}}.
    \end{align*}
    By the fact that $a_0,a_1,b_0>0$ and Taylor formula 
    \begin{align*}
        \Bigl|\frac{\sqrt{a_0+\delta}}{\sqrt{b_0+\delta}}-\frac{\sqrt{a_1+\delta}}{\sqrt{b_0+\delta}}\Bigr|\le  \frac1{\sqrt{\delta}}\frac{1}{2\sqrt{\delta}} \bigl|a_0-a_1\bigr|.
    \end{align*}
    By the fundamental theorem of calculus for $0<x\le y$ we have that $x^{-1/2}-y^{-1/2}=\frac12 \int_x^y z^{-3/2}\,dz\le |y-x|\frac12x^{-3/2}$. Therefore
    \begin{align*}
        \Bigl|
        \frac{\sqrt{a_1+\delta}}{\sqrt{b_0+\delta}}-\frac{\sqrt{a_1+\delta}}{\sqrt{b_1+\delta}}
        \Bigr|
        \le 
        \frac1{2} \cdot\frac{\sqrt{a_1+\delta}\cdot |b_0-b_1|}{(b_0+\delta)^{3/2}\wedge (b_1+\delta)^{3/2}}
        \le 
        \frac1{2\delta} \cdot\frac{\sqrt{a_1+\delta}\cdot |b_0-b_1|}{\sqrt{b_0+\delta}\wedge \sqrt{b_1+\delta}}.
    \end{align*}
    Summing up both estimates it follows that
    \begin{align*}
         \Bigl|\frac{\sqrt{a_0+\delta}}{\sqrt{b_0+\delta}}-\frac{\sqrt{a_1+\delta}}{\sqrt{b_1+\delta}}\Bigr|\le \frac1{2\delta}\Bigl( |a_0-a_1|+\frac{\sqrt{a_1+\delta}\cdot |b_0-b_1|}{\sqrt{b_0+\delta}\wedge \sqrt{b_1+\delta}}\Bigr).
    \end{align*}
    Similarly, by rewriting
    \begin{align*}
        \frac{\sqrt{a_0+\delta}}{\sqrt{b_0+\delta}}-\frac{\sqrt{a_1+\delta}}{\sqrt{b_1+\delta}}
        =
        \frac{\sqrt{a_0+\delta}}{\sqrt{b_0+\delta}}-\frac{\sqrt{a_0+\delta}}{\sqrt{b_1+\delta}}+\frac{\sqrt{a_0+\delta}}{\sqrt{b_1+\delta}}-\frac{\sqrt{a_1+\delta}}{\sqrt{b_1+\delta}},
    \end{align*}
    it follows from the same arguments that
     \begin{align*}
         \Bigl|\frac{\sqrt{a_0+\delta}}{\sqrt{b_0+\delta}}-\frac{\sqrt{a_1+\delta}}{\sqrt{b_1+\delta}}\Bigr|\le \frac1{2\delta}\Bigl( |a_0-a_1|+\frac{\sqrt{a_0+\delta}\cdot |b_0-b_1|}{\sqrt{b_0+\delta}\wedge \sqrt{b_1+\delta}}\Bigr).
    \end{align*}
    By combining these estimates we see that
     \begin{align*}
         \Bigl|\frac{\sqrt{a_0+\delta}}{\sqrt{b_0+\delta}}-\frac{\sqrt{a_1+\delta}}{\sqrt{b_1+\delta}}\Bigr|&\le \frac1{2\delta}\Bigl( |a_0-a_1|+|b_0-b_1|\frac{\sqrt{a_0+\delta}\wedge \sqrt{a_0+\delta}}{\sqrt{b_0+\delta}\wedge \sqrt{b_1+\delta}}\Bigr)
         \\
         &\le \frac1{2\delta}\Bigl( |a_0-a_1|+|b_0-b_1|C\Bigr).
    \end{align*}
    For the second claim observe that by Taylor formula applied to $x\mapsto \sqrt{x}$ again
    \begin{align*}
        \Bigl|\frac{\sqrt{a_0+\delta}}{\sqrt{b_0+\delta}}-\frac{\sqrt{a_1+\delta}}{\sqrt{b_1+\delta}}\Bigr|&\le \frac12 \min\Bigl( 
        \frac{\sqrt{b_0+\delta}}{\sqrt{a_0+\delta}},\frac{\sqrt{b_1+\delta}}{\sqrt{a_1+\delta}}
        \Bigr)  \Bigl|\frac{{a_0+\delta}}{{b_0+\delta}}-\frac{{a_1+\delta}}{{b_1+\delta}}\Bigr|
        \\
        &\le \frac1{2\sqrt{c}} \Bigl|\frac{{a_0+\delta}}{{b_0+\delta}}-\frac{{a_1+\delta}}{{b_1+\delta}}\Bigr|.
    \end{align*}
\end{proof}

\begin{lemma}\label{lem:estim_AB}
    Let $(\alpha_k)_{k=1}^n,(\beta_k)_{k=1}^n,(\gamma_k)_{k=1}^n$ and $(\Delta_k)_{k=1}^n$ be positive sequences and assume that $\alpha_k \le A$ as well as $a\le \gamma_k \le b$ for some $a\le 1\le A,b$. 
    
    Let $\delta > 0$ and assume that $|\alpha_k-\beta_k|\le C \max\bigl\{\alpha_k^p,\beta_k^p \bigr\} \Delta_k$ for some $C>0$ and $1/2\le p\le 1$. Then it holds that
   {
    \begin{align*}
    \Bigl|\frac{\frac1n\sum_{k=1}^n \alpha_k +\delta}{\frac1n\sum_{k=1}^n \alpha_k\gamma_k +\delta}-\frac{\frac1n\sum_{k=1}^n \beta_k +\delta}{\frac1n\sum_{k=1}^n \beta_k\gamma_k +\delta}
        \Bigr|
        \lesssim C \max_{j=1,\dots,n}\max\{\alpha_j,\beta_j\}^{p-\frac1{2}} \delta^{-1/2} \sqrt{\frac1n \sum_{k=1}^n \Delta_k^2}
        \end{align*}
    }
\end{lemma}
\begin{proof}
    By symmetry we assume without loss of generality that $\sum_{k=1}^n \beta_k^{2p} \le \sum_{k=1}^n \alpha_k^{2p}$ implying in particular $\sum_{k=1}^n \max\bigl\{\alpha_k^{2p},\beta_k^{2p} \bigr\}\le 2 \sum_{k=1}^n \alpha_k^{2p}$.
    We then rewrite
    \begin{align*}
        \frac{\frac1n\sum_{k=1}^n \alpha_k +\delta}{\frac1n\sum_{k=1}^n \alpha_k\gamma_k +\delta}-\frac{\frac1n\sum_{k=1}^n \beta_k +\delta}{\frac1n\sum_{k=1}^n \beta_k\gamma_k +\delta}
        &=
        \frac{\frac1n\sum_{k=1}^n \bigl(\alpha_k-\beta_k\bigr)}{\frac1n\sum_{k=1}^n \alpha_k\gamma_k +\delta}
        \\
        &\qquad+
        \frac{\bigl(\frac1n\sum_{k=1}^n \beta_k+\delta\bigr)\frac1n\sum_{k=1}^n \gamma_k\bigl(\beta_k-\alpha_k\bigr)}{\bigl(\frac1n\sum_{k=1}^n \alpha_k\gamma_k+\delta\bigr)\bigl(\frac1n\sum_{k=1}^n \beta_k\gamma_k+\delta\bigr)} 
    \end{align*}
    By our assumption
    \begin{align*}
        \Bigl| \frac{\bigl(\frac1n\sum_{k=1}^n \beta_k+\delta\bigr)\frac1n\sum_{k=1}^n \gamma_k\bigl(\beta_k-\alpha_k\bigr)}{\bigl(\frac1n\sum_{k=1}^n \beta_k\gamma_k+\delta\bigr)\bigl(\frac1n\sum_{k=1}^n \alpha_k\gamma_k+\delta\bigr)} \Bigr|
        \le 
        \frac1a \cdot \Bigl| \frac{\frac1n\sum_{k=1}^n \gamma_k\bigl(\beta_k-\alpha_k\bigr)}{\frac1n\sum_{k=1}^n \alpha_k\gamma_k+\delta}\Bigr|\le\frac ba \cdot \frac{\frac1n\sum_{k=1}^n \bigl|\beta_k-\alpha_k\bigr|}{\frac1n\sum_{k=1}^n \alpha_k\gamma_k+\delta}.
    \end{align*}
    By the assumptions and the Cauchy-Schwarz inequality, we see that
    \begin{align*}
        \Bigl|
        \frac{\frac1n\sum_{k=1}^n \bigl|\alpha_k-\beta_k\bigr|}{\frac1n\sum_{k=1}^n \alpha_k\gamma_k +\delta}
        \Bigr|
        \le
        C\frac{\frac1n\sum_{k=1}^n \max\bigl\{\alpha_k^p,\beta_k^p \bigr\} \Delta_k}{\frac1n\sum_{k=1}^n \alpha_k\gamma_k +\delta}
        \le
        C\frac{\sqrt{\frac1n\sum_{k=1}^n \max\bigl\{\alpha_k^{2p},\beta_k^{2p} \bigr\}} }{\frac1n\sum_{k=1}^n \alpha_k\gamma_k +\delta}\sqrt{
        \frac1n \sum_{k=1}^n \Delta_k^2
        }
    \end{align*}
    { 
    Our assumption at the beginning of the proof implies that
    \begin{align*}
        \frac{\sqrt{\frac1n\sum_{k=1}^n \max\bigl\{\alpha_k^{2p},\beta_k^{2p} \bigr\}} }{\frac1n\sum_{k=1}^n \alpha_k\gamma_k +\delta} \le
        \frac2a \cdot 
        \frac{\sqrt{\frac1n\sum_{k=1}^n \alpha_k^{2p}}}{\frac1n\sum_{k=1}^n \alpha_k +\delta}
        \le 
        \frac{2\max_{j=1}^n\alpha_k^{p-\frac1{2}}}a \cdot
        \frac{\Bigl(\frac1n\sum_{k=1}^n \alpha_k\Bigr)^{\frac{1}{2}}
        }{\Bigl(\frac1n\sum_{k=1}^n \alpha_k\Bigr) +\delta}.
    \end{align*}
    }
    {
    The last factor can be described by the function $y\mapsto \frac{\sqrt{y}}{y+\delta}$ on $\mathbb{R}^+$, which is bounded by $\delta^{-1/2}/2$. By combining everything, we see that
    \begin{align*}
         \Bigl|\frac{\frac1n\sum_{k=1}^n \alpha_k +\delta}{\frac1n\sum_{k=1}^n \alpha_k\gamma_k +\delta}-\frac{\frac1n\sum_{k=1}^n \beta_k +\delta}{\frac1n\sum_{k=1}^n \beta_k\gamma_k +\delta}
        \Bigr|
        \le \bigl(1+\frac ba\bigr)\frac1a 
        \max_{j=1,\dots,n}\alpha_j^{p-\frac1{2}} \delta^{-1/2}C \sqrt{
        \frac1n \sum_{k=1}^n \Delta_k^2
        }.
    \end{align*}
    }
\end{proof}

\end{document}